\documentclass{article}

\makeatletter
\def\set@curr@file#1{%
  \begingroup
    \escapechar\m@ne
    \xdef\@curr@file{\expandafter\string\csname #1\endcsname}%
  \endgroup
}
\def\quote@name#1{"\quote@@name#1\@gobble""}
\def\quote@@name#1"{#1\quote@@name}
\def\unquote@name#1{\quote@@name#1\@gobble"}
\makeatother

\usepackage{graphicx}
\usepackage{hyperref}

\usepackage{amsmath}
\usepackage{amsthm}
\usepackage{amssymb}
\usepackage{anyfontsize}
\usepackage{bm}
\usepackage{enumerate}
\usepackage{esint}
\usepackage{etoolbox}
\usepackage{isomath}
\usepackage{mathrsfs}
\usepackage{mathtools}
\usepackage{multicol}
\usepackage{pgfplots}
\usepackage{scalerel}
\usepackage{stackengine}
\usepackage{tabulary}
\usepackage{tikz}

\newcommand{\ds}[1]{\displaystyle{#1}}

\newcommand{\ns}{\!\!}
\newcommand{\nss}{\ns\ns}
\newcommand{\nsss}{\nss\nss}

\newcommand{\qqquad}{\qquad\qquad}
\newcommand{\qqqquad}{\qqquad\qqquad}

\newcommand{\set}[1]{\mathcal{#1}}
\newcommand{\class}[1]{\mathscr{#1}}

\newcommand{\Haus}{\mathscr{H}}

\newcommand{\bs}{\setminus}

\newcommand{\haus}{\mathscr{H}}

\newcommand{\nats}{\mathbb{N}}
\newcommand{\whls}{\mathbb{W}}

\newcommand{\re}{\mathbb{R}}

\newcommand{\rem}{\re^m}
\newcommand{\ren}{\re^n}

\newcommand{\reo}{\ov{\re}}

\newcommand{\bll}{\set{B}}
\newcommand{\sph}{\set{S}}
\newcommand{\ann}{\set{A}}

\newcommand{\cnt}{\class{C}}

\newcommand{\smfrac}[2]{{\textstyle{\frac{#1}{#2}}}}
\newcommand{\ee}{\mathrm{e}}

\newcommand{\lint}[1]{\int\limits_{#1}}
\newcommand{\luint}[2]{\int\limits_{#1}^{#2}}
\newcommand{\flint}[1]{\fint\limits_{#1}}
\newcommand{\dd}{\mathrm{d}}

\newcommand{\der}[2]{\frac{\dd #1}{\dd #2}}

\newcommand{\grd}{\te{\nabla}}

\makeatletter
\newcommand*\bdot{\mathpalette\bdot@{.65}}
\newcommand*\bdot@[2]{\mathbin{\vcenter{\hbox{\scalebox{#2}{$\m@th#1\bullet$}}}}}
\makeatother

\makeatletter
\newcommand*\bddot{\mathpalette\bddot@{.65}}
\newcommand*\bddot@[2]{\mathbin{\vcenter{\hbox{\scalebox{#2}
    {$\m@th#1\smash{{}_{\bullet}^{\bullet}}$}}}}}
\makeatother

\newcommand{\circled}[2][]{%
  \tikz[baseline=(char.base)]{%
    \node[shape = circle, draw, inner sep = .5pt]
    (char) {\phantom{\ifblank{#1}{#2}{#1}}};%
    \node at (char.center) {\makebox[0pt][c]{#2}};}}
\robustify{\circled}

\newcommand{\ve}[1]{\vectorsym{#1}}

\newcommand{\vv}{\ve{v}}
\newcommand{\vw}{\ve{w}}
\newcommand{\vx}{\ve{x}}
\newcommand{\vy}{\ve{y}}
\newcommand{\vz}{\ve{z}}

\newcommand{\lp}{\left(}
\newcommand{\rp}{\right)}
\newcommand{\ls}{\left[}
\newcommand{\rs}{\right]}
\newcommand{\lc}{\left\{}
\newcommand{\rc}{\right\}}

\newcommand{\te}[1]{\tensorsym{#1}}

\def\loc{\operatorname{loc}}
\def\dist{\operatorname{dist}}

\def\supp{\operatorname{supp}}
\def\esssup{\operatorname*{ess\;sup}}

\def\rank{\operatorname{rank}}

\def\diam{\operatorname{diam}}

\def\card{\operatorname{card}}

\newcommand{\nd}{\quad\text{ and }\quad}
\newcommand{\frl}{\quad\text{ for all }}

\newcommand{\ov}[1]{\overline{#1}}
\newcommand{\vep}{\varepsilon}

\newcommand{\wh}[1]{\widehat{#1}}

\newcommand{\suprd}[1]{#1^\text{rd}}
\newcommand{\supth}[1]{#1^\text{th}}

\stackMath
\newcommand\reallywidecheck[1]{%
\savestack{\tmpbox}{\stretchto{%
  \scaleto{%
    \scalerel*[\widthof{\ensuremath{#1}}]{\kern-.6pt\bigwedge\kern-.6pt}%
    {\rule[-\textheight/2]{1ex}{\textheight}}%WIDTH-LIMITED BIG WEDGE
  }{\textheight}%
}{0.5ex}}%
\stackon[1pt]{#1}{\scalebox{-1}{\tmpbox}}%
}

\newcommand{\dom}{\Omega}
\newcommand{\bnd}{\Gamma}

\newcommand{\bbD}{\mathbb{D}}

\newcommand{\valp}{\ve{\alpha}}

\newcommand{\vzet}{\ve{\zeta}}

\newcommand{\vg}{\ve{g}}
\newcommand{\vP}{\ve{P}}

\newcommand{\tI}{\te{I}}

\newcommand{\vom}{\ve{\omega}}

\newcommand{\orb}{\set{O}}

\theoremstyle{plain}
\newtheorem{theorem}{Theorem}[section]
\newtheorem{lemma}[theorem]{Lemma}
\newtheorem{corollary}[theorem]{Corollary}
\theoremstyle{definition}
\newtheorem{definition}{Definition}[section]
\newtheorem{example}{Example}[section]
\theoremstyle{remark}
\newtheorem*{remark}{Remark}
\begin{document}

\title{Nonlocal Poincar\'{e} Inequalities for Integral Operators with Integrable Nonhomogeneous Kernels}

\author{Mikil Foss\footnote{203 Avery Hall,
    Lincoln, NE 68588-0130 USA, University of Nebraska-Lincoln\newline
    email: mfoss3@unl.edu\newline\newline
    The author's research is supported by the NSF award DMS-1716790}}

\maketitle

\begin{abstract}
The paper provides two versions of nonlocal Poincar\'e-type inequalities for integral operators with a convolution-type structure and functions satisfying a zero-Dirichlet like condition. The inequalities extend existing results to a large class of nonhomogeneous kernels with supports that can vary discontinuously and need not contain a common set throughout the domain. The measure of the supports may even vanish allowing the zero-Dirichlet condition to be imposed on only a lower-dimensional manifold, with or without boundary. The conditions may be imposed on sets with co-dimension larger than one or even at just a single point. This appears to currently be the first such results in a nonlocal setting with integrable kernels. The arguments used are direct, and examples are provided demonstrating the explicit dependence of bounds for the Poincar\'{e} constant upon structural parameters of the kernel and domain.\\[5pt]
\end{abstract}

\section{Introduction}\label{S:Intro}
\subsection{Overview}\label{SS:Overview}
Nonlocal frameworks have empowered mathematical modelers with new tools to capture phenomena with non-differentiable or even discontinuous features. Classical models employ operators that inherently assume the solution possesses some level of differentiability. Modeling with nonlocal operators can expand the space of possible solutions to functions that lack such smoothness. Areas in which nonlocal models have had success include image processing~\cite{GilOsh:08a,KatFoiEgi:10a}, diffusion models~\cite{AndMazRos:14a}, population models~\cite{CovDup:07a}, flocking models~\cite{PoySol:17a}, phase transitions~\cite{Gal:18a,Han:04a}, and material deformation with failure~\cite{Sil:00a,Sil:17a,SilEptWec:07a}. A comprehensive discussion of the applications and analysis for nonlocal models and operators, with many additional references, is available in the monograph~\cite{Du:19a}.

In this paper, we establish Poincar\'{e} like inequalities for a broad family of nonlocal operators. Roughly speaking, Poincar\'{e} inequalities deliver a bound on either the norm, or the variance, of a function in terms of some measure for the total change of the function throughout its domain. The well-known classical Poincar\'{e} inequality states that for each $1\le p<\infty$, there exists a constant $C<\infty$ such that
\[
  \|u\|_{L^p(\dom)}\le C\|\grd u\|_{L^p(\dom)},
\]
for all $u\in W^{1,p}(\dom)$, with zero trace on the boundary $\partial\dom$. Here the domain $\dom\subset\ren$ is an open bounded set. There are many variants of this inequality, but our results can be best described as a nonlocal version of this one. In particular, we require $u$ to satisfy, in some sense, a zero-trace condition. In a classical (local) context, such inequalities are a key ingredient in establishing existence and stability for variational and PDE problems. The Poincar\'{e} constant is also related to the conditioning number and stability in certain methods for numerical solutions.

Our focus is on Poincar\'{e} inequalities where the measure for the total change in the function has a weighted difference operator $\Delta_{\vz}u(\vx):=[u(\vx+\vz)-u(\vx)]w(\vx,\vz)$ as their basis. Using $L$ to denote the space of $\re$-valued measurable functions, we assume $w\in L(\ren\times\ren)$ and $u\in L(\ren)$. (We may always extend $u$ by zero as needed.) Under various additional structural assumptions this operator has been identified as a nonlocal or two-point gradient~\cite{Du:19a,DuGunLeh:12a,GilOsh:08a}. We consider two related quantifications of the total change in $u$ over the domain $\dom$, both of which involve ``accumulating'' $\Delta_{\vz}u(\vx)$ over neighborhoods of $\vx\in\dom$. We establish conditions under which this nonlocal measure for total change yields a bound for $u$ in $L^p$, with $1\le p<\infty$.

In the remainder of this section, we give a general description of the results proved in this paper. The significance of the results, in the context of related work that can be found in the literature, is discussed in Section~\ref{SS:Significance}, and the organization of the remainder of the paper is in Section~\ref{SS:Organization}. After establishing some notation, in Section~\ref{S:IntroExamples}, we continue the introduction with a demonstration of the arguments used in the context of a few examples.

\subsubsection{Nonlocal Poincar\'{e} Inequality I}\label{SSS:IntroPoinc1}
The first version of the Poincar\'{e} inequality is stated in Theorem~\ref{T:MainPoinc1} and has the general form
\begin{equation}\label{E:IntroMainPoinc1}
  \|u\|^p_{L^p(\dom)}
  \le
  C\|\mathbb{G}u(\vx)\|_{L^p(\dom)}^p,
\end{equation}
with
\[
  \mathbb{G}u(\vx):=\lint{Z(\vx)}[u(\vx+\vz)-u(\vx)]
    \rho(\vx,\vz)\gamma(\vx)^\frac{1}{p}\dd\vz.
\]
Here $Z:\dom\twoheadrightarrow\ren$, $\rho\in L(\ren\times\ren)$, and $\gamma\in L^+(\ren)$. (We use $\twoheadrightarrow$ to indicate a set-valued function and use $L^+$ to denote the space of nonnegative measurable functions.) Define $\Psi:\dom\twoheadrightarrow\ren$ by $\Psi(\vx):=\vx+Z(\vx)$. The main requirements on the kernel components $\rho$ and $\gamma$ are
\begin{enumerate}[(i)]
\item\label{A:IntroRho} $\rho(\vx,\cdot)\in L^\frac{p}{p-1}(\ren)$ and
\[
  \lint{Z(\vx)}\rho(\vx,\vz)\dd\vz=1,\frl\vx\in\dom;
\]
\item\label{A:IntroGamma} $\gamma\ge 1$ and there exists $0<\nu<1$ such that
\[
  \lint{\Psi^{-1}(\vy)}R(\vx)\gamma(\vx)\dd\vx\le\nu\gamma(\vy),\frl\vy\in\dom.
\]
\end{enumerate}
Here $R(\vx):=\|\vz\mapsto\rho(\vx,\vz)\|^{p-1}_{L^\frac{p}{p-1}(Z(\vx))}$. There are some additional measurability assumptions, and we refer to the statement of Theorem~\ref{T:MainPoinc1} for these. If the set $Z$ is independent of $\vx$ and symmetric about $\ve{0}$ and $\rho(\vx,\cdot)=|Z|^{-1}$, for all $\vx\in\dom$, then $Z^{-1}=Z$ and~(\ref{A:IntroGamma}) becomes a reverse Jensen's inequality and is satisfied by any $\gamma$ that is uniformly concave. We also observe that $\mathbb{G}u\in L(\dom)$ whenever $u\in L^1(\dom\cup\bnd)$. As with the local Poincar\'{e} inequality, we require $u$ to satisfy a zero-Dirichlet like condition. Let $\bnd\subseteq\ren$ be a measurable set such that $\bigcup_{\vx\in\dom}\Psi(\vx)\subseteq\dom\cup\bnd$. Our nonlocal zero-trace conditions are
\[\tag{BC}
  u=0\text{ a.e. on }\bnd
  \nd
  \lint{\dom\bs\bnd}\lp\flint{\Psi(\vx)}|u(\vy)|\dd\vy\rp^p\gamma(\vx)\dd\vx<\infty.
\]
The term in the parentheses is the mean-value of $|u|$ over $\Psi(\vx)$. If $|\bnd|>0$, then the first condition in~(BC) becomes a volume-constraint. If in addition, we find $\Psi$ has uniformly positive measure in $\dom$ and $\gamma\in L^\infty$, then the second condition in~(BC) is always satisfied by $u\in L^1$. The uniform positivity of $|\Psi|$, however, is not a requirement of Theorem~\ref{T:MainPoinc1}. In fact, with $d_\bnd$ denoting the distance to $\bnd$, we may have $\Psi(\vx)\subseteq\bll_{d_\bnd(\vx)}(\vx)$, for all $\vx\in\dom$, so $\bnd$ has zero measure. In Lemma~\ref{L:DistClass}, we provide a large class of $\gamma\in L^\infty_{\loc}(\ren\bs\bnd)$ that satisfy~(\ref{A:IntroGamma}). In this class, we find $\gamma(\vx)\sim d_\bnd(\vx)^{-\beta}$, with $\beta>0$. Since we only assume $u\in L^1$, the first assumption in~(BC) is empty. The second assumption, however, requires $|u(\vx)|$ to converge to zero in the sense of averages as $\bnd$ is approached.

\subsubsection{Nonlocal Poincar\'{e} Inequality II}\label{SSS:IntroPoinc2}
A second version of the Poincar\'{e} inequality is provided in Theorem~\ref{T:MainPoinc2}. It has the form
\begin{equation}\label{E:IntroMainPoinc2}
  \|u\|^p_{L^p(\dom)}
  \le
  C\lint{\dom}\lint{Z(\vx)}
    \left|u(\vx+\vz)-u(\vx)\right|^p\mu(\vx,\vz)\dd\vz\dd\vx,
\end{equation}
with $\mu\in L^+(\ren\times\ren)$. The other assumptions on $\mu$ are somewhat technical but allow for integrable or continuous kernels. Furthermore, the supports need not contain a common set, so there may be $\vx_1,\vx_2\in\dom$ such that $Z(\vx_1)\cap Z(\vx_2)=\emptyset$. The definition of $\Psi$ and the assumptions for $\bnd$ are the same as before. We see that we can obtain~\eqref{E:IntroMainPoinc2} from~\eqref{E:IntroMainPoinc1} when $\rho(\vz,\vz)=|Z(\vx)|^{-1}$ and there is a constant $K<\infty$ such that $\mu(\vx,\vz)\ge K\gamma(\vx)|Z(\vx)|^{-1}$, for $\vx\in\dom$ and $\vz\in Z(\vx)$. A challenge to using Theorem~\ref{T:MainPoinc1}, however, is identifying $\Psi^{-1}$ and a weight $\gamma$ that satisfies assumption~(\ref{A:IntroGamma}). For Theorem~\ref{T:MainPoinc2}, we instead assume that the sets $\Psi(\vx)$ provide a family of paths that connect the points in $\dom$ to points in $\bnd$, where it is assumed that $u=0$ (see Figure~1 in Section~\ref{S:IntroExamples}). The paths are represented by the forward orbits of a family of discrete dynamical systems with a common absorption set contained within $\bnd$. The constant in the Poincar\'{e} inequality is related to the volume changes of $\supp\mu$ compounded by the number of ``steps'' in the path to the absorption set and blows-up if $|\Psi(\vx)|\to0$, which must happen if $|\bnd|=0$. Supplementing the dynamical system argument with a version of Poincar\'{e} Inequality I allows us to obtain inequalities of the form~\eqref{E:IntroMainPoinc2} even when $|\bnd|=0$. As an application, in Example~\ref{Ex:ManExample2} we establish the following: if $\dom$ is a convex set, $m\in\{0,1,\dots,n-1\}$, and $\bnd\subseteq\ov{\dom}$ is a compact $m$-dimensional $\cnt^2$-manifold without boundary, then for each $\beta>n-m$ and $\delta_0>0$, there exists $C<\infty$ such that
\begin{equation}\label{E:IntroLowDim}
  \|u\|_{L^p(\dom)}^p
  \le
  C\lint{\dom\bs\bnd}\frac{1}{\delta(\vx)^\beta}
  \flint{\dom\cap\bll_{\delta(\vx)}(\vx)}
    |u(\vy)-u(\vx)|^p\dd\vy\dd\vx,
\end{equation}
for every $u\in L^1(\dom)$ satisfying~(BC). Here $\delta(\vx)=\min\{\delta_0,d_\bnd(\vx)\}$, for each $\vx\in\dom$. The set $\bnd$, on which~(BC) is imposed, may be some or all of $\partial\dom$, or be contained in the interior of $\dom$. Moreover, the co-dimension can be larger than one. For example $\bnd$ can be a discrete set of points or a filament within $\ov{\dom}$. One can extend Example~\ref{Ex:ManExample2} to accommodate manifolds with an $m-1$-dimensional boundary, provided $\beta>n-m+1$ in an open neighborhood of the boundary of $\bnd$. Relaxing the smoothness requirement on $\bnd$ is briefly discussed in Remark~\ref{R:SmoothMan}.

\subsection{Significance of Results}\label{SS:Significance}

In this section, we provide a broader context for our results. Poincar\'{e} inequalities are an important component in the study of many mathematical problems. For example, the direct method for existence and standard arguments for stability rely on coercivity of the associated functional in an appropriate space, which is typically provided by a Poincar\'{e} inequality~\cite{BelMor:14a,DuTia:18a,FosRadWri:18a,HinRad:12a}. The dependence of an upper bound for the Poincar\'{e} constant on structural parameters in the problem is also often essential. Bounds on the Poincar\'{e} constant provide information about the conditioning number and stability of the numerical methods~\cite{AksMen:10a,AksPar:11a,AksUnl:14a}. In many models, the nonlocal operator appears with a parameter that captures the concentration of the kernel allowing one to consider a limit as the nonlocality vanishes. Identifying the local operator that corresponds to this limit relies on the bound for the Poincar\'{e} constant scaling properly with respect to the parameter~\cite{FosRad:18a,DuMen:15a}. We will outline some of the arguments for Poincar\'{e} inequalities that exist in the current literature and then describe the distinguishing features of this paper's results. For a much more comprehensive survey of the literature, we refer to~\cite{Du:19a}.

To facilitate the rest of the discussion, we introduce the parameter $\delta>0$ that captures the nonlocality of the operators. For each $\delta>0$, suppose that $J_\delta\in L^+([0,\infty))$ and $\valp_\delta\in L^2(\ren)$ satisfy
\begin{align}
\nonumber
  & \supp J_\delta\subseteq[0,\delta], &&&&
  \supp\valp_\delta\subseteq\ov{\bll}_\delta(\ve{0}),\\
\label{E:AlphaProp}
  &\valp_\delta(-\vz)=-\valp_\delta(\vz), && \text{and} &&
  \lint{\bll_\delta(\ve{0})}\nss\vz\bdot\valp_\delta(\vz)\dd\vz=1.
\end{align}
Thus, as $\delta\to0$, the masses of $J_\delta$ and $\valp_\delta$ concentrate around the origin. We set $\bnd_\delta:=\ann_{[0,\delta]}(\dom)$. Here, with $I$ an interval and $E\subseteq\ren$, we use $\ann_{I}(E)$ to denote the annular region $\{\vx\in\ren:d_E(\vx)\in I\}$, where $d_E$ is defined in~\eqref{D:Distance}.

In~\cite{DuTia:18a}, an argument for a one-dimensional version of~\eqref{E:IntroMainPoinc1} is provided. The proof relies on an analysis of the Fourier transform of $\mathbb{G}_\delta u$ and thus requires $u$ to satisfy periodic boundary conditions and have zero-mean over $\dom$. Additionally, the kernel $\rho_\delta$ is assumed to be odd and homogeneous throughout $\dom$. Ultimately, with $0<\alpha<1$, they show the stronger result that there is $C<\infty$, independent of $u$, such that $\|u\|_{H^\alpha(\dom)}\le C\|\mathbb{G}_\delta u\|_{L^2(\dom)}$. Moreover, in the limit as $\delta\to0$, the Poincar\'{e} constant becomes independent of $\delta$ and $\mathbb{G}_\delta$ converges to the classical derivative operator. The proof, however, requires the kernel to satisfy $\rho_\delta(z)=J_\delta(|z|)z^{-1}\sim|z|^{-1-\alpha}$ near the origin. (Note that $\rho_\delta=J_\delta$ in the notation of~\cite{DuTia:18a}.) Thus the kernel is not an integrable function, and the operator is understood in the sense of a Cauchy principle value. Under similar requirements, this result has been extended in~\cite{DuLee:19a} to nonsymmetric kernels supported on the half-interval $(0,\delta)$.

Regarding inequalities of the form~\eqref{E:IntroMainPoinc2}, we describe three strategies of proof. We assume that $u=0$ a.e. on $\bnd_\delta$. In the context of a nonlocal diffusion problem, a Poincar\'{e} inequality of the form~\eqref{E:IntroMainPoinc2} was proved in~\cite{AndMazRos:08a} (see also~\cite{AksMen:10a,AksPar:11a}). It is assumed that $J_\delta$ has the additional properties $\supp J_\delta=[0,\delta]$, $J_\delta\in\cnt([0,\infty))$, and $\|J_\delta\|_{L^1(\ren)}=1$. The Poincar\'{e} inequality is obtained by iterating and summing the inequalities
\[
  \|u\|_{L^p(S_k)}^p
  \le
  C'_k\lint{S_k}\lint{S_{k-1}}
    |u(\vy)-u(\vx)|^p J_\delta(\|\vy-\vx\|_{\ren})\dd\vy\dd\vx
  +C''_k\|u\|_{L^p(S_{k-1})}^p
\]
through $k=0,\dots,k_0$. Here $S_{-1}:=\bnd_\delta$ and $S_k:=\dom\cap\ann_{(kr,(k+1)r]}(\partial\dom)$, for a sufficiently small $0<r\le\frac{\delta}{2}$. The collection $\{S_k\}_{k=0}^{k_0}$ must be a decomposition of $\dom$ into concentric annular regions, so $k_0\sim\frac{\diam(\dom)}{r}$. Hence, for the resulting Poincar\'{e} constant, one finds that there exists an $M>1$ such that $C\sim M^\frac{\diam(\dom)}{r}\ge M^\frac{\diam(\dom)}{\delta}$, which scales poorly as $\delta\to0$. For local-nonlocal correspondence or condition analysis a different argument is needed. A second strategy to proving~\eqref{E:IntroMainPoinc2} requires establishing some variant of the following compactness result: if $V$ is a closed subspace of $L^2(\dom\cup\bnd_\delta)$ such that the only constant function is the zero function and $\{u_k\}_{k=1}^\infty\subset V$ converges weakly to some $u\in V$, then
\begin{equation}\label{E:Compactness}
  \lim_{k\to\infty}\lint{\dom}\lint{\bll_\delta(\vx)}
    |u_k(\vy)-u_k(\vx)|^2 J_\delta(\|\vy-\vx\|_{\ren})\dd\vy\dd\vx=0
  \Longrightarrow
  \lim_{k\to\infty}\|u_k\|_{L^2(\dom)}=0.
\end{equation}
The Poincar\'{e} inequality then follows a proof by contradiction. An argument along these lines has been used in~\cite{AksMen:10a,AksPar:11a,AksUnl:14a,DuMen:13a}. In~\cite{Du:19a} (Lemma~5.11), a Poincar\'{e} inequality with nonhomogeneous support $\bll_{d_{\partial\dom}(\vx)}(\vx)$ is proved. Poincar\'{e}-Korn type inequalities, in a vector-valued setting, are provided in ~\cite{DuMen:14a,DuMen:14b,DuMen:15a,Men:12a}. For versions of~\eqref{E:Compactness} with a general $1<p<\infty$, we refer to~\cite{BouBreMir:01a,DuMen:15a,Pon:03a,Pon:04a}. While this yields a Poincar\'{e} inequality that holds for general $\delta>0$, as this is an indirect argument, it is only possible to identify the dependence of the Poincar\'{e} constant on $\delta$ for sufficiently small $\delta$ . The last argument, that we describe, can be found in a currently unpublished work~\cite{DicEstTuc:Pre}. It follows a standard proof for the local Poincar\'{e} inequality based on Gauss' theorem~\cite{GueLee:96a}. Since $\valp_\delta$ is antisymmetric, the Fubini-Tonelli theorem and the change of variables $\vy\mapsto\vx+\vz$ yields, for any $\vx_0\in\ren$,
\begin{multline}
\label{E:ThirdStrat}
  \lint{\dom\cup\bnd_\delta}|u(\vy)|^2(\vy-\vx_0)
    \bdot\lint{\bll_\delta(\ve{y})}\ns\valp_\delta(\vy-\vx)\dd\vx\dd\vy\\
  \ns=\lint{\dom\cup\bnd_\delta}\lint{\bll_\delta(\ve{0})}
    |u(\vx+\vz)|^2(\vx+\vz-\vx_0)\bdot\valp_\delta(\vz)\dd\vz\dd\vx=0.
\end{multline}
Using this, H\"{o}lder's inequality, and the assumptions in~\eqref{E:AlphaProp},
\[
  \|u\|^2_{L^2(\dom)}
  =
  \lint{\dom\cup\bnd_\delta}\lint{\bll_\delta(\ve{0})}
    \lp|u(\vx)|^2-|u(\vx+\vz)|^2\rp(\vx+\vz-\vx_0)
      \bdot\valp_\delta(\vz)\dd\vz\dd\vx
  \le
  r_02^\frac{1}{2}I_1^\frac{1}{2}\cdot I_2^\frac{1}{2},
\]
with $r_0:=\sup_{\vy\in\dom\cup\bnd_\delta}\|\vy-\vx_0\|_{\ren}$,
\begin{align*}
  I_1&:=\lint{\dom\cup\bnd_\delta}\lint{\bll_\delta(\ve{0})}
    \ls|u(\vx+\vz)|^2+|u(\vx)|^2\rs\dd\vz\dd\vx,
\intertext{and}
  I_2&:=\lint{\dom\cup\bnd_\delta}\lint{\bll_\delta(\ve{0})}
    |u(\vx+\vz)-u(\vx)|^2\|\valp_\delta(\vz)\|_{\ren}^2\dd\vz\dd\vx.
\end{align*}
Selecting $\vx_0\in\dom\cup\bnd_\delta$ so that $r_0=\frac{1}{2}\diam(\dom\cup\bnd_\delta)$ one may obtain the Poincar\'{e} inequality
\begin{equation}\label{E:IntroPoinc3}
  \|u\|^2_{L^2(\dom)}
  \le
  \delta^n\diam(\dom\cup\bnd_\delta)^2||\bll_1|
  \lint{\dom\cup\bnd_\delta}\lint{\bll_\delta(\ve{0})}
    |u(\vx+\vz)-u(\vx)|^2\|\valp_\delta(\vz)\|_{\ren}^2\dd\vz\dd\vx.
\end{equation}
If $Z=\ann_{(\tau\delta,\delta)}(\ve{0})$ and $\valp_\delta(\vz)=A\vz\|\vz\|^{-\beta-1}\chi_{Z}(\vz)$, for some $\beta\neq n-1$ and $0<\tau<1$, then assumption~\eqref{E:AlphaProp} implies $A=\frac{n-\beta+1}{n(1-\tau^{n-\beta+1})}\delta^{\beta-1}|\bll_\delta|^{-1}$ and~\eqref{E:IntroPoinc3} becomes
\[
  \|u\|^2_{L^2(\dom)}
  \le
  \lp\frac{(n-\beta+1)\diam(\dom\cup\bnd_\delta)}{n(1-\tau^{n-\beta+1})}\rp^2\nss
  \ns\lint{\dom\cup\bnd_\delta}\!\flint{Z}
    \lp\frac{\delta}{\|\vz\|_{\ren}}\rp^{2\beta}
    \frac{|u(\vx+\vz)-u(\vx)|^2}{\|\vz\|_{\ren}^2}\dd\vz\dd\vx.
\]
The argument can easily be expanded to exponents $1<p<\infty$ by using the inequality $|b|^p-|a|^p\le\lp|a|^{p-1}+|b|^{p-1}\rp|b-a|$. When $\beta=0$ and $p=2$, we see that the Poincar\'{e} constant above is comparable to the one produced in Example~\ref{Ex:BasicEx2}. A shortcoming for this argument is that the outer domain of integration in the upper bound includes the annular region $\bnd_\delta$. This is a byproduct of the Fubini-Tonelli theorem used to obtain~\eqref{E:ThirdStrat}.

In this paper, we provide broadly applicable Poincar\'{e} inequalities in an arbitrary spatial dimension $n\in\nats$. We focus on function spaces satisfying a zero-Dirichlet like condition~(BC). The set $\bnd$ may include some, none, or all of an annular neighborhood around $\dom$ or be completely contained within $\dom$. The arguments described above require $u\in L^p(\dom)$ a priori. In contrast, we need only assume $u\in L^1(\dom\cup\bnd)$ in addition to~(BC) and in some cases, $u\in L(\dom\cup\bnd)$ (Remark~\ref{R:SpecPoincCases} and Corollaries~\ref{C:PoincCor1} and~\ref{C:PoincCor2}). Under the assumptions of this paper, the $p$-integrability of $u$ is a consequence of the Poincar\'{e} inequalities. Furthermore, the proofs are based on direct arguments. As can be seen in some of the examples, this allows us identify an upper bound for the Poincar\'{e} constant with explicit dependence on $\dom$ and on the structural parameters for the kernel.

A more significant contribution of this paper towards the study and well-posedness of nonlocal problems is the flexibility allowed for the kernels. Many physical applications are inherently nonhomogeneous, a feature that is reflected in the kernel of the operator. One can see that the Poincar\'{e} inequalities described earlier in this section can be extended to nonhomogeneous kernels $\mu$ that dominate either $J_\delta$ or $\|\valp_\delta\|^2_{\ren}$, but this requires the support of the kernel to uniformly contain a fixed symmetric set of positive measure; i.e. there exists $E\subseteq\ov{\bll}_\delta$ such that $|E|>0$ and $E\subseteq\supp\mu(\vx,\cdot)$, for all $\vx\in\dom$. (We note that Lemma 5.11 in~\cite{Du:19a} is an exception to this statement.) There are important settings that are incompatible with this requirement, such as models for fractured materials. Theorems~\ref{T:MainPoinc1} and~\ref{T:MainPoinc2} allow kernels without any such uniform support or symmetry. In fact, there can be open sets $\dom_1,\dom_2\subset\dom$ such that $\dom_1\cap\dom_2=\emptyset$, $\partial\dom_1\cap\partial\dom_2\neq\emptyset$ and $\supp\mu(\vx_1,\cdot)\cap\supp\mu(\vx_2,\cdot)=\emptyset$, for all $\vx_1\in\dom_1$ and $\vx_2\in\dom_2$. Thus the support of $\mu$ deforms discontinuously across $\partial\dom_1\cap\partial\dom_2$ and has no common support between the two regions (see Example~\ref{Ex:Discontinuous}).

This work also contributes to the study of problems with Dirichlet like constraints on sets with co-dimension larger than one. There has recently been much interest in studying PDEs with Dirichlet conditions on low dimensional sets. A theory for elliptic problems, including trace theorems, Harnack inequalities, regularity results, and harmonic measures, has already been developed in~\cite{MayGuyFen:19b,MayGuyFen:18a,MayFenZha:18a,LewNys:18a,MayZha:19a}. Very little has been done within a nonlocal framework. As explained in the previous subsection, one may have $\vx_0\in\dom$ where $|\supp\mu(\vx,\cdot)|\to0$ as $\vx\to\vx_0$, which allows one to consider a zero-Dirichlet like condition on a set $\bnd$ with dimension $0\le m<n$. In this setting, the Poincar\'{e} inequality~\eqref{E:IntroLowDim} is obtained in Example~\ref{Ex:ManExample2}. For the purposes of comparison, put $s=\frac{\beta}{p(n-m)}>\frac{1}{p}$ and recall that $\delta(\vx)=\min\{d_\bnd(\vx),\delta_0\}$. We may rewrite~\eqref{E:IntroLowDim} as
\begin{equation}\label{E:IntroPoincLowDim}
  \|u\|_{L^p(\dom)}^p
  \le
  C\underbrace{\lint{\dom\bs\bnd}\flint{\bll_{\delta(\vx)}(\vx)}
    \left|\frac{u(\vy)-u(\vx)}{\delta(\vx)^{s(n-m)}}\right|^p\dd\vy\dd\vx}
    _{:= J_{s,\bnd}^p[u]}.
\end{equation}
Let us denote the set of functions for which $J^p_{s,\bnd}$ is finite by $\class{J}^p_{s,\bnd}$. The finiteness of $J^p_{s,\bnd}[u]$ implies, in some sense, that $u$ has $s$-order differentiability localized at $\bnd$. The factor $(n-m)$ accounts for the scaling for the measure of the tubular $r$-neighborhood of $\bnd$ with respect to $r$. Due to the Poincar\'{e} inequality, the $p$-integrability of $u\in\class{J}^p_{s,\bnd}$ is expected, but outside any tubular neighborhood of $\bnd$, the function $u$ might exhibit no additional regularity. This is a generalization of Lemma 5.11 of~\cite{Du:19a} where $p=2$ and $\bnd=\partial\dom$, so $m=n-1$. In~\cite{TiaDu:17a}, it was shown that one can identify a trace that is continuous from $\class{J}^2_{1,\partial\dom}$ is finite into the Sobolev–Slobodeckij space $W^{\frac{1}{2},2}(\partial\dom)$. It seems natural to expect a similar result for the space $\class{J}^p_{s,\bnd}$. In fact, ongoing work suggests that, as in the setting of Besov spaces~\cite{JonWal:84a}, for $0\le m\le n-1$ not necessarily an integer, there is a trace operator that continuously maps $\class{J}^p_{s,\bnd}$ into $W^{s',p}(\bnd)$, with $s'=\lp s-1/p\rp(n-m)$. This is currently work that is still in progress. The Poincar\'{e} inequality in~\eqref{E:IntroPoincLowDim} appears to be the first result towards developing a theory for nonlocal problems with constraints on higher co-dimensional sets.

\subsection{Organization of Paper}\label{SS:Organization}

The next section establishes the notation used throughout the paper. In Section~\ref{S:IntroExamples}, we present the arguments used for the main results specialized to a couple of basic examples. The results from the literature that we need are collected in Section~\ref{SS:CitedResults}, and the main lemmas are proved in Section~\ref{SS:NewLemma}. The main Poincar\'{e} inequalities and a couple of corollaries are stated in Section~\ref{S:MainResults}. We provide a complete proof for Poincar\'{e} Inequality II. The proof for a Poincar\'{e} Inequality I is essentially the same as one of the components of the argument for Poincar\'{e} Inequality II. To demonstrate the applicability of the results, we produce several examples in Section~\ref{S:GenExamples}.

\section{Definitions and Notation}\label{S:Definitions}

In this section, we summarize terminology and notation that is common throughout the paper. Notation that is specific to a particular argument is introduced as needed.

The class of Borel-measurable sets in $\ren$ is denoted by $\class{B}(\ren)$ and the Lebesgue-measurable sets are denoted by $\class{L}(\ren)$. Unless indicated otherwise, Lebesgue-measurability is intended for sets and maps. Let $E,F\subseteq\ren$ be given. Given $\vx\in\ren$, the set $\vx+E\subseteq\ren$ is the translation of $E$ by $\vx$. If $E\in\class{B}(\ren)$, its $s$-dimensional Hausdorff measure is denoted by $\haus^s(E)$. We use $\chi_E:\ren\to\{0,1\}$ for the \textbf{characteristic function} for $E$. In addition to the standard notation for $L^p$ spaces, if $G\subseteq\rem$ is measurable (i.e., Lebesgue-measurable), we use $L(E;G)$ for the space of $G$-valued Lebesgue measurable maps. The codomain may be omitted if it is $\re$, and $L^+(E)$ is the space of non-negative measurable functions. If $u\in L(E)$ is differentiable at $\vx\in E$, then $\partial_{\vx}u(\vx)\in\ren$ denotes its gradient. The \textbf{distance function} from $E$ is $d_E:\ren\to[0,\infty)$ defined by
\begin{equation}\label{D:Distance}
  d_E(\vx):=\inf\lc\|\vx-\vy\|_{\ren}:\vy\in E\rc.
\end{equation}
It is well-known that $d_E$ is a Lipschitz function, is differentiable a.e. on $\ren$, and $\|\partial_{\vx}d_E\|_{\ren}=1$ for a.e. $\vx\in\ren$. If $E$ is compact, then by the Kuratowski–Ryll-Nardzewski measurable selection theorem~\cite{AliBor:06a}, there exists a measurable \textbf{metric projection} $\vP_E:\ren\to E$ such that $\|\vx-\vP_E(\vx)\|_{\ren}=d_E(\vx)$, so $\vP_E(\vx)$ is a point in $E$ that is nearest to $\vx$.
The \textbf{Hausdorff distance} between $E$ and $F$ is
\[
  \dist(E,F):=\max\lc\sup_{\vx\in E}d_F(\vx),\sup_{\vy\in F}d_E(\vy)\rc.
\]
For each $r>0$, the open (\textbf{tubular}) \textbf{$r$-neighborhood} (or \textbf{$r$-parallel neigborhood}) of $E$ is $\bll_r(E):=\lc\vx\in\ren: d_E(\vx)<r\rc$ and the \textbf{$r$-boundary} of $E$ is $\sph_r(E):=\lc\vx\in\ren: d_E(\vx)=r\rc$. Given an interval $I\subseteq[0,\infty)$, we define the annular neighborhood $\ann_I(E):=\bigcup_{r\in I}\sph_r$. For brevity, we will use $\ann_r(E):=\ann_{(0,r)}(E)$, and we may write $\bll_r(\vx)$, $\sph_r(\vx)$, or $\ann_I(\vx)$ if $E=\{\vx\}$. A superscript will be included when the dimension of the set needs to be made explicit, so for example, $\sph_r^{n-1}(\vx):=\lc\vy\in\ren:\|\vy-\vx\|_{\ren}=r\rc$. The center set is $E=\{\ve{0}\}$ or irrelevant whenever it is unspecified. Given $0<\theta\le\pi$, $\vx\in\ren$, and $\vom\in\sph_1^n$, set
\begin{equation}\label{D:Cone}
  \set{C}_\theta(\vx;\vom)
  :=
  \lc\vy\in\ren:(\vy-\vx)\bdot\vom
    >\cos\theta\cdot\|\vy-\vx\|_{\ren}
  \rc,
\end{equation}
so $\set{C}_\theta(\vx;\vom)$ is an $n$-dimensional open cone, with central axis parallel to $\vom$, aperture $2\theta$, and vertex at $\vx$. The set $\set{C}_\frac{\pi}{2}(\vx;\vom)$ is the open half-space of $\ren$ bounded by the hyperplane with normal vector $\vom$ and containing $\vx$. The $\supth{i}$-coordinate vector is $\ve{e}_i\in\ren$. Given a multi-index $\ve{i}=(i_1,\dots,i_m)\in\{1,\dots,n\}^m$, with $1\le i_1<\cdots<i_m\le n$, we define an associated $m$-dimensional coordinate-hyperplane by
\begin{equation}\label{D:Hyperplane}
  \mathbb{H}^m=\mathbb{H}^m_{\ve{i}}:=\{\vx\in\ren:x_{i_1}=\cdots=x_{i_m}=0\}.
\end{equation}
As in~\cite{AliBor:06a}, we use $\Psi:E\twoheadrightarrow F$ to denote a \textbf{correspondence}, so $\Psi(\vx)\subseteq F$, for each $\vx\in E$. The \textbf{inverse correspondence} of $\Psi$ is $\Psi^{-1}:F\twoheadrightarrow E$ defined by
\[
  \Psi^{-1}(\vy):=\{\vx\in E:\vy\in\Psi(\vx)\}.
\]

\begin{definition} Let $U\subseteq\ren$ be an open set.
\begin{enumerate}[(a)]
\item Recall that $V$ is \textbf{compactly contained} $U$, denoted by $V\subset\subset U$, if $\ov{V}$ is a compact subset of $U$.
\item We say that a measurable set $E$ is \textbf{essentially compactly contained} in $U$, denoted $E\subset\subset U$ a.e., if there exists a set $V\subset\subset U$ such that $|E\bs\ov{V}|=0$.
\end{enumerate}
\end{definition}

\begin{definition}\label{D:CountCont} Let $E\subset\rem$ be a measurable set, and let $\class{U}$ be a countable collection of sets.
\begin{enumerate}[(a)]
\item We call $\class{U}$ a \textbf{countable open partition of $E$} if
\begin{itemize}
\item $U\subseteq E$ and $U$ is open, for each $U\in\class{U}$;
\item For each $U_1,U_2\in\class{U}$, we find $U_1\cap U_2\neq\emptyset$ if and only if $U_1=U_2$ (i.e. the members of $\class{U}$ are mutually disjoint);
\item $\left|E\bs\bigcup_{U\in\class{U}}U\right|=0$.
\end{itemize}
\item We will call a map $\vy:E\to\ren$ \textbf{countably Lipschitz} on a countable open partition $\class{U}$ of $E$ if, for each $U\in\class{U}$, there exists $K_U<\infty$ such that
\[
  \|\vy(\vx_2)-\vy(\vx_1)\|_{\ren}\le K_U\|\vx_2-\vx_1\|_{\rem},
  \frl\vx_1,\vx_2\in U.
\]
\item We say that $\vy:E\to\ren$ is \textbf{countably injective} on $\class{U}$ if $\vy$ is injective on each $U\in\class{U}$.
\end{enumerate}
\end{definition}
One may note that the graph of a countably Lipschitz $\vy$ can be identified with an $m$-dimensional rectifiable set in $\re^{m+n}$.

\begin{definition}
A map $\vy:\dom\to\ren$ is a \textbf{Lusin map} if for any Lebesgue null set $E\subseteq\dom$, the image $\vy(E)$ is also a Lebesgue null set.
\end{definition}

It will be useful to adopt some terminology related to discrete dynamical systems.
\begin{definition}\label{D:DynSystem}
Let $\vy:\ren\to\ren$ be given.
\begin{enumerate}[(a)]
\item The (\textbf{discrete}) \textbf{dynamical system} generated by $\vy$ is $\bbD=\bbD[\vy]:\ren\times\whls\to\ren$ defined by $\bbD[\vy](\vx,k)=\bbD(\vx,k):=\vy^k(\vx)$, where
\[
  \vy^k(\vx):=\left\{\begin{array}{ll}
    \vx, & k=0,\\
    \vy(\vy^{k-1}(\vx)):=\underbrace{(\vy\circ\cdots\circ\vy)}_{k\text{ times}}(\vx),
      & k\in\nats.
  \end{array}\right.
\]
\item The (\textbf{forward}) \textbf{orbit} of a point $\vx\in\ren$ in $\mathbb{D}$ is $\orb^+(\vx)=\orb^+_{\bbD}(\vx):=\{\vy^k(\vx)\}_{k=0}^\infty$. We similarly define the (\textbf{forward}) \textbf{orbit} of a set $E\subseteq\dom$ in $\bbD$ by $\orb^+(E)=\orb^+_{\bbD}(E):=\{\vy^k(E)\}_{k=0}^\infty$.
\item\label{D:Absorbtion} An open set $U\subseteq\ren$ \textbf{essentially absorbs} a measurable set $E\subseteq\ren$ in $\bbD$ if there exists a finite \textbf{absorption index} $\alpha\in\whls$ defined by
\[
  \alpha:=\min\lc k_0\in\whls
    :\vy^{k}(E)\subset\subset U\text{ a.e., for all }k\ge k_0\rc.
\]
(This is a modification of the usual definition; see~\cite{RobTea:05a}).
\end{enumerate}
\end{definition}

\begin{definition}\label{D:ParamDynSystem}
Let $\Phi\subseteq\ren$ and $\vy:\ren\times\Phi\to\ren$ be given.
\begin{enumerate}[(a)]
\item We call $\{\bbD_{\vzet}\}_{\vzet\in\Phi}$ a \textbf{parameterized} (\textbf{discrete}) \textbf{dynamical system} generated by $\{\vy(\cdot,\vzet)\}_{\vzet\in\Phi}$, with parameter space $\Phi$, if $\bbD_{\vzet}=\bbD[\vy(\cdot,\vzet)]$ is a discrete dynamical system for each $\vzet\in\Phi$.
\item  We say that an open set $U\subseteq\ren$ \textbf{essentially absorbs} a measurable $E$ in $\{\bbD_{\vzet}\}_{\vzet\in\Phi}$ if $U$ essentially absorbs $E$ a.e., for each $\vzet\in\Phi$.
\end{enumerate}
\end{definition}
In the case of a parameterized dynamical system, the absorption index may depend on $\vzet\in\Phi$.

\begin{definition}\label{D:BanachInd}
Given a map $\vy:E\to\ren$ the \textbf{Banach indicatrix} or \textbf{multiplicity function} of $\vy$ at $\vw\in\vy(E)$ over $F\subseteq\ren$ is
\[
  N_{\vy}(\vw,F):=\card(\vy^{-1}(\{\vw\})\cap F).
\]
\end{definition}

The following definition is taken from~\cite{Haj:93a}.
\begin{definition}\label{D:AppDiff}
A map $\vy:E\to\rem$ is \textbf{approximately totally differentiable} at a density point $\vx_0\in\ren$ of $E$ if there exists $\te{F}\in\re^{m\times n}$ such $\vx_0$ is also a density point of
\[
  \lc\vx\in E:\frac{\|\vy(\vx)-\vy(\vx_0)-\te{F}(\vx-\vx_0)\|_{\rem}}
    {\|\vx-\vx_0\|_{\ren}}<\vep\rc,
  \frl\vep>0.
\]
\end{definition}

\section{Introductory Examples}\label{S:IntroExamples}

In this subsection, we present a couple of examples to provide a context in which to describe the key ideas underlying the proofs. Additional, more general, examples are found in Section~\ref{S:GenExamples}. We will use notation and definitions found in Section~\ref{S:Definitions}.

\subsection{Argument for Poincar\'{e} Inequality I}
The first example is a simplified version of Theorem~\ref{T:MainPoinc1}.
\begin{example}\label{Ex:BasicEx1}
Put $\eta:=\frac{2}{n+1}\cdot\frac{\haus^{n-1}(\bll^{n-1})}{\haus^n(\bll^n)}$. Let $0<\delta\le 1$ be given. Define the correspondences $\Psi:\dom\twoheadrightarrow\ren$ by
\[
  \Psi(\vx):=\bll_\delta(\vx)\cap\set{C}_{\frac{\pi}{2}}(\vx;\ve{e}_1)
    =\{\vy\in\bll_\delta(\vx):(\vy-\vx)\bdot\ve{e}_1>0\}.
\]
Set $\bnd:=\ov{\bigcup_{\vx\in\dom}\Psi(\vx)}\bs\dom$. If $u\in L^1(\dom\cup\bnd)$ satisfies $u=0$ a.e. on $\bnd$, then
\begin{equation}\label{E:BasicEx1Poinc}
  \|u\|^2_{L^2(\dom)}
  \le
  \frac{4(1+\diam(\dom))^3}{\eta^2}\lint{\dom}\left|\,
    \flint{\Psi(\vx)}\frac{u(\vy)-u(\vx)}{\delta}\dd\vy\right|^2\dd\vx.
\end{equation}
\end{example}
\begin{remark}
\begin{enumerate}[(a)]
\item We only require $u$ to be defined on $\dom\cup\bnd$, but it can always be extended by zero. As mentioned in the introduction, if $u\in L^1(\ren)$, then the upper bound is well-defined a.e. in $\dom$.
\item We have set the exponent to $p=2$ for the sake of simplicity only. In Example~\ref{Ex:BasicEx1Ext}, we provide the generalization to $1\le p<\infty$ and nonhomogeneous measurable supports. Example~\ref{Ex:SignChange1} provides an inequality similar to~\eqref{E:BasicEx1Poinc} with a sign-changing kernel with symmetric support.
\item The value of $\eta<1$ is the nonzero coordinate of the centroid for a hemisphere of $\bll_1(\ve{0})$. For the first few dimensions, we find $n=1\Rightarrow\eta=\frac{1}{2}$, $n=2\Rightarrow\eta=\frac{4}{3\pi}$, and $n=3\Rightarrow\eta=\frac{3}{8}$.
\end{enumerate}
\end{remark}
Since~\eqref{E:BasicEx1Poinc} is translation invariant, without loss of generality, we assume that $\inf_{\vx\in\dom}\vx\bdot\ve{e}_1=0$. Our argument is based on the one used in the proof for a nonlocal Hardy type inequality in~\cite{TiaDu:17a}. The idea is to design a weight $\gamma\ge 1$ so that~\eqref{E:BasicEx1Poinc} can be established, instead, for $\|u\cdot\gamma^\frac{1}{2}\|^2_{L^2(0,1)}$. Given $0<\lambda<1$,we may use the convexity of $\tau\mapsto\tau^2$ and Jensen's inequality to deduce
\begin{equation}\label{E:BasicEx1Ineq1}
  \lint{\dom}|u(\vx)|^2\gamma(\vx)\dd\vx
  \le
  \frac{1}{1-\lambda}\lint{\dom}\left|\,\flint{\Psi(\vx)}[u(\vy)-u(\vx)]
    \dd\vy\right|^2\gamma(\vx)\dd\vx
  +\frac{1}{\lambda}
  \lint{\dom}\flint{\Psi(\vy)}|u(\vx)|^2\gamma(\vy)
    \dd\vx\dd\vy.
\end{equation}
Here we have relabeled $\vx\leftrightarrow\vy$ in the last term. The goal is to choose $\gamma$ so that the last integral can be absorbed into the left-hand side. Using the Fubini-Tonelli theorem and the assumption that $u=0$ on $\bnd$,
\begin{equation}\label{E:BasicEx1GammaReq}
  \lint{\dom}\flint{\Psi(\vy)}|u(\vx)|^2\gamma(\vy)
    \dd\vx\dd\vy
  =
  \lint{\dom}|u(\vx)|^2\lp\lint{\Psi^{-1}(\vx)}
    \frac{\gamma(\vy)}{|\Psi(\vy)|}\dd\vy\rp\dd\vx.
\end{equation}
The key requirement that $\gamma$ must satisfy is that the term in the parentheses is uniformly smaller that $\gamma(\vx)$. Observing that
\[
  \Psi^{-1}(\vx)=\dom\cap\bll_\delta(\vx)
    \cap\set{C}_{\pi/2}(\vx;-\ve{e}_1)
  \subseteq\{\vy\in\bll_\delta(\vx):(\vy-\vx)\bdot\ve{e}_1<0\},
\]
we see that this requirement is satisfied by any $\gamma$ that is uniformly strictly increasing in the $\ve{e}_1$ direction. In particular, if we define $\gamma(\vx):=1+\vx\bdot\ve{e}_1$ and $\nu=1-\frac{\eta\delta}{1+\diam(\dom)}<\eta<1$, we can argue that
\begin{equation}\label{E:Ex1JensIneq}
  \lint{\Psi^{-1}(\vx)}\frac{\gamma(\vy)}{|\Psi(\vy)|}\dd\vy
  \le\lp1-\frac{\eta\delta}{\gamma(\vx)}\rp\gamma(\vx)
  \le\nu\gamma(\vx),
  \frl\vx\in\dom.
\end{equation}
Returning to~\eqref{E:BasicEx1Ineq1}, using~\eqref{E:BasicEx1GammaReq}, for any $\nu<\lambda<1$, we can absorb the last integral into the lower bound to obtain
\[
  \underbrace{\lint{\dom}|u(\vx)|^2\gamma(\vx)\dd\vx}_{\ge\|u\|^2_{L^2(\dom)}}
  \le
  \frac{\lambda}{(\lambda-\nu)(1-\lambda)}
  \lint{\dom}\left|\,\flint{\Psi(\vx)}[u(\vy)-u(\vx)]\dd\vy\right|^2
    \gamma(\vx)\dd\vx.
\]
Since $\gamma\le1+\diam(\dom)$ on $\dom$, the inequality in~\eqref{E:BasicEx1Poinc} follows from making the optimal choice for $\lambda=\sqrt{\nu}$ and using $\sqrt{1-\vep}\le1-\frac{1}{2}\vep$, for $0<\vep<1$. One issue with the argument, as presented, is that unless $|u|^2\cdot\gamma\in L^1(\dom)$, the last term in~\eqref{E:BasicEx1Ineq1} is not finite and cannot be absorbed. This is overcome using an $L^\infty$ sequence that approaches $u$ and the Lebesgue dominated convergence theorem. It is at this step where the second part of~(BC) is needed.

Theorem~\ref{T:MainPoinc1} generalizes this example in three main ways. Define $Z:=\bll_\delta\cap\set{C}_{\frac{\pi}{2}}(\ve{0};\ve{e}_1)$. For this example, using the definition of $\gamma$ in the proof, the theorem yields a slight improvement of~\eqref{E:BasicEx1Poinc}:
\[
  \|u\|_{L^2(\dom)}^2\le\frac{4(1+\diam(\dom))^2}{\eta^2}\|\mathbb{G}_\delta u\|_{L^2(\dom)}^2,
\]
with
\[
  \mathbb{G}_\delta u(\vx)
  :=
  \flint{Z}\frac{u(\vx+\vz)-u(\vx)}{\delta}
    \lp1+\vx\bdot\ve{e}_1\rp^\frac{1}{2}\dd\vz.
\]
In either case the kernel component $\rho$ is homogeneous and has an unchanging sign. In addition to allowing a general exponent $1\le p<\infty$, the general result allows nonconstant, sign-changing $\rho$ with nonhomogeneous support. (See Examples~\ref{Ex:SignChange1} and~\ref{Ex:SignChange2} in Section~\ref{S:GenExamples}.) As mentioned in the introduction, with a nonhomogeneous kernel, identifying $\Psi^{-1}$ and an appropriate $\gamma$ is nontrivial. Nevertheless, one can even have $\Psi(\vx)\subseteq\bll_{d_\bnd(\vx)}(\vx)$, so $\lim_{\vx\to\vx_0}|\Psi(\vx)|=0$ for each $\vx_0\in\bnd$. This allows one to produce Poincar\'{e} inequalities with $\bnd$ a set with measure zero and even Hausdorff dimension smaller than $n$.

\subsection{Argument for Poincar\'{e} Inequality II}
Next we present an example of Theorem~\ref{T:MainPoinc2} (Poincar\'{e} Inequality II).
\begin{example}\label{Ex:BasicEx2}
Let $1\le p<\infty$, $0<\vep<\delta$, and an open $Z\subseteq\ann_{(\vep,\delta)}(\ve{0})$ be given. Set $\bnd:=\bigcup_{\vx\in\dom}\ov{(\vx+Z)}\bs\dom$. If $u\in L(\dom\cup\bnd)$ and $u=0$ a.e. on $\bnd$, then
\begin{equation}\label{E:BasicEx2Poinc}
  \lint{\dom}|u(\vx)|^p\dd\vx
  \le
  \diam(\dom)^p\lint{\dom}\flint{Z}
    \frac{|u(\vx+\vz)-u(\vx)|^p}{\|\vz\|^p_{\ren}}\dd\vz\dd\vx.
\end{equation}
\begin{remark}
\begin{enumerate}[(a)]
\item Unlike Example~\ref{Ex:BasicEx1}, we do not require $u\in L^1(\dom\cup\bnd)$.
\item Using the compactness result in~\cite{Pon:04a}, inequality~\eqref{E:BasicEx2Poinc} has been established elsewhere (e.g.~\cite{AksUnl:14a}). As explained in Section~\ref{SS:Significance}, however, the argument is indirect and requires $\delta>0$ sufficiently small and the support $Z$ to be rotationally symmetric.
\item In~\cite{BouBreMir:01a}, it is further shown that, if $u\in W^{1,p}(\dom)$ and $\vep\to0^+$ as $\delta\to 0^+$, then there exists a constant $0<C_{p,n}<\infty$ such that
\[
  \lim_{\delta\to0^+}\flint{Z}
    \frac{|u(\vx+\vz)-u(\vx)|^p}{\|\vz\|^p_{\ren}}\dd\vz
    =C_{p,n}\|\partial_{\vx}u(\vx)\|^p_{\ren}
    \text{ in }L^1(\ren).
\]
\item This example is a special case of Example~\ref{Ex:GFlow}.
\end{enumerate}
\end{remark}
\end{example}
The idea for the argument is to associate a discrete dynamical system $\bbD_{\vz}$ with each $\vz\in Z$. For each $k\in\whls$, we define $\vy,\vy^k:\ren\times Z\to\ren$ by
\begin{equation}\label{D:BasicEx2yk}
  \vy(\vx,\vz):=\lc\begin{array}{ll}
    \vx+\vz, & \vx\in\dom,\\
    \vx, & \vx\notin\dom,
  \end{array}\right.
  \nd
  \vy^k(\vx,\vz):=\lc\begin{array}{ll}
    \vx, & k=0,\\
    (\vy\circ\vy^{k-1})(\vx,\vz), & k\in\nats.
  \end{array}\right.
\end{equation}
Since $\|\vz\|_{\ren}>0$, for each $\vx\in\dom$, there is an index $k_0=k_0(\vx,\vz)\le\frac{\diam(\dom)}{\|\vz\|_{\ren}}$ such that $\vy^k(\vx,\vz)\notin\dom$, for all $k\ge k_0$. The orbit $\{\vy^k(\vx,\vz)\}_{k=0}^\infty$, in $\bbD_{\vz}$, consists of a translation of $\vx$ by $k\vz$ for $k\le k_0$ and is stationary at $\vx+k_0\vz$, for all $k>k_0$. Thus $\ren$, or $\bnd$ to be precise, is an absorption set in $\bbD_{\vz}$ for $\dom$, with absorbtion index
\begin{equation}\label{E:BasicEx2Absorp}
  \alpha=\alpha(\vz)=\max_{\vx\in\dom}k_0(\vx,\vz)
  \le\frac{\diam(\dom)}{\|\vz\|_{\ren}}<\infty.
\end{equation}
Given $\vz\in Z$, we find $\vy^\alpha(\dom,\vz)\subseteq\bnd$. Since $u=0$ a.e. on $\bnd$, we may write
\begin{equation}\label{E:BasicEx2Ineq1}
  \lint{\dom}|u(\vx)|^p\dd\vx
  \le
  \alpha^{p-1}\sum_{k=0}^{\alpha-1}\underbrace{\lint{\dom}
    |u(\vy^{k+1}(\vx,\vz))-u(\vy^k(\vx,\vz))|^p\dd\vx}_{I_k}.
\end{equation}
By definition~\eqref{D:BasicEx2yk}, the integrand for $I_k$ vanishes whenever $\vx\notin\dom$. Set $X_k:=\vy^{-k}(\dom\cap\vy^k(\dom,\vz)\times\{\vz\})$, the change of variables $\vx+k\vz\mapsto\vx$ delivers
\begin{align}
\nonumber
  I_k
  &=
  \lint{X_k}|u(\vx+\vz+k\vz)-u(\vx+k\vz)|^p\dd\vx
  =
  \lint{\dom\cap\vy^k(\dom,vz)}|u(\vx+\vz)-u(\vx)|^p\dd\vx\\
\label{E:ExChangeOfVar}
  &\le
  \lint{\dom}|u(\vx+\vz)-u(\vx)|^p\dd\vx.
\end{align}
Incorporating this and the bound on the absorption index in~\eqref{E:BasicEx2Absorp} into~\eqref{E:BasicEx2Ineq1} yields
\[
  \lint{\dom}|u(\vx)|^p\dd\vx
  \le
  \diam(\dom)^p\lint{\dom}\frac{|u(\vx+\vz)-u(\vx)|^p}{\|\vz\|^p}\dd\vx
\]
To conclude ~\eqref{E:BasicEx2Poinc}, we integrate both sides of the above inequality with respect to $\vz$ over $Z$ and invoke the Fubini-Tonelli theorem.

For the general statement for Theorem~\ref{T:MainPoinc2}, we extend the argument presented for Example~\ref{Ex:BasicEx2} in three directions. First, the support of the kernel is not required to be homogeneous. Second, we augment it with a version of Poincar\'{e} Inequality I. The dynamical system argument is used to obtain an intermediate inequality for $\|u\|_{L^p(\dom\bs\ov{U})}$, where $U$ is an open absorbing set. If $u=0$ a.e. on $U$, then the argument is done. Otherwise, a specialized version of Theorem~\ref{T:MainPoinc1} is used to produce a bound for $\|u\|_{L^p(\ov{U})}$. Finally, there is the option of restricting the domain of integration in the lower bound to a measurable subset of $\dom$ and a, potentially smaller, outer domain of integration in the upper bound.

\begin{figure}
\begin{center}
\includegraphics[width=.8\textwidth,trim=145 500 145 125,clip]{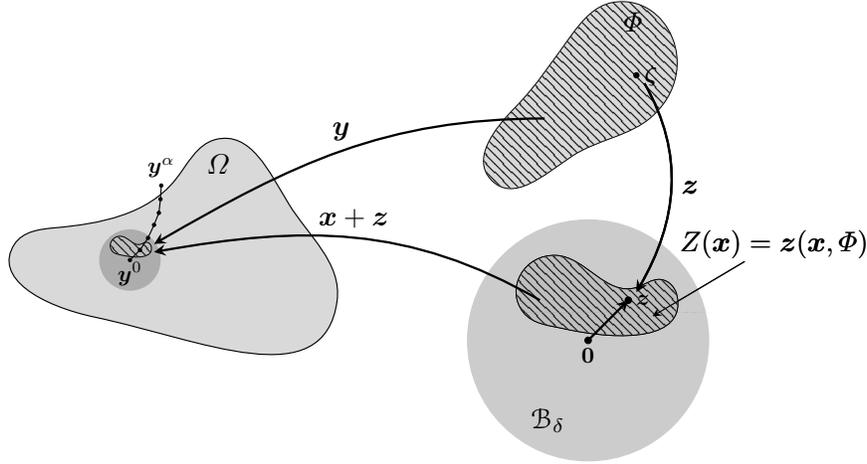}
\end{center}
\caption{Setup for the proof of Poincar\'{e} Inequality II}\label{F:Setup}
\end{figure}
To extend the argument to kernels with inhomogeneous supports, we introduce a parameter space $\Phi\subset\ren$ (Figure~1). For each $\vx\in\dom$ and $\vzet\in\Phi$ there corresponds a unique $\vz(\vx,\vzet)\in\supp(\mu(\vx,\cdot))=:Z(\vx)$. We define a map $\vy:\ren\times\Phi\to\ren$ using~\eqref{D:BasicEx2yk} and composition with $\vz$. Associated with $\vy$ is a parameterized family of discrete dynamical systems $\{\bbD_\vzet\}_{\vzet\in\Phi}$. For each $\vzet\in\Phi$, the forward orbit of $\vx\in\dom$ is $\{\vy^k(\vx,\vzet)\}_{k=0}^\infty$. In this setting, we assume $\bnd$ is a common absorbing set for $\dom$. The outline of the proof for the Poincar\'{e} inequality is essentially the same as presented in Example~\ref{Ex:BasicEx2}.
\begin{enumerate}
  \item\label{PhiDecomp} The parameter set $\Phi$ is decomposed into subsets $\Phi'_\alpha$ whose elements have a common associated absorption index $\alpha$, which was introduced in~\eqref{E:BasicEx2Absorp}.
  \item\label{xChangeOfVar} We produce the integrals in~\eqref{E:BasicEx2Ineq1} and use Theorem~\ref{T:ChangeVar} (a general change of variables theorem) to obtain the analogue of~\eqref{E:ExChangeOfVar}.
  \item\label{zChangeOfVar} We integrate over $\Phi'_\alpha$, apply the Fubini-Tonelli theorem, and use Theorem~\ref{T:ChangeVar} to transform the integral over $\Phi'_\alpha$ into one over $Z'_\alpha(\vx)\subseteq Z(\vx)$.
  \item\label{AbsorpStep} The Poincar\'{e} inequality is obtained by summing the resulting inequalities over the absorption indices.
\end{enumerate}
A minor refinement of the argument allows one to work with a measurable subset $\dom'\subseteq\dom$ and yields
\[
  \|u\|^p_{L^p(\dom')}
  \le
  C\lint{\dom\cap Y'}\lint{Z(\vx)}|u(\vx+\vz)-u(\vx)|^p\mu(\vx,\vz)\dd\vz\dd\vx.
\]
Here $Y'=\bigcup_{k=0}^\infty\vy^k(\dom'\times\Phi)$ is the union of all the forward orbits of $\dom'$.

In Example~\ref{Ex:BasicEx2}, the parameter set was $Z$ itself, so the second change of variables was not needed, and~\eqref{E:BasicEx2Absorp} provided a direct connection between $\vz$ and $\alpha$, so step~\ref{AbsorpStep} was also unnecessary. Steps~\ref{xChangeOfVar} and~\ref{zChangeOfVar} are the content of Lemma~\ref{L:OpenAbsorptionParamSet}. For the change of variables, Theorem~\ref{T:ChangeVar} is flexible enough to allow a null set $\dom_0$ across which $Z(\vx)$ changes discontinuously (see Definition~\ref{D:CountCont} and Example~\ref{Ex:Discontinuous}). A Lusin condition is imposed on $\vy$ to ensure this null set of discontinuities remains a null set under compositions of $\vy$. Another issue that arises is that $\dom_0$ might not be absorbed by $\bnd$. This leads to the introduction of essentially absorbing sets (Definition~\ref{D:DynSystem}(\ref{D:Absorbtion})). Lemma~\ref{L:Phi'Measurable} shows that the presence of a non-absorbed null set does not affect the measurability of $\Phi'_\alpha$.
%\end{example}

\section{Supporting Theorems and Lemmata}\label{SS:Lemmas}

In this section, we collect several results that provide key components in the proofs of the main results. Results found in the literature first subsection, we collect several results from the literature. In the second subsection, some new results are proved.

\subsection{Cited Results}\label{SS:CitedResults}

The first statement follows from the discussion in~\cite{BojIwa:83a} (p. 304).
\begin{lemma}\label{L:LusinImages}
If $\vy\in\cnt(\dom;\ren)$ is a Lusin map, then $\vy(E)\in\class{L}(\ren)$ whenever $E\subseteq\class{L}(\dom)$.
\end{lemma}

Regarding the measurability of the Banach indicatrix function, we have~\cite{Evs:18a} (see also~\cite{BojIwa:83a,Haj:93a})
\begin{lemma}\label{L:IndicatrixMeas}
If $\vy:E\to\ren$ is a Lebesgue measurable Lusin map, with $E\in\class{L}(\rem)$, then for each measurable $E\subseteq\dom$, the function $\vw\mapsto N_{\vy}(\vw,E)$ is Lebesgue measurable on $\ren$.
\end{lemma}

The following is the statement of Theorem 2 in~\cite{Haj:93a} (see also~\cite{BojIwa:83a}).
\begin{theorem}[Change of Variables Formula]\label{T:ChangeVar}
Let $U\subseteq\ren$ be an open set, and suppose that $\vy:U\to\ren$ is a Lusin map and approximately totally differentiable a.e. in $U$. Then given any measurable $E\subseteq U$ and any $f\in L+(\vy(E))$, the functions $\vx\mapsto f(\vy(\vx))|\det\partial\vy(\vx)|$ and $\vw\mapsto f(\vw)N_{\vy}(\vw,E)$ are well-defined and measurable and
\begin{equation}\label{E:ChangeVar}
  \lint{E}f(\vy(\vx))|\det\partial\vy(\vx)|\dd\vx
  =
  \lint{\vy(E)}f(\vw)N_{\vy}(\vw,E)\dd\vw.
\end{equation}
\end{theorem}

We will also use a corollary to the coarea formula, taken from~\cite{EvaGar:15a}.
\begin{corollary}\label{C:DistCoarea}
Let $0<a<b<\infty$ and a nonempty compact set  $G\subset\ren$ be given. Then, for any non-negative $f\in L^1(\ann_{(a,b)}(\Gamma)$,
\[
  \lint{\ann_{(a,b)}(G)}f(\vx)\dd\vx
  =\lint{(a,b)}\lint{\sph_\tau(G)}f(\vw)\dd\Haus^{n-1}(\vw)\dd\tau.
\]
\end{corollary}

The following theorem is proved in~\cite{Pon:87a}.
\begin{theorem}
If $U\subseteq\ren$ is an open set and $\vy\in\cnt(U;\ren)$ is differentiable a.e. in $\dom$, then the following are equivalent:
\begin{itemize}
\item[$\bullet$] $|\vy^{-1}(E)|=0$ for all measurable sets $E\subseteq\dom$ that have zero measure.
\item[$\bullet$] $\rank\partial_{\vx}\vy(\vx)=n$ for a.e. $x\in\dom$
\end{itemize}
\end{theorem}
We will need the immediate
\begin{corollary}\label{C:LipPreImage}
Suppose $\vy\in\cnt(U;\ren)$ is locally Lipschitz and $|\det\partial_{\vx}\vy(\vx)|\neq 0$ for a.e. $\vx\in U$. Then $|\vy^{-1}(E)|=0$, for all $E\in\class{L}(U)$ with zero measure.
\end{corollary}

\subsection{New Lemmata}\label{SS:NewLemma}

Throughout this section
\begin{itemize}
\item[$\bullet$] $I$ and $J$ are nonempty countable index sets;
\item[$\bullet$] $\dom,\Phi\subset\ren$ are nonempty, open and bounded sets;
\item[$\bullet$] $E\subseteq\ren$ is a measurable set;
\item[$\bullet$] $\{\dom_i\}_{i\in I}$ and $\{\Phi_j\}_{j\in J}$ are countable open partitions of $\dom$ and $\Phi$, respectively;
\item[$\bullet$] $\vz\in L(\ren\times\Phi;\ren)$ satisfies $\vz(\vx,\vzet)=\ve{0}$ for all $\vx\notin\dom$ and $\vzet\in\Phi$.
\item[$\bullet$] $\vy,\vy^k\in L(\ren\times\Phi;\ren)$ are defined by, for all $k\in\whls$,
\begin{equation}\label{D:yMap}
  \vy(\vx,\vzet)=\vx+\vz(\vx,\vzet)
  \text{ and }
  \vy^k(\vx,\vzet)
  :=\lc\begin{array}{ll}
    \vx, & k=0,\\
    \vy(\vy^{k-1}(\vx,\vzet),\vzet), & k>0.
    \end{array}\right.
\end{equation}
Since $\vz(\vx,\vzet)=\ve{0}$, for $\vx\notin\dom$, we find $\vy^k(\vx,\vzet)=\vy^{k'}(\vx,\vzet)$, for all $k\ge k'$, whenever $\vy^{k'}(\vx,\vzet)\notin\dom$.
\item[$\bullet$] For each $\vzet\in\Phi$, we define the section $\vy_\vzet:\ren\to\ren$ by $\vy_\vzet(\vx):=\vy(\vx,\vzet)$.
\end{itemize}

The first lemma in this section is based on a generalization of the argument described in Example~\ref{Ex:BasicEx2}. Given a measurable $\dom'$, we iteratively use the change of variables formula to compare $\|u\|_{L^p(\dom')}$ to the $L^p$-norm of $u$ over a set in the forward orbit of $\dom'$ associated to the parameterized dynamical system generated by $\vy$.
\begin{lemma}\label{L:OpenAbsorptionParamSet}
Let $\dom'\subseteq\dom$ and $\Phi'\subseteq\Phi$ be measurable sets. Define $Z':\dom\twoheadrightarrow\ren$ and the sets $\{Y'_k\}_{k\in\whls}$ by
\[
  Z'(\vx):=\vz(\{\vx\}\times\Phi')
  \nd
  Y'_k:=\vy^k(\dom'\times\Phi').
\]
Assume each of the following.
\begin{enumerate}[(i)]
\item\label{A:zLipschitz} $\vz$ is countably Lipschitz on $\{\dom_i\times\Phi'\}_{i\in I}$
\item\label{A:zInjLusin} $\vzet\mapsto\vz(\vx,\vzet)$ is a Lusin map and injective on $\Phi'$, for each $\vx\in\dom'$;
\item\label{A:yInjLusin} $\vx\mapsto\vy(\vx,\vzet)$ is a Lusin map on $\dom$ and countably injective on $\{\dom_i\}_{i\in I}$, for each $\vzet\in\Phi'$
%\item\label{A:yInjective} $\vx\mapsto\vy(\vx,\vzet)$ is injective on $\dom_i$, for each $i\in I$ and each $\vzet\in\Phi$;
%\item\label{A:yMapLusin} $\vx\mapsto\vy(\vx,\vzet)$ is a Lusin map for each $\vzet\in\Phi'$;
\item\label{A:yIndicatrix} There exists $\{N'_k\}_{k\in\whls}\subset L^\infty(\ren\times\ren;\whls)$ such that
\[
  N_{\vy^k_{\vzet}}(\vx,\dom')\le N'_k(\vx,\vz(\vx,\vzet)),
  \quad\text{for a.e. }(\vx,\vzet)\in\ren\times\Phi';
\]
\item\label{A:yzDetBounds} There exists $\Lambda,\Theta<\infty$ such that for a.e. $(\vx,\vzet)\in\dom\times\Phi'$,
\[
  |\det\partial_{\vx}\vy(\vx,\vzet)|\Theta\ge 1,
  \nd
  |\det\partial_{\vzet}\vz(\vx,\vzet)|\Lambda\ge 1.
\]
\end{enumerate}
Then for each $k_0\in\nats$,
\begin{multline}\label{E:PoincareIneqBase}
  |\Phi'|\lint{\dom'}|u(\vx)|^p\dd\vx
  \le \Lambda k_0^{p-1}\lp\sum_{k=0}^{k_0-1}\Theta^k
  \lint{\dom\cap Y'_k}\lint{Z'(\vx)}
    N'_k(\vx,\vz)|u(\vx+\vz)-u(\vx)|^p\dd\vz\dd\vx\right.\\
  \left.+\Theta^{k_0}\lint{U\cap Y'_{k_0}}\lint{Z'(\vx)}
    N'_{k_0}(\vx,\vz)|u(\vx)|^p\dd\vz\dd\vx\rp,
\end{multline}
where $U\subseteq\ren$ is any measurable set satisfying $|Y'_{k_0}\bs U|=0$.
\end{lemma}
\begin{remark}
\begin{enumerate}[(a)]
\item By definition, for each $\vzet\in\Phi$, the map $\vy^0(\cdot,\vzet)$ is the identity. Thus, we may take $N'_0\equiv 1$.
\item For each $\vzet\in\Phi$ and $i\in I$, the Lipschitz assumption implies $\vy_{\vzet}$ is a Lusin map on $\dom_i$. Kirszbraun's theorem (see~\cite{Fed:69a}) yields a Lipschitz, and hence Lusin, extension of $\vz_{\vzet}$ and $\vy_{\vzet}$ to $\ov{\dom}_i$. Nevertheless, there might be $\vzet\in\Phi$ such that there is no extension $\ov{\vz}_\vzet:\dom\to\ren$ that is Lusin on $\dom$ and equal to $\vz$ on $\bigcup_{i\in I}\dom_i$. (Consider, for example, the Cantor function.) Therefore, the Lusin condition in assumption~(\ref{A:yInjLusin}) is necessary. Moreover, since the sum of two Lusin maps need not be a Lusin map~\cite{Pok:08a}, the Lusin assumption must be imposed on $\vy_{\vzet}$, rather than $\vz_{\vzet}$.
\end{enumerate}
\end{remark}
\begin{proof}
We may assume that the integrals in the upper bound appearing in~\eqref{E:PoincareIneqBase} are finite. Note that $\vy$ and $\vy_{\vzet}$ are both measurable and thus, $u\circ\vy^k$ and $u\circ\vy_{\vzet}^k$, are both measurable for any $k\in\whls$ and any $\vzet\in\Phi$. Moreover, the countably Lipschitz property and Radamacher's theorem imples $\vy$ is approximately totally differentiable a.e. in $\ren$. Put $\dom_0:=\dom\bs\bigcup_{i\in I}\dom_i$, so $|\dom_0|=0$. This and assumption~(\ref{A:yzDetBounds}) imply, for each $\vzet\in\Phi$, that $|\vy_{\vzet}^{-1}(\dom_0)|=0$. By assumption~(\ref{A:zInjLusin}) and Lemma~\ref{L:LusinImages}, the correspondence $Z'$ is measurable-valued. For each $i\in I$, we introduce $X_i:\Phi'\twoheadrightarrow\dom$ defined by $\dom_i\cap\vy^{-1}_{\vzet}(\dom\bs\dom_0)$, which is also measurable-valued.

Fix $\vzet\in\Phi'$. We will use an induction argument to establish the following intermediate inequality (Step~\ref{xChangeOfVar} in the discussion after Example~\ref{Ex:BasicEx2}):
\begin{align}\label{E:UnparamPoincareIneq}
\nonumber
  \lint{\dom'}|u(\vx)|^p\dd\vx
  &\le
  k_0^{p-1}\lp\sum_{k=0}^{k_0-1}\Theta^k\lint{\dom\cap Y'_k}
    N_{\vy^k_{\vzet}}(\vx,\dom')|u(\vx+\vz(\vx,\vzet))-u(\vx)|^p\dd\vx\right.\\
  &\qqqquad\qqqquad+
  \left.\Theta^{k_0}\lint{Y'_{k_0}}N_{\vy_{\vzet}^{k_0}}(\vx,\dom')
    |u(\vx)|^p\dd\vx\rp.
\end{align}
For each $\vx\in\dom'$,
\[
  |u(\vx)|^p
  \le
  k_0^{p-1}\lp\sum_{k=0}^{k_0-1}
    |u(\vy^{k+1}_{\vzet}(\vx))-u(\vy^k_{\vzet}(\vx))^p
      +|u(\vy^{k_0}_{\vzet}(\vx))|^p\rp.
\]
Since $\vx\in\dom'$ is arbitrary, we may integrate over $\dom'$ and obtain
\begin{align}
\nonumber
  \lint{\dom'}|u(\vx)|^p\dd\vx
  &\le
  k_0^{p-1}\lp\sum_{k=0}^{k_0-1}
    \underbrace{\lint{\dom'}|u(\vy^{k+1}_{\vzet}(\vx))
      -u(\vy^k_{\vzet}(\vx))|^p\dd\vx}
      _{:=I_k}
    +\lint{\dom'}|u(\vy^{k_0}_{\vzet}(\vx))|^p\dd\vx\rp\\
\nonumber
  &=
  k_0^{p-1}\sum_{k=0}^0
  \underbrace{\lint{\dom'}|u(\vx+\vz(\vx,\vzet))-u(\vx)|^p\dd\vx}_{I_0}\\
\label{E:SumIneq1}
  &\quad+
  k_0^{p-1}\lp\sum_{k=1}^{k_0-1}
    \lint{\dom'}|u(\vy^{k+1}_{\vzet}(\vx))-u(\vy^k_{\vzet}(\vx))|^p\dd\vx
    +\lint{\dom'}|u(\vy^{k_0}_{\vzet}(\vx))|^p\dd\vx\rp
\end{align}
Let $0\le\ell< k_0-1$ be given. For the purposes of induction, assume that
\[
  I_k=\ns\lint{\dom\cap\vy^k_{\vzet}(\dom')}\nsss
  N_{\vy^k_{\vzet}}(\vx,\dom')
    |u(\vx+\vz(\vx,\vzet))-u(\vx)|^p\dd\vx,
  \frl 0\le k\le\ell
\]
and
\begin{equation}\label{E:IkForm}
  I_k=\ns\lint{\dom\cap\vy_{\vzet}^\ell(\dom')}\nsss
  N_{\vy^\ell_{\vzet}}(\vx,\dom')
    |u(\vy_{\vzet}^{k+1-\ell}(\vx))-u(\vy^{k-\ell}_{\vzet}(\vx))|^p\dd\vx,
  \frl \ell<k\le k_0-1.
\end{equation}
Our base case can be seen in~\eqref{E:SumIneq1}. For brevity, we will write $X_i$ for $X_i(\vzet)$. Since $\{\dom_i\}_{i\in I}$ is an open partition of $\dom$, if $\vx\notin\bigcup_{i\in I}X_i$, then either $\vy_\vzet(\vx)\notin\dom$ or $\vx\in\dom_0\cup\vy^{-1}(\dom_0)$. In the first case, we find $\vy^{k'}_{\vzet}(\vx)=\vy_{\vzet}(\vx)\notin\dom$, for all $k'\in\nats$. Consequently
\[
  |u(\vy_{\vzet}^{k+1-\ell}(\vx))-u(\vy^{k-\ell}_{\vzet}(\vx))|^p=0,
  \frl k>\ell.
\]
Since $\{X_i\}_{i\in I}$ are mutually disjoint and $|\dom_0\cup\vy^{-1}(\dom_0)|=0$, we deduce from~\eqref{E:IkForm} that
\begin{align}
\nonumber
  I_k
  &=
  \lint{\bigcup_{i\in I}X_i\cap\vy_{\vzet}^\ell(\dom')}\nsss
  N_{\vy^\ell_{\vzet}}(\vx,\dom')
    |u(\vy_{\vzet}^{k+1-\ell}(\vx))-u(\vy^{k-\ell}_{\vzet}(\vx))|^p\dd\vx\\
\label{E:IkiDef}
  &=
  \sum_{i\in I}
    \underbrace{\lint{X_i\cap\vy_{\vzet}^\ell(\dom')}\nsss
    N_{\vy^\ell_{\vzet}}(\vx,\dom')
    |u(\vy^{k+1-\ell}_{\vzet}(\vx))-u(\vy_{\vzet}^{k-\ell}(\vx))|^p\dd\vx}
      _{:=I_{k,i}},
\end{align}
for all $\ell<k\le k_0-1$. We want to make the change of variables $\vy_{\vzet}(\vx)\mapsto\vy'$. Assumption~(\ref{A:yInjLusin}), allows us to define a measurable inverse map $\vx_i:\dom\cap\vy_{\vzet}(\dom_i)\to\dom_i$ such that
\begin{equation}\label{E:xiInverse}
  \vy_{\vzet}(\vx_i(\vy_{\vzet}(\vx)))=\vx,
  \frl\vx\in\dom_i.
\end{equation}
Define $\wh{N}'_{\ell,i}:\dom\cap\vy_{\vzet}(\dom_i)\to\whls$ by
\begin{equation}\label{E:hatN'}
  \wh{N}'_{\ell,i}(\vy')
  :=
  N_{\vy^\ell_{\vzet}}(\vx_i(\vy'),\dom').
\end{equation}
We note that $\|\wh{N}'_{\ell,i}\|_{L^\infty}\le\|N'_\ell\|_{L^\infty}$, by assumption~(\ref{A:yIndicatrix}). The inverse property in~\eqref{E:xiInverse} implies
\[
  N_{\vy^\ell_{\vzet}}(\vx,\dom')
  =
  \wh{N}'_{\ell,i}(\vy_{\vzet}(\vx))
  \frl\vx\in\dom_i.
\]
Lemma~\ref{L:IndicatrixMeas} yields the measurability of $\vx\mapsto N_{\vy^\ell_{\vzet}}(\vx,\dom')$. Since $\vx_i$ is also measurable, we conclude that $N'_{\ell,i}$ is measurable. We may therefore make the desired change of variables in~\eqref{E:IkiDef} and obtain
\[
  I_{k,i}
  \le
  \Theta\nss\lint{\vy_{\vzet}(X_i\cap\vy_{\vzet}^\ell(\dom'))}\nsss
    \wh{N}'_{\ell,i}(\vy')|u(\vy_{\vzet}^{k-\ell}(\vy')
    -u(\vy^{k-\ell-1}_{\vzet}(\vy')))|^p \dd\vy'.
\]
Let $\vy'\in\bigcup_{i\in I}\vy_{\vzet}(X_i\cap\vy_{\vzet}^\ell(\dom'))$ be given. Recalling~\eqref{E:hatN'} and that $\vy_{\vzet}$ is injective on each $X_i\subseteq\dom_i$, we may write
\begin{align*}
  \sum_{i\in I}\wh{N}'_{\ell,i}(\vy')
  \cdot\chi_{\vy_{\vzet}(X_i\cap\vy_{\vzet}^\ell(\dom'))}(\vy')
  &=
  \sum_{i\in I}N_{\vy^\ell_{\vzet}}(\vx_i(\vy'),\dom')\cdot
    \chi_{X_i\cap\vy_{\vzet}^\ell(\dom')}(\vx_i(\vy'))\\
  &=
  \sum_{\vx\in\vy_{\vzet}^{-1}(\vy')\bs\dom_0}
    N_{\vy^\ell_{\vzet}}(\vx_i(\vy'),\dom')\\
  &=
  N_{\vy_{\vzet}^{\ell+1}}(\vy',\dom').
\end{align*}
We note that each of the sums above are infinite if and only if $\card(\vy_{\vzet}^{-1}(\vy')\bs\dom_0)=\infty$. Returning to~\eqref{E:IkiDef}, we may use the monotone convergence theorem to produce
\begin{align*}
  I_k
  &=
  \sum_{i\in I}I_{k,i}
  \le
  \Theta\sum_{i\in I}
    \ns\lint{\vy_{\vzet}(X_i\cap\vy_{\vzet}^\ell(\dom'))}\nsss
    \wh{N}'_{\ell,i}(\vy')|u(\vy_{\vzet}^{k-\ell}(\vy')
    -u(\vy^{k-\ell-1}_{\vzet}(\vy')))|^p \dd\vy'\\
  &=\Theta\nsss
  \lint{\vy_{\vzet}(\vy^{-1}(\dom)\cap\vy_{\vzet}^\ell(\dom'))}\ns
    \lp\sum_{i\in I}\wh{N}'_{\ell,i}(\vy')
      \cdot\chi_{\vy_{\vzet}(X_i\cap\vy_{\vzet}^\ell(\dom'))}(\vy')\rp\\
  &\qqqquad\qqquad
    \times|u(\vy_{\vzet}^{k-\ell}(\vy')
    -u(\vy^{k-\ell-1}_{\vzet}(\vy')))|^p \dd\vy'\\
  &\le
  \Theta\lint{\dom\cap\vy_{\vzet}^{\ell+1}(\dom')}\nss
    N_{\vy_{\vzet}^{\ell+1}}(\vy',\dom')
    |u(\vy_{\vzet}^{k-\ell}(\vy')
      -u(\vy^{k-\ell-1}_{\vzet}(\vy')))|^p \dd\vy'
\end{align*}
In particular,
\[
  I_{\ell+1}
  \le
  \Theta\lint{\dom\cap\vy_{\vzet}^{\ell+1}(\dom')}\nss
    N_{\vy_{\vzet}^{\ell+1}}(\vy',\dom')
    |u(\vy'+\vz(\vy',\vzet))-u(\vy')|^p\dd\vy'.
\]
This concludes the inductive step. A similar argument provides the corresponding bound for the last term in~\eqref{E:SumIneq1}. Using the fact that $Y'_k\supseteq\vy^k_{\vzet}(\dom')$, we may expand the domains of integration to obtain~\eqref{E:UnparamPoincareIneq}.

For the last part (Step~\ref{zChangeOfVar} in the Introduction) of the proof for~\eqref{E:PoincareIneqBase}, we integrate the result with respect to $\vzet\in\Phi'$. After using the Fubini-Tonelli theorem, we make another change of variables $\vz(\vx,\vzet)\to\vz'$. Since the lower bound of~\eqref{E:UnparamPoincareIneq} is independent of $\vzet$, we may use assumption~(\ref{A:yIndicatrix}) and write
\begin{align*}
  |\Phi'|\lint{\dom'}|u(\vx)|^p\dd\vx
  &\le
  k_0^{p-1}\lp\sum_{k=0}^{k_0-1}\Theta^k
  \lint{\dom\cap Y'_k}\lint{\Phi'}
    N'_k(\vx,\vz(\vx,\vzet))|u(\vx+\vz(\vx,\vzet))-u(\vx)|^p\dd\vzet\dd\vx\right.\\
  &\qqqquad+
  \left.\Theta^{k_0}\lint{Y'_{k_0}}\lint{\Phi'}\
    N'_{k_0}(\vx,\vz(\vx,\vzet))|u(\vx)|^p\dd\vzet\dd\vx\rp\\
  &\le
  \Lambda k_0^{p-1}\lp\sum_{k=0}^{k_0-1}\Theta^k
  \lint{\dom\cap Y'_k}\lint{Z'(\vx)}
    N'_k(\vx,\vz')|u(\vx+\vz')-u(\vx)|^p\dd\vz'\dd\vx\right.\\
  &\qqqquad+
  \left.\Theta^{k_0}\lint{U\cap Y'_{k_0}}\lint{Z'(\vx)}
    N'_{k_0}(\vx,\vz')|u(\vx)|^p\dd\vz'\dd\vx\rp.
\end{align*}
In the inner integral of the rightmost term, we used the assumption that $|Y'_{k_0}\bs U|=0$.
\end{proof}

When proving the Main Poincar\'{e} Inequality II, we will decompose the parameter space $\Phi$ into subsets corresponding to distinct absorption indices. The next lemma ensures these subsets are measurable, so that the previous lemma is applicable.
\begin{lemma}\label{L:Phi'Measurable}
Let a measurable $\dom'\subseteq\dom$ and an open $U\subseteq\ren\bs\dom'$ be given. Suppose that $U$ essentially absorbs $\dom'$ in $\{\bbD_{\vzet}\}_{\vzet\in\Phi}$. For each $\alpha\in\nats$ and $k\in\whls$, set
\[
  \Phi'_\alpha:=\{\vzet\in\Phi:\text{the absorption index for $\dom'$ in $\bbD_{\vzet}$ is $\alpha$}\},
\]
Assume the following.
\begin{enumerate}[(i)]
\item\label{A:zxZetaLipschitz} $\vz$ is countably Lipschitz on $\{\dom_i\times\Phi_j\}_{i\in I,j\in J}$.
\item\label{A:yxInjLusin} $\vx\mapsto\vy(\vx,\vzet)$ is a Lusin map on $\dom$ and countably injective on $\{\dom_i\}_{i\in I}$, for each $\vzet\in\Phi$
\item\label{A:detBounds} $\det\partial_{\vx}\vy(\vx,\vzet)\neq0$, for a.e. $(\vx,\vzet)\in\dom\times\Phi$.
\end{enumerate}
Then, for each $\alpha\in\nats$, the set $\Phi'_\alpha\in\class{B}(\ren)$.
\end{lemma}
\begin{proof}
We introduce some families of sets and correspondences. These include partitions of $\dom$ on which we show that iterated compositions of $\vy$ are countably Lipschitz continuous. Ultimately, this will allow us to argue that $\Psi'_\alpha$ can be identified as a member of the Borel hierarchy on $\ren$.

Define the Borel-null set $\dom_0:=\dom\bs\bigcup_{i\in I}\dom_i$. Fix $i_0\in I$ and $j_0\in J$, and set $\Upsilon:=\dom_{i_0}\times\Phi_{j_0}$. For each $k\in\nats$, $\ve{i}=\{(i_1,\dots,i_k)\}\in I^k$, define $Y_{\ve{i}}\subseteq\dom$, $\Upsilon_{\ve{i}}\subseteq\dom_{i_0}\times\Phi_{j_0}$ and $\dom_{\ve{i}},\dom_{0,k}:\Phi_{j_0}\twoheadrightarrow\dom$ as follows:
\begin{align}
\nonumber
  & Y_{(i_1,\dots,i_k)}
  :=\left\{\begin{array}{ll}
    \vy(\Upsilon)\cap\dom_{i_1}, & k=1,\\
    \vy^k(\Upsilon_{(i_1,\dots,i_{k-1})})\cap\dom_{i_k}, & k>1,
  \end{array}\right.\\
\label{E:UpsilonDef}
  &\Upsilon_{\ve{i}}
  :=
  \vy^{-k}(Y_{\ve{i}})\cap\Upsilon,
  \qqquad\dom_{\ve{i}}(\vzet):=\lc\vx\in\dom_{i_0}:(\vx,\vzet)
    \in\Upsilon_{\ve{i}}\rc.
\end{align}
If $\vx\in\dom_{\ve{i}}(\vzet)$, then
\[
  \vx\in\dom_{i_0},\:\vy(\vx,\vzet)\in\dom_{i_1},\:\vy^2(\vx,\vzet)\in\dom_{i_2},
  \dots,\:\vy^k(\vx,\vzet)\in\dom_{i_k}.
\]
The collection $\{\dom_{\ve{i}}(\vzet)\}_{\ve{i}\in I^k}$ identifies those points in $\dom_{i_0}$ with an orbit, in $\bbD_{\vzet}$, that remains in the partition $\{\dom_i\}_{i\in I}$ for at least the first $k$ steps. The multi-index $\ve{i}\in I^k$ captures the sequence of subdomains $\{\dom_{i_\ell}\}_{\ell=1}^k$ that the orbit of $\vx$ passes through in those first $k$ steps. Thus, for $k\ge 1$, we see that $\bigcup_{\ve{i}\in I^k}\dom_{\ve{i}}(\vzet)\subseteq\bigcup_{\ve{i}\in I^{k+1}}\dom_{\ve{i}}(\vzet)$. The set $\dom_{\ve{i}}(\vzet)$ can be empty, if there exists $1\le k'\le k$ for which $\vy^{k'}_{\vzet}(\dom_{i_0})\subseteq\dom_{0}\cup U$. We will show that the composition $\vy_{\vzet}^{k+1}$ is continuous on $\dom_{\ve{i}}$, for each $\ve{i}\in I^k$.

To this end there are some additional properties of the sets defined in~\eqref{E:UpsilonDef} that we need. We will show by induction that, for each $k\in\nats$,
\begin{itemize}
\item[$\bullet$] $Y_{\ve{i}}$ is an open set, for each $\ve{i}\in I^k$;
%\item[$\bullet$] $Y_{(i_1,\dots,i_k,i_{k+1})}=\dom_{i_{k+1}}
%    \cap\bigcup_{\vzet\in\Phi_{j_0}}\vy_{\vzet}(\dom_{(i_1,\dots,i_k)})$
\item[$\bullet$] $\{\Upsilon_{\ve{i}}\}_{\ve{i}\in I^k}$ and, for each $\vzet\in\Phi_j$, $\{\dom_{\ve{i}}(\vzet)\}_{\ve{i}\in I^k}$ are all collections of open sets;
\item[$\bullet$] $\vy^{k+1}$ is countably Lipschitz continuous on $\{\Upsilon_{\ve{i}}\}_{\ve{i}\in I^k}$;
\item[$\bullet$] $\vy^{k+1}_{\vzet}$ is countably injective on $\{\dom_{\ve{i}}(\vzet)\}_{\ve{i}\in I^k}$, for each $\vzet\in\Phi_{j_0}$.
\end{itemize}
First, we establish the base case. By hypothesis $\Upsilon$ is an open set and $\vy$ is continuous on $\Upsilon$. From their definitions, for each $i_1\in I$,
\begin{equation}\label{E:BaseCaseSets}
  Y_{(i_1)}
  =
  \dom_{i_1}\cap\bigcup_{\vzet\in\Phi_{j_0}}
    \vy_\vzet(\dom_{i_0})
  \nd
  \dom_{(i_1)}(\vzet)
  =\vy^{-1}_{\vzet}(\dom_{i_1})\cap\dom_{i_0}.
\end{equation}
By assumptions~(\ref{A:zxZetaLipschitz}) and~(\ref{A:yxInjLusin}), for each $\vzet\in\Phi_{j_0}$, the map $\vy_{\vzet}$ is Lipschitz and injective on $\dom_{i_1}$. It follows, by the domain invariance theorem, that $\vy_\vzet(\dom_{i_0})$ is open, and hence, so is $Y_{(i_1)}\subseteq\dom_{i_1}$. Invoking the Lipschitz continuity of $\vy$ on $\dom_{i_1}$, again, we conclude that $\Upsilon_{(i_1)}$ is the intersection of two open sets, as well. Since, for each $i_1\in I$, the map $\vy$ is Lipschitz continuous on both $\Upsilon_{(i_1)}\subseteq\Upsilon$ and $\vy(\Upsilon_{(i_1)})\times\Phi_{j_0}\subseteq\dom_{i_1}\times\Phi_{j_0}$, we see that $\vy^2$ is countably Lipschitz on $\{\Upsilon_{(i_1)}\}_{i_1\in I}$. The injectivity property for $\vy_{\vzet}^2$ follows from the injectivity of $\vy_\vzet$ on $\dom_{i_0}$ and on $\dom_{(i_1)}(\vzet)\subseteq\dom_{i_1}$. For the inductive argument, let $k\in\nats$ be given, and for each $1\le k'\le k$, assume that each of the properties listed above hold. For each $\ve{i}=(i_1,\dots,i_k)\in I^k$ and $i_{k+1}\in I$, put $(\ve{i},i_{k+1}):=(i_1,\dots,i_k,i_{k+1})$. We find the following analogues of~\eqref{E:BaseCaseSets}:
\begin{align*}
  Y_{(\ve{i},i_{k+1})}
  &=\dom_{i_{k+1}}
      \cap\bigcup_{\vzet\in\Phi_{j_0}}
        \vy^{k+1}_{\vzet}(\dom_{\ve{i}}(\vzet))
\intertext{and}
  \dom_{(\ve{i},i_{k+1})}(\vzet)
   &=\vy^{-k-1}_{\vzet}(\dom_{i_{k+1}}
      \cap\vy^{k+1}_{\vzet}(\dom_{\ve{i}}(\vzet)))\\
   &=\vy^{-k}_{\vzet}(\vy^{-1}_{\vzet}(\dom_{i_{k+1}})
      \cap\vy^k_{\vzet}(\dom_{\ve{i}}(\vzet)))
\end{align*}
As in the base case, assumption~(\ref{A:yxInjLusin}), the domain invariance theorem together with the inductive assumption that $\vy^{k+1}_\vzet$ is injective and Lipschitz continuous on the open set $\dom_{\ve{i}}(\vzet)$ implies $\dom_{(\ve{i},i_{k+1})}(\vzet)$, and thus $Y_{(\ve{i},i_{k+1}})=\dom_{i_{k+1}}\cap\vy^{k+1}(\Upsilon_{\ve{i}})$, is open. Using the inductive assumption that $\vy^{k+1}$ is Lipschitz on $\Upsilon_{\ve{i}}$, we deduce that $\Upsilon_{(\ve{i},i_{k+1})}$ is open. The Lipschitz continuity of $\vy^{k+2}$ and injectivity of $\vy^{k+2}_{\vzet}$ follows from the corresponding properties assumed for $\vy^{k+1}$, $\vy$, $\vy^{k+1}_{\vzet}$, and $\vy_{\vzet}$. This concludes the inductive argument. As $i_0\in I$ and $j_0\in J$ were arbitrary, this construction of open sets can be done for each member of $\{\dom_i\times\Phi_j\}_{i\in I,j\in J}$.

Fix $\alpha\in\nats$, $i\in I$, and $j\in J$. Set $\dom'_i:=\dom'\cap\dom_i$. A straightforward induction argument and assumption~(\ref{A:yxInjLusin}) verifies that $\vy_{\vzet}^\alpha$ is a Lusin map on $\dom$, for each $\vzet\in\Phi_j$. Let $\{\Upsilon_{\ve{i}}\}_{\ve{i}\in I^\alpha}$, and $\{\dom_{\ve{i}}\}_{\ve{i}\in I^\alpha}$ denote the sets and correspondences constructed above, with $\Upsilon=\dom_i\times\Phi_j$. For each $\ve{i}\in I^\alpha$, we use $\wh{\vy}_{\ve{i}}:\Upsilon_{\ve{i}}\to\ren$ to denote the restriction of $\vy$ to $\Upsilon_{\ve{i}}$. We will prove that
\[
  \Phi'_{\alpha,i,j}:=\{\vzet\in\Phi_j:
    \vy^\alpha_{\vzet}(\dom'_i)\subset\subset U\text{ a.e.}\}
\]
can be identified with the countable union of countable intersections of Borel sets. To this end, we make the observation that given an open set $U'\subseteq\ren$, a set $\dom''\subseteq\dom_i$, and $\vzet\in\Phi_j$,
\begin{align}\label{E:RestrictionEquiv}
  &\qquad\vy^\alpha_{\vzet}(\dom'')\subset\subset U'\text{ a.e.}\\
\nonumber
  \Longleftrightarrow
  &\qquad\text{\parbox[t]{3in}{there exists an open set $U''\subset\subset U'$ such that $\wh{\vy}_{\ve{i}}^\alpha(\dom''\cap\dom_{\ve{i}}(\vzet)\times\{\vzet\})\subset\subset U''$ a.e., for all $\ve{i}\in I^\alpha.$}}
\end{align}

For each $\ell\in\nats$, define
\[
  U_\ell:=\lc\vx\in U: d_U(\vx)>\smfrac{1}{\ell}\rc,
\]
so
\[
  U_\ell\subset\subset U_{\ell+1}\subset\subset U
  \nd
  U=\bigcup_{\ell=1}^\infty U_\ell.
\]
For each $\ve{i}\in I^\alpha$ and $\ell\in\nats$, define $\dom'_{\ve{i}}:\Phi_j\twoheadrightarrow\dom_i$ by $\dom'_{\ve{i}}(\vzet):=\dom'\cap\dom_{\ve{i}}(\vzet)$ and
\[
  \Phi'_{\ve{i},\ell}:=\lc\vzet\in\Phi_j:
    \wh{\vy}_{\ve{i}}^\alpha(\dom'_{\ve{i}}(\vzet)\times\{\vzet\})
    \subset\subset U_\ell\text{ a.e.}\rc
\]
We claim that $\Phi'_{\alpha,i,j}=\bigcup_{\ell=1}^\infty\bigcap_{\ve{i}\in I^\alpha}^\infty\Phi'_{\ve{i},\ell}$. Indeed, if $\vzet_0\in\Phi'_{\alpha,i,j}$, then there exists an $\ell\in\nats$ such that $\vy^\alpha_{\vzet_0}(\dom'_i)\subset\subset U_{\ell}$ a.e.. It follows from~\eqref{E:RestrictionEquiv} that $\vzet_0\in\Phi'_{\ve{i},\ell}$, for all $\ve{i}\in I^\alpha$. Hence $\Phi'_{\alpha,i,j}\subseteq\bigcup_{\ell=1}^\infty\bigcap_{\ve{i}\in I^\alpha}^\infty\Phi'_{\ve{i},\ell}$. For the opposing inclusion, suppose that $\vzet_0$ is a member of the countable union of countable intersections. Then there exists $\ell\in\nats$ such that $\wh{\vy}_{\ve{i}}^\alpha(\dom_{\ve{i}}'(\vzet_0)\times\{\vzet_0\})\bs\ov{U}_{\ell}$ has zero measure for every $\ve{i}\in I^\alpha$. Since $\ov{U}_{\ell}\subset U_{\ell+1}$, we conclude, from~\eqref{E:RestrictionEquiv}, that $\vy^\alpha(\dom'_i\times\{\vzet_0\})\subset\subset U_{\ell+1}$ a.e.. Thus $\Phi'_{\alpha,i,j}\supseteq\bigcup_{\ell=1}^\infty\bigcap_{\ve{i}\in I^\alpha}^\infty\Phi'_{\ve{i},\ell}$.

Fix $\ve{i}\in I^\alpha$ and $\ell\in\nats$. We will show that $\Phi_j\bs\Phi'_{\ve{i},\ell}$ is an open set. Let $K_{\ve{i}}<\infty$ be the Lischitz constant for $\wh{\vy}_{\ve{i}}$ in $\Upsilon_{\ve{i}}$. Let $\vzet_0\in\Phi_j\bs\Phi'_{\ve{i},\ell}$ be given. Then $|\wh{\vy}_{\ve{i}}^\alpha(\dom'_{\ve{i}}(\vzet_0)\times\{\vzet_0\})\bs\ov{U}_\ell|>0$. Since the Lebesgue measure is inner regular, there exists a compact $Y'_0\subseteq\wh{\vy}_{\ve{i}}^\alpha(\dom'_{\ve{i}}(\vzet_0)\times\{\vzet_0\})
\bs\ov{U}_\ell$ such that $|Y'_0|>0$. As $Y'_0\subseteq\ren\bs\ov{U_\ell}$, we find $\vep_0:=\dist(Y'_0,\ov{U_\ell})>0$. Set
\[
  X'_0:=\lc\vx\in\dom_{\ve{i}}(\vzet_0):\wh{\vy}_{\ve{i}}^\alpha(\vx,\vzet_0)\in Y'_0\rc.
\]
Now $\dom'_{\vzet_0}\subseteq\dom_{\ve{i}}(\vzet_0)$ and $\vy^\alpha_{\ve{i}}(\cdot,\vzet_0)$ is continuous on $\dom_{\ve{i}}(\vzet_0)$.
Thus assumption~(\ref{A:detBounds}) and Corollary~\ref{C:LipPreImage} imply $|X'_0|>0$. From the continuity and injectivity of $\wh{\vy}_{\ve{i}}(\cdot,\vzet_0)$, we conclude that $X'_0$ is compact and that $X'_0\subseteq\dom'_{\ve{i}}(\vzet_0)$. Recalling that $\Upsilon_{\ve{i}}$ is an open set, we conclude that $\delta_0:=\dist(X'_0\times\{\vzet_0\},\ren\bs\Upsilon_{\ve{i}})>0$. Let $\vzet\in\Phi_j$ such that $\|\vzet-\vzet_0\|_{\ren}<\delta:=\min\lc\delta_0,\frac{\vep_0}{K_{\ve{i}}}\rc$. Then $X'_0\times\{\vzet\}\subseteq\Upsilon_{\ve{i}}$, and so $X'_0\subseteq\dom'_{\ve{i}}(\vzet)$. Set $Y':=\wh{\vy}_{\ve{i}}^\alpha(X'_0,\vzet)$. Again, since $X'_0$ is compact, so is $Y'$. For each $\vx\in X_0'$, we have
\[
  \|\wh{\vy}^\alpha_{\ve{i}}(\vx,\vzet)-\wh{\vy}^\alpha_{\ve{i}}(\vx,\vzet_0)\|
  \le
  K_{\ve{i}}\|\vzet-\vzet_0\|_{\ren}
  <\vep_0.
\]
Thus, we must find $\wh{\vy}^\alpha_{\ve{i}}(\vx,\vzet)\in\ren\bs\ov{U_\ell}$, since $\wh{\vy}^\alpha_{\ve{i}}(\vx,\vzet_0)\in Y'_0$ and $\dist(Y'_0,\ov{U_\ell})=\vep_0$. Hence $Y'\subset\ren\bs\ov{U_\ell}$. In fact, the compactness of $Y'$ yields $\dist(Y',\ov{U_\ell})>0$. Finally, assumption~(\ref{A:detBounds}) allows us to conclude that there exists a $0<C<\infty$ such that
\[
  |Y'|=\lint{\wh{\vy}_{\ve{i}}^\alpha(X'_0,\vzet)}\dd\vy'
  =\lint{X'_0}|\det\partial_{\vx}\wh{\vy}_{\ve{i}}^\alpha(\vx,\vzet)|\dd\vx
  =C|X'_0|>0.
\]
This shows that $\vzet\notin\Phi'_{\ve{i},\ell}$. Since $\vzet\in\bll_\delta(\vzet_0)\subseteq\Phi_j\bs\Phi'_{\ve{i},\ell}$ was arbitrary, we have verified that $\Phi_j\bs\Phi'_{\ve{i},\ell}$ is open, and thus $\Phi'_{\ve{i},\ell}$ is relatively closed in $\Phi_j$. This establishes $\Phi'_{\alpha,i,j}\in\class{B}(\ren)$.

We conclude the proof with a simple induction argument. Clearly
\[
  \Phi'_1=\bigcap_{i\in I}\bigcup_{j\in J}\Phi_{1,i,j}\in\class{B}(\ren).
\]
Let $\alpha\in\nats$ be given, and suppose that $\Phi'_{\alpha'}\in\class{B}(\ren)$, for each $1\le\alpha'\le\alpha$. Then
\[
  \Phi'_{\alpha+1}=\bigcap_{i\in I}\bigcup_{j\in J}\Phi_{\alpha+1,i,j}
    \bs\bigcup_{\alpha'=1}^{\alpha}\Phi'_{\alpha'}\in\class{B}(\ren).
\]
\end{proof}

The next two lemmas are inspired by some of the arguments in~\cite{TiaDu:17a}. The first is a generalization of the argument presented in Example~\ref{Ex:BasicEx2}. The next one is provides the basis for applying Poincar\'{e} Inequality II when the zero-Dirichlet constraint is imposed on a lower dimensional manifold.
\begin{lemma}\label{L:ControlLemmaBase}
Let $\gamma\in L^+(E)$, $\rho\in L(E\times\ren)$, $u\in L(\ren)$, and $\Psi:E\twoheadrightarrow\ren$ be given. Assume the following:
\begin{enumerate}[(i)]
\item\label{A:PsiMeas} $\Psi$ and $\Psi^{-1}$ are measurable-valued and $[\vx\mapsto|\Psi(\vx)|]\in L(E;(0,\infty])$.
\item\label{A:MeasCorrRange} The range $F:=\bigcup_{\vx\in E}\Psi(\vx)$ is measurable.
\item\label{A:ugammaInt} $|u|^p\gamma\in L^1(E)$, for some $1\le p<\infty$;
\item\label{A:uZero} $u(\vx)=0$, for a.e. $\vx\in F\bs E$.
\item\label{A:RNonzero} For each $\vx\in E$, we find $\rho(\vx,\cdot)\in L^\frac{p}{p-1}(\ren)$ and
\[
  \lint{\Psi(\vx)}\rho(\vx,\vy-\vx)\dd\vy=1.
\]
  Define $R\in L(E;(0,\infty))$ by
\[
  R(\vx):=\lc\begin{array}{ll}
    \ds{\lp\lint{\Psi(\vx)}
    |\rho(\vx,\vy-\vx)|^\frac{p}{p-1}\dd\vy\rp^{p-1}},
    & 1<p<\infty,\\
    \ds{\esssup_{\vy\in\Psi(\vx)}|\rho(\vx,\vy-\vx)|},
    & p=1.
  \end{array}\right.
\]
\item There exists a $0<\nu<1$ such that for each $\vy\in E$,
\begin{equation}\label{A:gammaBound}
  \lint{\Psi^{-1}(\vy)}R(\vx)\gamma(\vx)\dd\vx\le\nu \gamma(\vy).
\end{equation}
\end{enumerate}
Put
\begin{equation}\label{D:ControlLemmaBaseConst}
  C(\nu,p)
  :=
  \lc\begin{array}{ll}
    \frac{1}{1-\nu}, & p=1\\
    \min\lc\frac{1}{\lambda^{p-1}-\nu}\lp\frac{\lambda}{1-\lambda}\rp^{p-1}
      :\nu^{\frac{1}{p-1}}<\lambda<1\rc,
    & 1<p<\infty.
  \end{array}\right.
\end{equation}
Then
\begin{equation}\label{E:ControlLemmaBaseIneq2}
  \lint{E}|u(\vx)|^p\gamma(\vx)\dd\vx
  \le C(\nu,p)
  \lint{E}\left|\lint{\Psi(\vx)}
    [u(\vy)-u(\vx)]\rho(\vx,\vy-\vx)\dd\vy\right|^p
    \gamma(\vx)\dd\vx,
\end{equation}
provided the upper bound is well-defined in $[0,\infty]$. If, in particular, we have $0<|\Psi(\vx)|<\infty$, $\rho(\vx,\vz)=|\Psi(\vx)|^{-1}\chi_{\Psi(\vx)}(\vz)$, for all $(\vx,\vz)\in E\times\ren$, then
\begin{equation}\label{E:ControlLemmaBaseIneq1}
  \lint{E}|u(\vx)|^p\gamma(\vx)\dd\vx
  \le C(\nu,p)
  \lint{E}\flint{\Psi(\vx)}|u(\vy)-u(\vx)|^p\gamma(\vx)\dd\vy\dd\vx.
\end{equation}
\end{lemma}

\begin{remark}\label{R:OptControlConst}
\begin{enumerate}[(a)]
\item For~\eqref{E:ControlLemmaBaseIneq2}, it is not necessary for $\vx\mapsto|\Psi(\vx)|$ to be finitely valued.
\item A straightforward computation shows that $C(\nu,2)=\lp\frac{1}{1-\sqrt{\nu}}\rp^2$.
\end{enumerate}
\end{remark}

\begin{proof}
Inequality~\eqref{E:ControlLemmaBaseIneq1} is trivial, if the upper bound is infinite. We focus on establishing~\eqref{E:ControlLemmaBaseIneq2}, since~\eqref{E:ControlLemmaBaseIneq1} follows from~\eqref{E:ControlLemmaBaseIneq2} and Jensen's inequality. Assume the upper bound in~\eqref{E:ControlLemmaBaseIneq2} is well-defined and finite. We may assume
\[
  \lint{E}\left|\lint{\Psi(\vx)}
    [u(\vy)-u(\vx)]\rho(\vx,\vy-\vx)\dd\vy\right|^p
      \gamma(\vx)
    \dd\vx<\infty.
\]
Otherwise the result is immediate.

The convexity of $\tau\mapsto|\tau|^p$ implies that, for any $0<\lambda<1$ and $a,b\in\re$,
\[
  |a|^p\le(1-\lambda)^{1-p}|b-a|^p+\lambda^{1-p}|b|^p.
\]
Using this inequality, we may write
\begin{align}\label{E:IteratedIneq1}
\nonumber
  \lint{E}|u(\vx)|^p\gamma(\vx)\dd\vx
  &=
  \lint{E}\left|\lint{\Psi(\vx)}u(\vx)
    \rho(\vx,\vy-\vx)\dd\vy\right|^p
    \gamma(\vx)\dd\vx\\
\nonumber
  &\le
  (1-\lambda)^{1-p}\lint{E}\left|\lint{\Psi(\vx)}[u(\vy)-u(\vx)]
    \rho(\vx,\vy-\vx)\dd\vy\right|^p
    \gamma(\vx)\dd\vx\\
\nonumber
  &\qqquad+\lambda^{1-p}\lint{E}\left|\lint{\Psi(\vx)}
    u(\vy)\rho(\vx,\vy-\vx)\dd\vy\right|^p
    \gamma(\vx)\dd\vx\\
  &\le
  (1-\lambda)^{1-p}\lint{E}\left|\lint{\Psi(\vx)}
    [u(\vy)-u(\vx)]\rho(\vx,\vy-\vx)\dd\vy\right|^p
    \gamma(\vx)\dd\vx\\
\nonumber
  &\qqquad+\lambda^{1-p}\lint{E}
    \lint{\Psi(\vx)}R(\vx)\gamma(\vx)|u(\vy)|^p\dd\vy\dd\vx.
\end{align}
In the last integral, we applied H\"{o}lder's inequality and the definition of $R$ in~\eqref{A:RNonzero}. Focusing on the last iterated integral, we see that
\begin{align}
\nonumber
  \lint{E}\lint{\Psi(\vx)}
    R(\vx)\gamma(\vx)|u(\vx)|^p\dd\vy\dd\vx
  &=
  \lint{E}\lint{F}
    R(\vx)\gamma(\vx)\chi_{\Psi(\vx)}(\vy)
    |u(\vy)|^p\dd\vy\dd\vx\\
\nonumber
  &\qqquad=
  \lint{F}\lint{E}R(\vx)\gamma(\vx)\chi_{\Psi^{-1}(\vy)}(\vx)
    |u(\vy)|^p\dd\vx\dd\vy\\
\nonumber
  &\qqquad=
  \lint{E\cap F}\lp\lint{\Psi^{-1}(\vy)}R(\vx)\gamma(\vx)\dd\vx\rp
    |u(\vy)|^p\dd\vy\\
\label{E:ControlBaseIneq1}
  &\qqquad\le
  \nu\lint{E}\gamma(\vy)|u(\vy)|^p\dd\vy,
\end{align}
where assumption~(\ref{A:MeasCorrRange}) allowed us to use the Fubini-Tonelli theorem for the second equality. From assumption~(\ref{A:uZero}), we deduced the third. Assumption~(\ref{A:ugammaInt}) was used to obtain the final inequality. Incorporating the above bound into the inequality in~\eqref{E:IteratedIneq1} yields
\begin{multline*}
  \lint{E}|u(\vx)|^p\gamma(\vx)\dd\vx
  \le
  (1-\lambda)^{1-p}\lint{E}\left|\flint{\;\Psi(\vx)}
    [u(\vy)-u(\vx)]\rho(\vx,\vy-\vx)
    \dd\vy\right|^p\gamma(\vx)\dd\vx\\
  +\lambda^{1-p}\nu\lint{E}|u(\vx)|^p\gamma(\vx)\dd\vx.
\end{multline*}
Absorbing the rightmost integral into the lower bound yields the result.
\end{proof}

The following provides a class of correspondences $\Psi$ and functions $\gamma$ for which the previous lemma is applicable when $\rho\equiv|\Psi|^{-1}$. The conditions allow the set $G$ to be a part of the boundary of $E$ or some subset of $\ren$ with co-dimension larger than one.

\begin{lemma}\label{L:DistClass}
Let $G\subset\ren$ be a compact set and $E\subseteq\ren\bs G$ be a bounded measurable set. Suppose that $\gamma\in L^+(E)$ and that $\Psi:E\twoheadrightarrow E$ satisfies hypothesis~(\ref{A:PsiMeas}) in Lemma~\ref{L:ControlLemmaBase}. Let $0<a<b\le1$ be given, and define $I:E\twoheadrightarrow(0,\infty)$ by $I(\vx):=\lp\frac{1}{b}\cdot d_G(\vx),\frac{1}{a}\cdot d_G(\vx)\rp$. Assume that there exists $K_1,K_2,K_3<\infty$, $1\le\alpha_1\le n$, $1\le\alpha_2\le n-\alpha_1$, and $\beta\ge\alpha_2$ such that following hold.
\begin{enumerate}[(i)]
  \item\label{A:PsiInvSuperset} $\Psi^{-1}(\vx)\subseteq E\cap\ann_{I(\vx)}(G)$, for all $\vx\in E$.
  \item\label{A:PsiMeasBnd} There exists $1\le\alpha\le n$, $1\le\alpha_2\le n-\alpha_1$, and $K_1,K_2<\infty$ such that, for $\vx,\vy\in E$ and a.e. $\tau\in I(\vx)$,
\[
  \ds{|\Psi(\vx)|\ge\frac{b^{\alpha_1}-a^{\alpha_1}}{K_1}d_G(\vx)^{\alpha_1}}
\]
  and
\begin{equation}\label{E:TubNbhdFormula}
  \Haus^{n-1}\lp\sph_\tau(G)\cap\Psi^{-1}(\vy)\rp\le K_2 d_G(\vy)^{\alpha_1-\alpha_2}\tau^{\alpha_2-1}.
\end{equation}
  \item\label{A:gammaDistBnd} There exists $K_3<\infty$ such that $1\le d_G(\vy)^\beta \gamma(\vy)\le K_3$, for all $\vy\in E$.
\end{enumerate}
Then
\[
  \lint{\Psi^{-1}(\vy)}\frac{\gamma(\vx)}{|\Psi(\vx)|}\dd\vy
  \le
  K_1K_2K_3\lp\frac{b^{\beta-\alpha_2}}{\alpha_1}\rp \gamma(\vy).
\]
\end{lemma}
\begin{remark}
\begin{enumerate}[(a)]
\item The set $G$ need not be compact, so long as $d_G$ is a Lipschitz map so that the co-area formula is applicable.
\item Clearly it is irrelevant whether $I(\vx)$ includes its end points.
\item $\sph_\tau(G)$ is the $\tau$-boundary of $G$ with radius $\tau$. If $G$ is an $m$-dimensional manifold and $\alpha_2=n-m$, then the upper bound in~\eqref{E:TubNbhdFormula} resembles the derivative of Weyl's tube formula with respect to the radius $\tau$~\cite{Gra:04a}.
\end{enumerate}
\end{remark}
\begin{proof}
Let $\vy\in E$ be given. By assumptions~(\ref{A:PsiInvSuperset}) and~(\ref{A:PsiMeasBnd}), for each $\vx\in\Psi^{-1}(\vy)$, we have $d_G(\vx)>\frac{1}{b}d_G(\vy)$, so
\[
  0
  \le
  \frac{\gamma(\vx)}{|\Psi(\vx)|}
  \le
  \frac{K_1K_3}{d_G(\vx)^{\beta+\alpha_1}\lp b^{\alpha_1}-a^{\alpha_1}\rp}
  \le
  \frac{K_1K_3 b^{\beta+\alpha_1}}{d_G(\vy)^{\beta+\alpha_1}
    \lp b^{\alpha_1}-a^{\alpha_1}\rp}.
\]
It follows that $\ls\vx\mapsto\frac{\gamma(\vx)}{|\Psi(\vx)|}\rs\in L^1(\Psi^{-1}(\vy))$, for each $\vy\in E$. In fact, by the co-area formula in Corollary~\ref{C:DistCoarea}, we find
\begin{align*}
  \lint{\Psi^{-1}(\vy)}\frac{\gamma(\vx)}{|\Psi(\vx)|}\dd\vy
  &\le
  \frac{K_1K_3}{b^{\alpha_1}-a^{\alpha_1}}
  \lint{I(\vy)}\lp\lint{\sph_\tau(G)\cap\Psi^{-1}(\vy)}
    \tau^{-\alpha_1-\beta}\dd\Haus^{n-1}(\vw)\rp\dd\tau\\
  &\le
  \frac{K_1K_2K_3 d_G(\vy)^{\alpha_1-\alpha_2}}{b^{\alpha_1}-a^{\alpha_1}}
    \luint{\frac{1}{b}d_G(\vy)}{\frac{1}{a}d_G(\vy)}
      \tau^{\alpha_2-\alpha_1-\beta-1}\dd\tau.
\end{align*}
Since $\alpha_2-\beta\le 0$ and $\alpha_1>0$,
\begin{align*}
  0\le
  \luint{\frac{1}{b}d_G(\vy)}{\frac{1}{a}d_G(\vy)}
      \tau^{\alpha_2-\alpha_1-\beta-1}\dd\tau
  &\le
  \lp\frac{1}{b}d_G(\vy)\rp^{\alpha_2-\beta}
    \luint{\frac{1}{b}d_G(\vy)}{\frac{1}{a}d_G(\vy)}
      \tau^{-\alpha_1-1}\dd\tau\\
  &\le
  \lp\frac{b}{d_G(\vy)}\rp^{\beta-\alpha_2}
  \lp\frac{1}{\alpha_1}\rp
  \ls\lp\frac{b}{d_G(\vy)}\rp^{\alpha_1}
    -\lp\frac{a}{d_G(\vy)}\rp^{\alpha_1}\rs\\
  &=
  \lp\frac{b^{\beta-\alpha_2}}{\alpha_1}\rp\lp\frac{b^{\alpha_1}-a^{\alpha_1}}
    {d_G(\vy)^{\beta+\alpha_1-\alpha_2}}\rp.
\end{align*}
Hence, from assumption~(\ref{A:gammaDistBnd}),
\[
  \lint{\Psi^{-1}(\vy)}\frac{\gamma(\vx)}{|\Psi(\vx)|}\dd\vx
  \le
  K_1K_2K_3\lp\frac{b^{\beta-\alpha_2}}{\alpha_1}\rp\frac{1}{d_G(\vy)^\beta}
  \le
  K_1K_2K_3\lp\frac{b^{\beta-\alpha_2}}{\alpha_1}\rp \gamma(\vy).
\]
\end{proof}

The following corollary provides the basis for examples of Poincar\'{e} inequalities with zero-Dirichlet type conditions on flat sets with integer co-dimension equal to one or more. For each $\vx$, the set $\Psi(\vx)$ will be a portion of a cone with a vertex in $\ov{E}\bs E\subseteq G:=\mathbb{H}^m_{(m+1,\dots,n)}$ and contained in a ball with radius proportional to $d_G(\vx)$. Thus, the measure $|\Psi(\vx)|$ decreases as $\vx$ approaches $G$. For convenience, we define $G^\perp:=\mathbb{H}^{n-m}_{(1,\dots,m)}$ and for the volume of a cone contained in an annular region around $G$, we use
\begin{equation}\label{D:ConeMeas}
  c_{m,\theta}:=|\ann_1(G)\cap\set{C}_\theta(\ve{0};\ve{e}_n)|
\end{equation}

\begin{corollary}\label{C:LowDimWts}
Let $R<\infty$, $0<\theta<\frac{\pi}{2}$, $0\le m\le n-1$, $0<b<\frac{1}{\sqrt{n-m}}$, and $\beta>n-m$ be given. Set $E:=(-R,R)^n\bs G$. Define $\Psi:E\twoheadrightarrow E$ by
\begin{equation}\label{D:FlatLowDimPsi}
  \Psi(\vx):=E\cap\ann_{b\cdot d_G(\vx)}(G)
    \cap\set{C}_\theta(\vx';\vom(\vx'')),
\end{equation}
where $\vx':=\vP_{G}(\vx)$ and $\vx'':=\vx-\vx'=\vP_{G^\perp}(\vx)$. Suppose that $\gamma:E\to[0,\infty)$ is measurable and satisfies $1\le d_G^\beta\gamma\le K_3$, for some $K_3<\infty$, then
\begin{equation}\label{E:DistClassConeIneq}
  \lint{\Psi^{-1}(\vy)}\frac{\gamma(\vx)}{|\Psi(\vx)|}\dd\vx
  \le
  \lp\frac{K_3K_4}{n}\rp b^{\beta-(n-m)}\gamma(\vy),
  \frl\vy\in E,
\end{equation}
with
\begin{equation}\label{D:K4}
  K_4=K_4(m,n,\theta)
  :=\lp\frac{2^m\tan^m\theta}{c_{m,\theta}}\rp
  \haus^m(\bll^m_1)\haus^{n-m-1}(\sph^{n-m}_1)
\end{equation}
\end{corollary}
\begin{remark}\label{R:LowDimWts}
\begin{enumerate}[(a)]
\item\label{R:HyperCaps} For any $0\le m\le n-1$, the set $\ann_1(G)\cap\set{C}_\theta$ contains an $n$-dimensional cone with height $\cos\theta$ over the base $\bll^{n-1}_{\sin\theta}$. Thus
\[
  c_{m,\theta}
  >
  c'_{n,\theta}:=\frac{\sin^{n-1}\theta\cos\theta}{n+1}\haus^{n-1}(\bll^{n-1}_1)\
  =\frac{\tan^{n-1}\theta\cos^n\theta}{n+1}\haus^{n-1}(\bll_1^{n-1}).
\]
  A precise formula for $c_{n-1,\theta}$ can be obtained using the volume of hyperspherical caps found in~\cite{Li:11a} along with the volume of the cone.
\item A similar bound can be obtained using cones with vertices at $\vx$ that open towards $G$. If $\Psi:E\twoheadrightarrow E$ is given by
\[
  \Psi(\vx):=E\cap\ann_{b\cdot d_G(\vx)}(G)
    \cap\set{C}_\frac{\pi}{2}(\vx';\vom(\vx'')
    \cap\set{C}_\theta(\vx;-\vom(\vx'')),
\]
  then inequality~\eqref{E:DistClassConeIneq} holds, provided $0<b<\frac{1}{2}$.
\item The set $E$ need not be a cube and may include some, or all, points in $[\partial(-R,R)^n]\bs G$. The key requirements are that there exists $1\le\eta<\infty$ such that
\[
  \eta|\Psi(\vx)|\ge|\ann_{b\cdot d_G(\vx)}(G)\cap\set{C}_\theta|
\]
  and $d_G(\vx)=\|\vx''\|_{\ren}$, for all $\vx\in E$. In the formula for $K_4$, the factor $2^m$ becomes $\eta$. For example, we may take
\[
  E:=\lc\vx\in(-R,R)^n: x_n>0\rc,
\]
  which is useful when imposing zero-Dirichlet conditions on subsets of $\partial\dom$. In this case, we find $\eta=2^{m+1}$.
\end{enumerate}
\end{remark}
\begin{proof}
Clearly $d_G(\vx)=\|\vx''\|_{\ren}$, for each $\vx\in E$. We note that although $G$ is not compact, the distance function $d_G$ is Lipschitz on $E$. For each $r>0$, define $\set{C}'_r:E\twoheadrightarrow\ren$ by
\[
  \set{C}'_r(\vx)
  :=\ann_{r\cdot d_G(\vx)}(G)\cap
    \set{C}_\theta(\vx';\vom(\vx'')).
\]
Thus $\Psi=E\cap\set{C}'_r$.  For each $0<a<b$, define $\Psi_{a,b}:\twoheadrightarrow E$ by
\[
  \Psi_{a,b}:=(E\cap\set{C}'_b)\bs(E\cap\set{C}'_a)=\Psi\bs(E\cap\set{C}'_a).
\]
We will verify the assumptions in Lemma~\ref{L:DistClass} for $\Psi_{a,b}$, and then take the limit as $a\to0^+$.

We first establish the first part of assumption~(\ref{A:PsiMeasBnd}). Since  $0<b<\frac{1}{\sqrt{n-m}}$, we find $b\cdot d_G(\vx)<R$ and, for all $0<r\le b$ and $\vx\in E$,
\[
  |E\cap\set{C}'_r(\vx)|=r^n|E\cap\set{C}'_b(\vx)|
  \nd
  |E\cap\set{C}'_b(\vx)|
  \ge b^n\cdot d_G(\vx)^n\lp\frac{c_{m,\theta}}{2^m}\rp.
\]
Thus
\[
  |\Psi_{a,b}(\vx)|
  =
  |E\cap\set{C}'_b(\vx)|-|E\cap\set{C}'_a(\vx)|
  \ge
  \lp\frac{c_{m,\theta}}{2^m}\rp\lp b^n-a^n\rp d_G(\vx)^n.
\]
Hence the first part of assumption~(\ref{A:PsiMeasBnd}) is satisfied with $\alpha_1=n$ and $K_1=2^m/c_{m,\theta}$.

For the second part of assumption~(\ref{A:PsiMeasBnd}) in Lemma~\ref{L:DistClass}, we note that $d_G(\vx)=\|\vx''\|_{\ren}$ and $\vx'\bdot\vx''=0$, for each $\vx\in E$. Let $0<y_n<R$ be given, and without loss of generality, we assume $y_n\ve{e}_n\in E$ and consider $\Psi_{a,b}^{-1}(y_n\ve{e}_n)$. Suppose that $\vx\in\Psi^{-1}_{a,b}(y_n\ve{e}_n)$. Then, by definition, we find $y_n\ve{e}_n\in\Psi_{a,b}(\vx)$. We conclude that
\begin{equation}\label{E:PsiInvCond1}
  a^2\|\vx''\|_{\ren}^2\le y_n^2<b^2\|\vx''\|^2_{\ren}
  \nd
  \lp\frac{\vx''\bdot y_n\ve{e}_n}{\|\vx''\|_{\ren}}\rp^2
  >\cos^2\theta\cdot\|y_n\ve{e}_n-\vx'\|^2.
\end{equation}
Let $\vx\in\sph_\tau(G)$, with $y_n/b<\tau<y_n/a$, so $\|\vx''\|_{\ren}=\tau$. The last inequality in~\eqref{E:PsiInvCond1} may be rewritten as
\[
  \lp y_n^2+x_1^2+\cdots x_m^2\rp
  <
  \lp\frac{1}{\tau\cos\theta}\rp^2x_n^2\cdot y_n^2,
\]
which implies
\[
  \|\vx'\|_{\ren}^2<y_n^2\lp\frac{x^2_n}{\tau^2\cos^2\theta}-1\rp.
\]
Since $\|\vx''\|=\tau$, we must have $|x_n|\le\tau$. As $d_G(y_n\ve{e}_n)=y_n$, it follows that
\[
  \|\vx'\|_{\ren}<d_G(y_n\ve{e}_n)\tan\theta.
\]
We conclude that $\vx\in\bll^m_s(\ve{0})\times\sph^{n-m}_\tau(\wh{\vx}'')$, where $s:=d_G(y_n\ve{e}_n)\tan\theta$ and $\wh{\vx}''=(x_{m+1},\dots,x_n)\in\re^{n-m}$. Consequently,
\begin{equation}\label{E:PsiInvMeasEst}
  \haus^{n-1}(\sph_\tau(G)\cap\Psi^{-1}_{a,b}(y_n\ve{e}_n))
  \le
  K_2d_G(y_n\ve{e}_n)^m\tau^{n-m-1},
\end{equation}
with
\[
  K_2=\lp\smfrac{1}{\theta^2}-1\rp^\frac{m}{2}\haus^m(\bll_1)\haus^{n-m-1}(\sph_1).
\]
Clearly the inequality in~\eqref{E:PsiInvMeasEst} continues to hold for general $\vy\in E$. This verifies the second part of assumption~(\ref{A:PsiMeasBnd}) with $K_2$ identified above and $\alpha_2=n-m$.

Lemma~\ref{L:DistClass} yields
\[
  \lint{\Psi_{a,b}^{-1}(\vy)}\frac{\gamma(\vx)}{|\Psi_{a,b}(\vx)|}\dd\vy
  \le
  \frac{2^m}{c_{m,\theta}}K_2K_3
  \lp\frac{b^{\beta-(n-m)}}{n}\rp\gamma(\vy)
\]
We observe that $|\Psi_{a,b}(\vx)|\le|\Psi(\vx)|$, for each $\vx\in E$ and $0<a<b$, so $\gamma(\vx)/|\Psi(\vx)|\le\gamma(\vx)/|\Psi_{a,b}(\vx)|$. Furthermore, we see that $\Psi^{-1}_{a,b}(\vy)=\Psi^{-1}(\vy)$, for any $\vy\in E$ and $0<a<1/(\diam(E)\cdot d_G(\vy))$. Thus,
\[
  \lint{\Psi^{-1}(\vy)}\frac{\gamma(\vy)}{|\Psi(\vx)|}\dd\vx
  \le
  \lint{\Psi_{a,b}^{-1}(\vy)}\frac{\gamma(\vx)}{|\Psi_{a,b}(\vx)|}\dd\vx
  \le
  \frac{2^m}{c_{m,\theta}}K_2K_3
  \lp\frac{b^{\beta-(n-m)}}{n}\rp\gamma(\vy).
\]

\end{proof}

From the last remark we deduce a version of Corollary~\ref{C:LowDimWts}, where $G$ corresponds to a half-space/line with a boundary of dimension $m$ imbedded in an $n$-dimensional box.
\begin{corollary}\label{C:LowDimWtsBndry}
Let $R<\infty$, $0<\theta<\frac{\pi}{2}$, $1\le m\le n-1$, $0<b<\frac{1}{\sqrt{n-m+1}}$, $\beta_1>n-m$, and $\beta_2>n-m+1$ be given. Set $G_1:=\mathbb{H}^m_{(m+1,\dots,m)}$, $G_2:=\mathbb{H}^{m-1}_{(1,0,\dots,0,m+1,\dots,n)}$,
\[
  E_1:=(-2R,0)\times(-R,R)^{n-1}\bs G_1
  \nd
  E_2:=[0,R)\times(-R,R)^{n-1}\bs G_2
\]
Define $E:=E_1\cup E_2$, $G:=G_1\cup G_2$, and $\Psi:E\twoheadrightarrow E$ by
\[
  \Psi(\vx):=
  \lc\begin{array}{ll}
    E_1\cap\ann_{b\cdot d_{G_1}(\vx)}(G_1)
      \cap\set{C}_\theta(\vx';\vom(\vx'')),
      & \vx\in E_1\\
    E_2\cap\ann_{b\cdot d_{G_2}(\vx)}(G_2)
      \cap\set{C}_\theta(\vx';\vom(\vx'')),
      & \vx\in E_2,
  \end{array}\right.
\]
where $\vx':=\vP_G(\vx)$, $\vx'':=\vx-\vx'$. Suppose that $\gamma:E\to[0,\infty)$ is measurable and, for $\ell=1,2$, satisfies $1\le d_G^{\beta_\ell}\gamma\le K_{3,\ell}$ on $E_\ell$, for some $K_{3,\ell}$, then
\[
  \lint{\Psi^{-1}(\vy)}\frac{\gamma(\vx)}{|\Psi(\vx)|}\dd\vx
  \le
  \lc\begin{array}{ll}
    \lp\frac{K_{3,1}K_{4,1}}{n}\rp b^{\beta_1-(n-m)}\gamma(\vy),
      & \vy\in E_1,\\
    \lp\frac{K_{3,2}K_{4,2}}{n}\rp b^{\beta_2-(n-m+1)}\gamma(\vy),
      & \vy\in E_2,
  \end{array}\right.
  \frl\vy\in E
\]
with $K_{4,1}:=K_4(m,n,\theta)$ and $K_{4,2}:=K_4(m+1,n,\theta)$.
\end{corollary}
\begin{proof}
The dimension of $G\cap\ov{E}_1$ is $m$ and the dimension of $G\cap\ov{E}_2$ is $m-1$. The set $G_2\cap E$ can be identified as the boundary of $G$ within $E$. After translation Corollary~\ref{C:LowDimWts} is immediately applicable in the region $E_1$. In view of Remark~\ref{R:LowDimWts}, Corollary~\ref{C:LowDimWts} is also applicable in $E_2$ after elabeling coordinates.
\end{proof}

Interestingly, the exponent for the parameter $b$ in~\eqref{E:DistClassConeIneq} is sensitive to the shape of $\Psi$ as well as how it scales as $\vx$ approaches $G$. For example, if instead of conical regions, one used the cap-like subsets of the balls $\bll_{d_G(\vx)}(\vx)$ contained within tubular neighborhoods of $G$, we may obtain the following:
\begin{corollary}\label{C:GenAttractorExample}
Let $0<b<1$, $0\le m\le n-1$, and $\beta>n-\frac{m}{2}$ be given. Set $G:=\mathbb{H}^m_{(m+1,\dots,n)}$ and $E:=(-1,1)^m\times\bll_1^{n-m}(\ve{0})\bs G$. Define $\Psi:E\twoheadrightarrow E$ by
\[
  \Psi(\vx):=E\cap\bll_{d_G(\vx)}(\vx)\cap\ann_{b\cdot d_G(\vx)}(G).
\]
If $\gamma:E\to[0,\infty)$ is measurable and satisfies $1\le d_G^\beta\gamma\le K_3$, for some $K_3<\infty$, then there exists a constant $K_4'=K_4'(n-m)$ such that
\[
  \lint{\Psi^{-1}(\vx)}\frac{\gamma(\vy)}{|\Psi(\vy)|}\dd\vy
  \le n2^\frac{3m}{2}K_3 K_4'
  \lp\frac{b^{\beta+\frac{m}{2}-n}}{(1-b)^{n-m-1}}\rp \gamma(\vy),
  \frl\vy\in E.
\]
\end{corollary}
The argument is similar to the one used for Corollary~\ref{C:LowDimWts}. A key difference is that $\Psi^{-1}(\vy)$ can be identified with a slice of a paraboloid. As $\vy$ approaches $G$, the projection of $\Psi^{-1}(\vy)$ onto directions orthogonal to $G$ scales differently from the projection onto directions tangential to $G$.

The next corollary used to establish Poincar\'{e} inequalities when the zero-Dirichlet conditions $m$-dimensional manifolds that can be identified with a transformation of $G:=\mathbb{H}^m_{(m+1,\dots,n)}$. We use the constant $K_4=K_4(m,n,\frac{\pi}{4})$ defined in~\eqref{D:K4}. We will use $\re^{n\otimes3}:=\ren\otimes\ren\otimes\ren$ to denote the space of real $\suprd{3}$-order tensors, with the usual Eucliden norm.
\begin{corollary}\label{C:LowDimSmoothMan}
Let $\bnd_0\subset\ren$, $\beta>n$, $0<\nu<1$, and $\gamma_0\in L^+(\ren\bs\bnd_0)$ be given. Assume the following
\begin{enumerate}[(i)]
\item\label{A:ContDiffg} There is $1\le K_0<\infty$, an open bounded set $U_0$, and a bijective map $\vg\in\cnt^2((-1,1)^n;U_0)$ such that
\begin{itemize}
\item[$\bullet$] $\vg((-1,1)^n\cap G)=U_0\cap\bnd_0$;
\item[$\bullet$] $\vg(\vP_G(\vv))=\vP_{\bnd_0}(\vg(\vv))$, for all $(-1,1)^n$.
\item[$\bullet$] $\vg^{-1}\in\cnt^2(U_0;(-1,1)^n)$;
\item[$\bullet$] $\|\partial_{\vv}\vg\|_{\re^{n\times n}}\le K_0$, $\|\partial_{\vx}\vg^{-1}\|_{\re^{n\times n}}\le K_0$, and $\|\partial^2_{\vv}\vg\|_{\re^{n\otimes 3}}\le\frac{1}{2}K_0^{-1}$.
\end{itemize}
  \item\label{A:gammaGrowthCond} There exists $K_3<\infty$ such that $1\le d_{\bnd_0}^\beta(\vx)\gamma_0(\vx)\le K_3$ for all $\vx\in U_0\bs\bnd_0$.
\end{enumerate}
Then there exists $0<b_0<\frac{1}{\sqrt{n-m}}$ and $0<\theta_0<\frac{\pi}{4}$ such that for all $0<b\le b_0$ and $0<\theta\le\theta_0$,
\begin{equation}\label{E:PsiContainment}
  |\Psi_0(\vx)|^{-1}\le n^\frac{n}{2}K_0^n|\Psi(\vg^{-1}(\vx))|^{-1},\quad
  \Psi_0(\vx)\subseteq\bll_{d_{\bnd_0}(\vx)}(\vx),
  \quad\text{for all }\vx\in U_0\bs\bnd_0,
\end{equation}
and
\[
  \lint{\Psi_0^{-1}(\vy)}\frac{\gamma_0(\vx)}{|\Psi_0(\vx)|}\dd\vx
  \le
  \nu\gamma_0(\vy),
  \frl\vy\in U_0\bs\bnd_0.
\]
Here $\Psi_0:U_0\bs\bnd_0\twoheadrightarrow U_0\bs\bnd_0$ is defined by
\[
  \Psi_0(\vx):=\vg(\Psi(\vg^{-1}(\vx))),
\]
with $\Psi$ provided by~\eqref{D:FlatLowDimPsi}.
\end{corollary}
\begin{remark}
\begin{enumerate}[(a)]\label{R:SmoothMan}
\item\label{R:Rescale} If $\bnd$ is a smooth, say $\cnt^2$, manifold, then the normal subbundle to $\bnd$ is a smooth subbundle of the tangent space to $\ren$ (see, for example, Corollary 10.36 in J.~M.~Lee's textbook~\cite{Lee:13a}). Thus at each point $\vx_0\in\bnd$, there exists a neighborhood $U$ of $\vx_0$ and a map $\vg_0$ satisfying the first three bullets in~(\ref{A:ContDiffg}). A rescaling of $\vg_0$ can be used to produce $\vg$ that also satisfies the fourth part of~(\ref{A:ContDiffg}). We refer to the proof for Example~\ref{Ex:LowDimCompMan} for additional details.
\item\label{R:PsiMeasBound} From Remark~\ref{R:LowDimWts}(\ref{R:HyperCaps}) and~\eqref{E:compDist}, we find that
\[
  |\Psi(\vg^{-1}(\vx))|
  \ge\frac{b^n\cdot c_{m,\theta}}{2^m}\cdot d_G(\vg^{-1}(\vx))^n
  >\frac{c'_{n,\theta}}
    {2^m}\lp\frac{b\cdot d_{\bnd_0}(\vx)}{K_0}\rp^n.
\]
  From the proof, we see that we may take
\[
  \theta\le\theta_0=\tan^{-1}\lp\frac{1}{4K_0^2}\rp
  \text{ and }
  b\le b_0=\min\lc\frac{1}{\sqrt{n-m}},\frac{\cos\theta_0}{2K_0^4},
    \lp\frac{\nu}{n^{n-1}K_0^{2n-\beta}K_3K_4}\rp^\frac{1}{\beta-(n-m)}\rc.
\]
Selecting $\theta=\theta_0$ and $b=b_0$ and recalling the formula for $c'_{n,\theta_0}$, we obtain the upper bound
\[
  |\Psi_0(\vx)|^{-1}
  \le
  \frac{2^m}{c'_{n,\theta_0}}
  \lp\frac{n^\frac{1}{2}K_0^{2}}{b_0\cdot d_{\bnd_0}(\vx)}\rp^n
  =
  \frac{2^m(n+1)}{4K_0^2\haus^{n-1}(\bll_1^{n-1})}
  \lp\frac{16n^\frac{1}{2}K_0^6}{b_0\cdot d_{\bnd_0}(\vx)
    \sqrt{1+16K_0^4}}\rp^n.
\]
\item We see that $\ov{\bnd}_0$ is compact, since upon continuously extending $\vg$ to $[-1,1]^n$, it is a continuous image of $[-1,1]^n\cap G$. Thus $d_{\bnd_0}=d_{\ov{\bnd}_0}$ is Lipschitz continuous.
\item\label{R:ManifoldWithBound} There is an obvious adaptation of the proof to establish a version of Corollary~\ref{C:LowDimSmoothMan} for manifolds with a boundary $\partial\bnd_0$. Using the notation from Corollary~\ref{C:LowDimWtsBndry}, in assumption~(\ref{A:ContDiffg}), we require $\partial\bnd_0$ to be the image under $\vg$ of $[0,1)\times(-1,1)^{n-1}\cap G_2$.
\item\label{R:LipManifold} In the case where we only assume $\vg$ and $\vg^{-1}$ are Lipschitz with constant $K_0$, we cannot expect the projection operator to commute with $\vg$. Nevertheless the same result can be obtained, but with $\Psi(\vx)\subseteq\bll_{K_0^2\cdot d_{\bnd_0}(\vx)}(\vx)$ for the last part of~\eqref{E:PsiContainment}. It seems likely, however, that curvature measures and a relative Minkowski content, as introduced in~\cite{Win:19a}, can provide a framework in which one can work directly with $\bnd$ without using a transform to ``flatten it''. This may lead to a result, with the containment $\Psi(\vx)\subseteq\bll_{d_{\bnd_0}(\vx)}(\vx)$, for a general class of compact sets.
\end{enumerate}
\end{remark}
\begin{proof}
We will identify the bounds $b_0$ and $\theta_0$ in the course of the proof. Set $E:=(-1,1)^n\bs G$. Recall that $[\partial_{\vv}\vg(\vv)]^{-1}=\partial_{\vx}\vg^{-1}(\vg(\vv))$, for a.e. $\vv\in(-1,1)^n$. Thus assumption~(\ref{A:ContDiffg}) and Hadamard's inequality (see Remark~\ref{R:FlowRem}(\ref{R:Hadamard})) imply, for all $\vx\in U_0\bs\bnd_0$ and all $\vv\in(-1,1)^n$,
\begin{equation}\label{E:gInvDetBnd}
  |\det\partial_{\vv}\vg(\vv)|^{-1}
  =
  |\det\ls\partial_{\vv}\vg(\vv)\rs^{-1}|
  \le
  n^\frac{n}{2}K_0^n
  \nd
 |\det\partial_{\vv}\vg(\vv)|\le n^\frac{n}{2}K_0^n.
\end{equation}
This and the change of variables formula implies $|\Psi_0(\vx)|\ge n^{-\frac{n}{2}}K_0^{-n}|\Psi(\vg^{-1}(\vx))|$ and thus first part of~\eqref{E:PsiContainment}.

We now verify the second part. Our argument involves showing $\Psi$ is contained within a cone. Let $\vx\in U_0\bs\bnd_0$ and $\vy\in\Psi_0(\vx)$ be given. Put $\vv:=\vg^{-1}(\vx)$ and $\vw:=\vg^{-1}(\vy)$, so $\vw\in\Psi(\vv)$. Using translations and rotations, without loss of generality, we may assume that $\vx=x_n\ve{e}_n$, $\vP_{\bnd_0}(\vx)=\ve{0}$, and $\vv=v_n\ve{e}_n$, for some $x_n>0$ and $0<v_n<1$. Thus
\[
  \vP_{\bnd_0}(\vx)=\ve{0},\quad\vP_G(\vv)=\ve{0},
  \quad d_{\bnd_0}(\vx)=x_n,\quad d_G(\vv)=v_n,
  \quad\text{and}\quad
  \Psi(\vv)\subseteq\set{C}_{\theta}(\ve{0};\ve{e}_n).
\]
Moreover assumption~(\ref{A:ContDiffg}) implies
\begin{equation}\label{E:compDist}
  K_0^{-1}v_n\le x_n\le K_0v_n,
  \nd
  K_0^{-1}d_G(\vw)\le d_{\bnd_0}(\vy)\le K_0 d_G(\vw).
\end{equation}
By the mean value theorem and~\eqref{E:compDist}, for all $0<t<v_n$,
\begin{align}
\nonumber
  &\frac{1}{2K_0}
  \ge
  \sup_{s\in(0,v_n)}\der{{}^2}{s^2}\ls
    \vg(s\ve{e}_n)\bdot\ve{e}_n\rs
  \ge
  \frac{1}{v_n}\lp\frac{x_n}{v_n}-\der{}{t}\vg(t\ve{e}_n)\bdot\ve{e}_n\rp\\
\label{E:ddtgBound}
  \Longrightarrow&
  \der{}{t}\ls\vg(t\ve{e}_n)\bdot\ve{e}_n\rs
  \ge\frac{x_n}{v_n}-\frac{v_n}{2K_0}\ge\frac{1}{K_0}.
\end{align}
Put $w_n:=\vw\bdot\ve{e}_n$, $y_n:=\vy\bdot\ve{e}_n$, $\wh{\vw}:=\vw-w_n\ve{e}_n$, and $\wh{\vy}:=\vy-y_n\ve{e}_n$. Since $\vw\in\set{C}_\theta(\ve{0};\ve{e}_n)$,
\begin{equation}\label{E:hatWBound}
  \|\wh{\vw}\|^2_{\ren}=\|\vw-\ve{0}\|_{\ren}^2-w_n^2< w_n^2\tan^2\theta,
\end{equation}
The bound on $\|\partial^2_{\vv}\vg\|_{\re^{n\otimes3}}$ and~\eqref{E:ddtgBound} implies
\[
  \der{}{t}\ls\vg(t\ve{e}_n+\wh{\vw})\bdot\ve{e}_n\rs>\frac{1}{2K_0},
  \frl 0<t<v_n.
\]
Assuming $\tan\theta_0\le\frac{1}{4}K_0^{-2}$, we deduce that
\[
  y_n=\vg(\vw)\bdot\ve{e}_n>\frac{w_n}{2K_0}-K_0\|\wh{\vw}\|_{\ren}
  \ge\frac{w_n}{4K_0}>\frac{\cos\theta}{4K_0}\|\vw\|_{\ren}.
\]
It follows that
\begin{align*}
  \|\vy-x_n\ve{e}_n\|_{\ren}^2
  &=
  x_n^2-2x_ny_n+\|\wh{\vy}\|^2_{\ren}
  <x_n^2-\frac{\cos\theta}{2K_0}\|\vw\|_{\ren}x_n+\|\vy\|^2\\
  &\le
  x_n^2-\frac{\cos\theta}{2K_0}\|\vw\|_{\ren}x_n+K_0^2\|\vw\|^2\\
  &\le
  x_n^2-\|\vw\|x_n\lp\frac{\cos\theta}{2K_0}-bK_0^3\rp.
\end{align*}
For the last inequality, we used $\Psi(v_n\ve{e}_n)\subseteq\ann_{b\cdot v_n}(\ve{0})$. Suppose that $b_0\le\frac{\cos\theta_0}{2K_0^4}$. Then $\vy\in\bll_{x_n}(x_n\ve{e}_n)$, which proves the second part of~\eqref{E:PsiContainment}.

To prove the last part of the corollary. Define $\gamma:=K_0^\beta(\gamma_0\circ\vg)\in L^+((-1,1)^n\bs G)$. Hence, by~\eqref{E:compDist} and assumption~(\ref{A:gammaGrowthCond}),
\[
  1\le d_{\bnd_0}^\beta(\vg(\vv))\gamma_0(\vg(\vv))
  \le d_G(\vv)^\beta\gamma(\vv)
  \le K_0^{2\beta}K_3,
  \frl\vv\in(-1,1)^n\bs G.
\]
Next we observe that, given $\vy\in U_0\bs\bnd_0$,
\[
  \vx\in\Psi_0^{-1}(\vy)
  \Leftrightarrow\vy\in\Psi_0(\vx)
  \Longleftrightarrow
  \vg^{-1}(\vx)\in\Psi^{-1}(\vg^{-1}(\vy)).
\]
Hence $\Psi_0^{-1}(\vy)=\vg(\Psi^{-1}(\vg^{-1}(\vy)))$. By using the change of variables formula and Corollary~\ref{C:LowDimWts}, we may write
\begin{align*}
  \lint{\Psi_0^{-1}(\vy)}\frac{\gamma_0(\vx)}{|\Psi_0(\vx)|}\dd\vx
  &\le
  n^nK_0^{2(n-\beta)}
  \lint{\vg(\Psi^{-1}(\vg^{-1}(\vy)))}
    \frac{\gamma(\vg^{-1}(\vx))}{|\Psi(\vg^{-1}(\vx))|}
    |\det\partial_{\vx}\vg^{-1}(\vx)|\dd\vx\\
  &=
  n^nK_0^{2(n-\beta)}
  \lint{\Psi^{-1}(\vg^{-1}(\vy))}
    \frac{\gamma(\vv)}{|\Psi(\vv)|}\dd\vv\\
  &\le
  n^{n-1}K_0^{2(n-\beta)}K_3K_4 b^{\beta-(n-m)}\gamma(\vg^{-1}(\vy)).
\end{align*}
The result follows by imposing $b_0\le\lp \nu^{-1}n^{n-1}K_0^{2n-\beta}K_3K_4\rp^{\frac{1}{(n-m)-\beta}}$.
\end{proof}

\section{Nonlocal Poincar\'{e} Inequalities}\label{S:MainResults}

In this section, we present the two main nonlocal Poincar\'{e} inequalities. The first version, Theorem~\ref{T:MainPoinc1}, is a generalization of Example~\ref{Ex:BasicEx1}. It relies on Lemma~\ref{L:ControlLemmaBase} being applicable throughout $\dom$. The second version, Theorem~\ref{T:MainPoinc2}, uses Lemma~\ref{L:ControlLemmaBase} as a complement to Lemma~\ref{L:OpenAbsorptionParamSet}. It covers the setting where the support of the kernel throughout a part of $\dom$ can be identified with a dynamical system with an absorption set $U$. Within the set $U$, the support of the kernel is assumed to have properties that allow the application of Lemma~\ref{L:ControlLemmaBase}. The theorem is stated to include the extreme cases where only one or the other of Lemmas~\ref{L:OpenAbsorptionParamSet} and~\ref{L:ControlLemmaBase} is needed. There are several technical components for the hypotheses of the second version. A couple of corollaries with a simplified set of assumptions are provided after the main results. More explicit examples are presented in the next section. We continue to use the notation introduced at the beginning of Section~\ref{SS:NewLemma}.

\begin{theorem}[Nonlocal Poincar\'{e} Inequality I]\label{T:MainPoinc1}
Let a measurable set $\bnd\subseteq\ren$ be given, and let $\gamma\in L^1_{\loc}(\ren\bs\bnd;[1,\infty))$, $\rho\in L(\dom\times\ren)$, and $\Psi:\dom\twoheadrightarrow\ren$ be given. Define $Z:\dom\twoheadrightarrow\ren$ by $Z(\vx):=\Psi(\vx)-\vx$. Assume the following
\begin{itemize}
%\item $\gamma\in L^1(\dom\bs U')$ for every open set $U'\subseteq\ren$ satisfying $\bnd\subseteq U'$
\item[$\bullet$] $\Psi$ and $\Psi^{-1}$ are measurable-valued and $[\vx\mapsto|\Psi(\vx)|]\in L^+(\dom)$;
\item[$\bullet$] $F:=\bigcup_{\vx\in\dom}\Psi(\vx)$ is measurable and $F\subseteq\dom\cup\bnd$;
\item[$\bullet$] For each $\vx\in\dom$, we find $\rho(\vx,\cdot)\in L^\frac{p}{p-1}(\ren)$ and
\[
  \lint{Z(\vx)}\rho(\vx,\vz)\dd\vz=1
\]
  Define $R\in L^+(\dom)$ by
\[
  R(\vx):=\lc\begin{array}{ll}
    \ds{\lp\lint{Z(\vx)}
    |\rho(\vx,\vz)|^\frac{p}{p-1}\dd\vz\rp^{p-1}},
    & 1<p<\infty,\\
    \ds{\esssup_{\vz\in Z(\vx)}|\rho(\vx,\vz)|},
    & p=1;
  \end{array}\right.
\]
\item[$\bullet$] There exists a $0<\nu<1$ such that
\begin{equation}\label{E:WeightedJensen}
  \lint{\Psi^{-1}(\vy)}R(\vx)\gamma(\vx)\dd\vx
    \le\nu\gamma(\vy),
  \frl\vy\in\dom.
\end{equation}
\end{itemize}
If $u\in L^1(\ren)$ satisfies (BC) in Section~\ref{S:Intro}, then
\begin{equation}\label{E:MainPoincIneq2}
  \lint{\dom}|u(\vx)|^p\dd\vx
  \le C\lint{\dom}\left|\lint{Z(\vx)}
      [u(\vx+\vz)-u(\vx)]\rho(\vx,\vz)\dd\vz\right|^p\gamma(\vx)\dd\vx.
\end{equation}
Here $C=C(\nu,p)$ denotes the constant defined in~\eqref{D:ControlLemmaBaseConst} of Lemma~\ref{L:ControlLemmaBase}.
\end{theorem}
\begin{remark}\label{R:MainPoinc1}
\begin{enumerate}[(a)]
\item For each $\vx\in\dom$, the assumptions $u\in L^1(\ren)$, $\Psi(\vx)$ has finite measure, and $\rho(\vx,\cdot)\in L^\frac{p}{p-1}(\ren)$ imply
\[
  \lint{Z(\vx)}[u(\vx+\vz)-u(\vx)]\rho(\vx,\vz)\dd\vz\in[-\infty,\infty],
  \quad\text{for a.e. }\vx\in\dom.
\]
Therefore, the upper bound in~\eqref{E:MainPoincIneq2} always has well-defined value in $[0,\infty]$.
\item Suppose that $\gamma\in L^1(\dom)$ and there exists $\vep>0$ such that $|\Psi(\vx)|\ge\vep$, for all $\vx\in\dom$. Then $u\in L^1(\ren)$ implies~(BC). Indeed
\[
  \lint{\dom}
  \lp\flint{\Psi(\vx)}|u(\vy)|\dd\vy\rp^p \gamma(\vx)\dd\vx
  \le
  \frac{1}{\vep}\|u\|_{L^1(\ren)}^p\|\gamma\|_{L^1(E)}<\infty.
\]
\item If $|\bnd|=0$, then the requirement that $u=0$ a.e. on $\bnd$ is irrelevant, since such a function always belongs to the equivalence class of $u$ in $L^1$.
\end{enumerate}
\end{remark}
\begin{proof}
  Using~\eqref{E:ControlLemmaBaseIneq2} of Lemma~\ref{L:ControlLemmaBase}, the proof is a minor modification of the one used to obtain~\eqref{E:uOnUIneq} in the proof for the following Theorem~\ref{T:MainPoinc2}.
\end{proof}

\begin{theorem}[Nonlocal Poincar\'{e} Inequality II]\label{T:MainPoinc2}
Let a measurable $\dom'\subseteq\dom$, an open $U\subseteq\ren$, and a compact set $\bnd\subseteq\ov{U}$ be given. Define, for each $i\in I$, the sets
\[
  V:=\dom\bs\ov{U},
  \quad
  V_i:=\dom_i\cap V,
  \quad
  V':=\dom'\cap V,
  \nd
  E:=\ov{U}\bs\bnd
\]
and the maps $\vy_{V},\vz_V:\ren\times\Phi\to\ren$ by $\vz_V:=\vz\chi_{V\times\Phi}$ and
\[
  \vz_V(\vx,\vzet):=\vz(\vx,\vzet)\cdot\chi_V(\vx)
  \nd
  \vy_V(\vx,\vzet):=\vx+\vz_V(\vx,\vzet).
\]
Suppose that $U$ essentially absorbs $V'$ a.e. in the parameterized dynamical system $\{\bbD_{\vzet}\}_{\vzet\in\Phi}$ generated by $\{\vy_V(\cdot,\vzet)\}_{\vzet\in\Phi}$. For each $\alpha\in\nats$,
\begin{itemize}
\item[$\bullet$] set
\[
  \Phi'_\alpha:=\{\vzet\in\Phi:\text{the absorption index into $U$ for $V'$ in $\bbD_{\vzet}$ is $\alpha$}\},
\]
\[
  A:=\{\alpha\in\nats:|\Phi'_\alpha|>0\}
  \nd
  A_0:=A\cup\{0\};
\]
\item[$\bullet$] define $Z'_\alpha:\dom\twoheadrightarrow\ren$ and $\{Y'_{k,\alpha}\}_{k\in\whls}$
\[
  Z'_\alpha(\vx):=\vz(\{\vx\}\times\Phi'_\alpha)
  \nd
  Y'_{k,\alpha}:=\vy_V^k(V'\times\Phi'_\alpha).
\]
\end{itemize}
Assume the following.
\begin{enumerate}[(i)]
\item\label{A:yzPoinc}
\begin{itemize}
\item[$\bullet$] $\vz$ is countably Lipschitz on $\{V_i\times\Phi_j\}_{i\in I,j\in J}$.
\item[$\bullet$] $\vzet\mapsto\vz(\vx,\vzet)$ is an injective Lusin map on $\Phi$, for each $\vx\in V$.
\item[$\bullet$] $\vx\mapsto\vy(\vx,\vzet)$ is a Lusin map on $V$ and countably injective on $\{V_i\}_{i\in I}$, for each $\vzet\in\Phi$.
\end{itemize}
\item\label{A:IndicatrixPoinc} There exists $\{N_k\}_{k\in\whls}\subset L^\infty(\ren\times\ren;\whls)$ such that
\[
  N_{\vy^k_{V}(\cdot,\vzet)}(\vx,V')
  \le
  N_k(\vx,\vz(\vx,\vzet)),
  \quad\text{for a.e. }(\vx,\vzet)\in\ren\times\Phi.
\]
  Put $\|N_\alpha\|_{L^\infty L^1(Y'_{\alpha,\alpha}\times Z'_\alpha(\cdot))}
  :=
  \esssup_{\vx\in Y'_{\alpha,\alpha}}\|N_\alpha(\vx,\cdot)\|_{L^1(Z'_\alpha(\vx))}$.
\item\label{A:BoundsPoinc} There exists $M_V<\infty$ and $\{\Lambda_\alpha\}_{\alpha=1}^\infty,\{\Theta_\alpha\}_{\alpha=1}^{\alpha_0}\subset[0,\infty)$ such that for each $\alpha\in\nats$, $k=0,\dots,\alpha-1$, and a.e. $\vzet\in\Phi'_\alpha$, we have
\begin{align*}
  & |\det\partial_{\ve{x}}\vy(\vx,\vzet)|\Theta_\alpha\ge 1
  \text{ and }
  |\det\partial_{\vzet}\vz(\vx,\vzet)|\Lambda_\alpha\ge 1,
  && \text{for a.e. }\vx\in V
\intertext{and}
  & \mu(\vx,\vz(\vx,\vzet))M_V
  \ge
  \alpha^p\Theta_\alpha^k N_k(\vx,\vz(\vx,\vzet))
    \frac{\Lambda_\alpha}{|\Phi|},
  && \text{for a.e. }\vx\in V\cap Y'_{k,\alpha}.
\end{align*}
\item\label{A:gammaPoinc} There exists $\{\nu_\alpha\}_{\alpha\in A_0}\subset(0,1)$, $\{M_{E,\alpha}\}_{\alpha\in A_0}\subset[0,\infty)$, $\{\gamma_\alpha\}_{\alpha\in A_0}\subset L^1_{\loc}(\ren\bs\bnd)$, %measurable functions $\{\gamma_\alpha:E\to[1,\infty)\}_{\alpha\in A_0}$,
    and correspondences $\{\Psi_\alpha:E\twoheadrightarrow\ren\}_{\alpha\in A_0}$ such that for each $\alpha\in A_0$,
\begin{itemize}
\item[$\bullet$] $\Psi_\alpha$ and $\Psi_\alpha^{-1}$ are measurable-valued and $[\vx\mapsto|\Psi_\alpha(\vx)|]\in L^+(E)$;
\item[$\bullet$] $F_\alpha:=\bigcup_{\vx\in E}\Psi_\alpha(\vx)$ is measurable and $F_\alpha\subseteq\ov{U}$
\item[$\bullet$] For each $\vy\in E$,
\[
  \gamma_\alpha(\vy)\ge 1
  \nd
  \lint{\Psi_\alpha^{-1}(\vy)}\frac{\gamma_\alpha(\vx)}
    {|\Psi_\alpha(\vx)|}\dd\vx
    \le\nu_\alpha\gamma_\alpha(\vy).
\]
\item[$\bullet$] for a.e. $\vx\in E$ and a.e. $\vy\in\Psi_\alpha(\vx)$,
\[
  \mu(\vx,\vz)M_{E,0}
  \ge
  \frac{C_0}{|\Psi_0(\vx)|}\gamma_0(\vx),
\]
and for all $\alpha\in A$,
\[
  \mu(\vx,\vz)M_{E,\alpha}
  \ge
  \alpha^p\Theta_\alpha^\alpha\|N_\alpha\|_{L^\infty L^1}
  \frac{\Lambda_\alpha}{|\Phi|}\frac{C_\alpha}{|\Psi_\alpha(\vx)|}
  \gamma_\alpha(\vx).
\]
  Here $C_\alpha=C_\alpha(\nu_\alpha,p)$ denotes the constant defined in~\eqref{D:ControlLemmaBaseConst} of Lemma~\ref{L:ControlLemmaBase}.
\item[$\bullet$] $M_E:=\sup_{\vx\in E}\sup_{\vy\in\Psi_\alpha(\vx)}\sum_{\alpha\in A_0}M_{E,\alpha}\chi_{\Psi_\alpha(\vx)}(\vy)<\infty$.
\end{itemize}
\end{enumerate}
Define $M:=\max\{M_E,M_V\}$, $Y':=V\cap\bigcup_{k=0}^\infty\vy^k_V(V'\times\Phi)$, and $Z:\dom\twoheadrightarrow\ren$ by
\[
  Z(\vx):=\lc\begin{array}{ll}
    \vz(\vx,\Phi), & \vx\in V,\\
    \bigcup_{\alpha=0}^\infty\Psi_\alpha(\vx)-\vx, & \vx\in E.
  \end{array}\right.
\]
If $u\in L^1(\ren)$ satisfies $u(\vx)=0$ for a.e. $\vx\in\bnd$ and
\begin{equation}\label{E:uAssumMainPoinc}
  \lint{E}
  \lp\flint{\Psi_\alpha(\vx)}|u(\vy)|\dd\vy\rp^p \gamma_\alpha(\vx)\dd\vx
  <\infty.
\end{equation}
then
\begin{equation}\label{E:FinalMainPoincIneq}
  \lint{\dom'}|u(\vx)|^p\dd\vx
  \le M
    \lint{Y'\cup E}\lint{Z(\vx)}
      |u(\vx+\vz)-u(\vx)|^p\mu(\vx,\vz)\dd\vz\dd\vx.
\end{equation}
\end{theorem}
\begin{remark}\label{R:SpecPoincCases}
\begin{enumerate}[(a)]
\item If $\vy_V^{k'}(\vx,\vzet)\notin V$, then $\vy_V^k(\vx,\vzet)=\vy_V^{k'}(\vx,\vzet)$, for all $k\ge k'$. Since $U$ is an absorption set for $V$, we may take $N_k(\vx,\cdot)\equiv 1$, for $\vx\notin U\cup V$.
\item If it is assumed that $\ov{U}\bs\dom\subseteq\Gamma$, then $Y'\cup E\subseteq\dom$, so~\eqref{E:FinalMainPoincIneq} implies
\[
    \lint{\dom'}|u(\vx)|^p\dd\vx
  \le M
    \lint{\dom}\lint{Z(\vx)}
      |u(\vx+\vz)-u(\vx)|^p\mu(\vx,\vz)\dd\vz\dd\vx.
\]
\item If $\bnd=\ov{U}$, then $E=\emptyset$ and assumption~(\ref{A:gammaPoinc}) is an empty assumption. Moreover, since only the measurability of $u$ is needed for Lemma~\ref{L:OpenAbsorptionParamSet}, in this case, it is not necessary to assume $u\in L^1$ a priori.
\item If $\dom\subseteq U$, then $V=\emptyset$ and only assumption~(\ref{A:gammaPoinc}) needs to be satisfied.
\item As explained in Remark~\ref{R:MainPoinc1}, if $|\Psi_\alpha(\vx)|$ is uniformly positive throughout $E$ and $\gamma_\alpha\in L^1(E)$, then the requirement in~\eqref{E:uAssumMainPoinc} is implied by $u\in L^1(\ren)$. Moreover, the requirement that $u=0$ a.e. on $\bnd$ is only relevant if $|\bnd|>0$.
\item In the proof, we see that it the sequence of open sets $\{U'_\ell\}_{\ell=1}^\infty$ need only satisfy $|\bigcap_{\ell=1}^\infty U'_\ell\bs\bnd|=0$. Thus, it is only necessary for $\bnd\subseteq\ov{U}$ to be measurable and for $\gamma_\alpha\in L^1(\re\bs U'_\ell)$, for each $\ell\in\nats$ and $\alpha\in A_0$.
\end{enumerate}
\end{remark}

\begin{proof}
Clearly, we may assume that the upper bound in~\eqref{E:FinalMainPoincIneq} is finite. We decompose the domain of integration in the lower bound in~\eqref{E:FinalMainPoincIneq} into the set $V'$ and $\dom'\cap\ov{U}$. The proof relies on using Lemma~\ref{L:OpenAbsorptionParamSet} on the set $V'$ and  Lemma~\ref{L:ControlLemmaBase} on $E$.

By Lemma~\ref{L:Phi'Measurable}, we find $\{\Phi'_\alpha\}_{\alpha=1}^\infty\subset\class{B}(\ren)$. The Lusin property, in assumption~(\ref{A:yzPoinc}), implies the correspondences $\{Z'_\alpha:V\twoheadrightarrow\ren\}_{\alpha=1}^\infty$, defined by $Z'_\alpha(\vx):=\vz(\vx,\Phi'_\alpha)$ are each Lebesgue measurable-valued. Moreover, the injectivity property implies $\{Z'_\alpha(\vx)\}_{\alpha=1}^\infty$ is a partition of $Z(\vx)$, for each $\vx\in V$. Assumptions~(\ref{A:yzPoinc}),~(\ref{A:IndicatrixPoinc}), and~(\ref{A:BoundsPoinc}) allow us to use Lemma~\ref{L:OpenAbsorptionParamSet} to obtain, for each $\alpha\in A$,
\begin{multline*}
  |\Phi'_\alpha|\lint{V'}|u(\vx)|^p\dd\vx
  \le
  \alpha^{p-1}\sum_{k=0}^{\alpha-1}\Lambda_\alpha\Theta_\alpha^k
    \lint{V\cap Y'_{k,\alpha}}\lint{Z'_\alpha(\vx)}
      N_k(\vx,\vz)|u(\vx+\vz)-u(\vx)|^p\dd\vz\dd\vx\\
    +\Lambda_\alpha\alpha^p\Theta_\alpha^\alpha
  \lint{U\cap Y'_{\alpha,\alpha}}\lint{Z'_\alpha(\vx)}
    N_\alpha(\vx,\vz)|u(\vx)|^p\dd\vz\dd\vx.
\end{multline*}
Recall that, by definition of an essentially absorbing set, we find $|\vy^\alpha(V'\times\{\vzet\})\bs U|=0$, for each $\vzet\in\Phi'_\alpha$
For each of the integrals in the summation, we may expand the domain of integration for the outer integral to $Y'$. Since $u=0$ a.e. in $\bnd$, so
\begin{multline}\label{E:IntermPoincIneq1}
  |\Phi'_\alpha|\lint{V'}|u(\vx)|^p\dd\vx
  \le M_V|\Phi|
    \lint{Y'}\lint{Z'_\alpha(\vx)}
      |u(\vx+\vz)-u(\vx)|^p\mu(\vx,\vz)\dd\vz\dd\vx\\
  +\Lambda_\alpha\alpha^p\Theta_\alpha^\alpha\|N_\alpha\|_{L^\infty L^1}
    \lint{E}|u(\vx)|^p\dd\vx.
\end{multline}

We now work to bound the last integral above.\\
\parbox[t]{2.3in}{
Select open sets $\{U'_\ell\}_{\ell=1}^\infty$ such that
\begin{itemize}
\item[$\bullet$] $\bnd\subseteq U'_\ell$, for each $\ell\in\nats$,
\item[$\bullet$] $U'_{\ell+1}\subseteq U'_\ell$, for each $\ell\in\nats$,
\item[$\bullet$] $\bnd=\bigcap_{\ell=1}^\infty U'_\ell$,
\item[$\bullet$] $E_\ell:=E\bs U'_\ell$,
\end{itemize}}
\parbox[t]{2.3in}{
For each $\ell\in\nats$, define
\begin{itemize}
\item[$\bullet$] $T_\ell:=\{\vx\in E:|u(\vx)|\le\ell\}$,
\item[$\bullet$] $u_\ell:=u\chi_{E_\ell}\chi_{T_\ell}\in L^\infty(\ren)$,
\item[$\bullet$] $f_{\alpha,\ell}\in L^\infty(E)$ by
\[
  f_{\alpha,\ell}(\vx):=\flint{\Psi_\alpha(\vx)}u_\ell(\vy)\dd\vy
    -u_\ell(\vx).
\]
\end{itemize}
}\\
We also define the measurable function $f:E\to\reo$ by
\[
  f_\alpha(\vx):=\flint{\Psi_\alpha(\vx)}u(\vy)\dd\vy
    -u(\vx).
\]
By assumption $\gamma_\alpha\in L^1_{\loc}(\ren\bs\bnd)$. Thus $|u_\ell|^p\gamma_\alpha\in L^1(E)$. We may therefore apply Lemma~\ref{L:ControlLemmaBase}. From~\eqref{E:ControlLemmaBaseIneq1},
\begin{equation}\label{E:u_lIneq1Poinc2}
  \lint{E}|u_\ell(\vx)|^p\gamma_\alpha(\vx)\dd\vx
  \le
  C_\alpha\lint{E}|f_{\alpha,\ell}(\vx)|^p
    \gamma_\alpha(\vx)\dd\vy\dd\vx.
\end{equation}
The hypotheses for $\Psi_\alpha$, $|\Psi_\alpha|$, and $\gamma_\alpha$ imply $|f_{\alpha,\ell}|^p\gamma_\alpha$ is nonnegative and measurable. We will use the Lebesgue dominated convergence theorem to argue that the inequality above is preserved in the limit as $\ell\to\infty$. Clearly $\{u_\ell\}_{\ell=1}^\infty$  converges a.e. (in fact everywhere) to $u$ on $E$. By hypothesis, given $\vx\in E$, we find $u\in L^1(\Psi_\alpha(\vx))$ and $|u_\ell(\vy)|\le|u(\vy)|$, for all $\vy\in\Psi_\alpha(\vx)$. The dominated convergence theorem implies
\[
  \lim_{\ell\to\infty}\lint{\Psi_\alpha(\vx)}u_\ell(\vy)\dd\vy=
  \lint{\Psi_\alpha(\vx)}u(\vy)\dd\vy.
\]
Thus $\lim_{\ell\to\infty}f_{\alpha,\ell}(\vx)=f_\alpha(\vx)$ for a.e. $\vx\in E$. Hence $\{|f_{\alpha,\ell}|^p\gamma_\alpha\}_{\ell=1}^\infty$ converges to $|f_\alpha|^p\gamma_\alpha$ a.e. in $E$. Next we observe that, for each $\vx\in E$,
\[
  |f_{\alpha,\ell}(\vx)|
  \le
    \flint{\Psi_\alpha(\vx)}|u_\ell(\vy)|\dd\vy
    +|u_\ell(\vx)|
  \le
  |f_\alpha(\vx)|+2\flint{\Psi_\alpha(\vx)}|u(\vy)|\dd\vy.
\]
Hence
\begin{equation}\label{E:f_alphaBound}
  |f_{\alpha,\ell}(\vx)|^p\gamma_\alpha(\vx)
  \le
  2^{p-1}|f_\alpha(\vx)|^p\gamma_\alpha(\vx)
    +2^p\lp\flint{\Psi_\alpha(\vx)}|u(\vy)|\dd\vy\rp^p\gamma_\alpha(\vx).
\end{equation}
The rightmost term above with a function in $L^1(E)$. For the first term on the right in~\eqref{E:f_alphaBound}, we recall that we are working under the assumption that the upper bound in~\eqref{E:FinalMainPoincIneq} is finite. Suppose that $\alpha\in A$. Jensen's inequality and the last part of assumption~(\ref{A:gammaPoinc}) imply
\begin{align*}
  \lint{E}|f_\alpha(\vx)|^p\gamma_\alpha(\vx)\dd\vx
  &\le
  \lint{E}\flint{\Psi_\alpha(\vx)}|u(\vy)-u(\vx)|^p
    \gamma_\alpha(\vx)\dd\vy\dd\vx\\
  &\le
  \frac{M_{E,\alpha}|\Phi|}
  {\alpha^p\Theta_\alpha^\alpha\|N_\alpha\|_{L^\infty L^1}
    \Lambda_\alpha C_\alpha}
  \lint{E}\flint{\Psi_\alpha(\vx)}|u(\vy)-u(\vx)|^p
    \mu(\vx,\vy-\vx)\dd\vy\dd\vx\\
  &\le
  \frac{M_{E,\alpha}|\Phi|}
  {\alpha^p\Theta_\alpha^\alpha\|N_\alpha\|_{L^\infty L^1}
    \Lambda_\alpha C_\alpha}
  \lint{Y'\cup E}\flint{\Psi_\alpha(\vx)}|u(\vy)-u(\vx)|^p
    \mu(\vx,\vy-\vx)\dd\vy\dd\vx
  <\infty.
\end{align*}
Thus, we may apply the Lebesgue dominated convergence theorem to the upper bound in~\eqref{E:u_lIneq1Poinc2}. Using the monotone convergence theorem for the lower bound, we conclude that
\begin{align}
\nonumber
  \lint{E}|u(\vx)|^p\dd\vx
  &\le
  \lint{E}|u(\vx)|^p\gamma_\alpha(\vx)\dd\vx
  \le
  C_\alpha\lint{E}|f_\alpha(\vx)|^p\gamma_\alpha(\vx)\dd\vx\\
\label{E:uOnUIneq}
  &\le
  C_\alpha\lint{E}
    \flint{\Psi_\alpha(\vx)}|u(\vy)-u(\vx)|^p\gamma_\alpha(\vx)\dd\vy\dd\vx.
\end{align}
The last item in assumption~(\ref{A:gammaPoinc}) implies, for all $\alpha\in A_0$,
\[
    \Lambda_\alpha\alpha^p\Theta_\alpha^\alpha\|N_\alpha\|_{L^\infty L^1}
  \lint{E}|u(\vx)|^p\dd\vx
  \le
  M_{E,\alpha}\lint{E}\lint{\ren}
    |u(\vy)-u(\vx)|^p\mu(\vx,\vy-\vx)\chi_{\Psi_\alpha(\vx)}(\vy)\dd\vy\dd\vx.
\]

Returning to~\eqref{E:IntermPoincIneq1}, we obtain, for each $\alpha\in A$,
\begin{multline*}
  |\Phi'_\alpha|\lint{V'}|u(\vx)|^p\dd\vx
  \le M_V|\Phi|\lint{Y'}\lint{Z'_\alpha(\vx)}
      |u(\vx+\vz)-u(\vx)|^p\mu(\vx,\vz)\dd\vz\dd\vx\\
  +M_{E,\alpha}|\Phi|\lint{E}\lint{\ren}
    |u(\vy)-u(\vx)|^p\mu(\vx,\vy-\vx)\chi_{\Psi_\alpha(\vx)}(\vy)\dd\vy\dd\vx.
\end{multline*}
Recall that $\{Z'_\alpha(\vx)\}_{\alpha=1}^\infty$ is a partition of $Z(\vx)$, for each $\vx\in V$. Summing over $\alpha\in A$ and using the monotone convergence leads to
\begin{multline}\label{E:V'PoincIneq}
  \lint{V'}|u(\vx)|^p\dd\vx
  \le M_V
    \lint{Y'}\lint{Z(\vx)}
      |u(\vx+\vz)-u(\vx)|^p\mu(\vx,\vz)\dd\vz\dd\vx\\
  +\lint{E}\lint{\ren}
    |u(\vy)-u(\vx)|^p\mu(\vx,\vy-\vx)
    \sum_{\alpha\in A}M_{E,\alpha}\chi_{\Psi_\alpha(\vx)}(\vy)\dd\vy\dd\vx.
\end{multline}
Since $u=0$ a.e. in $\bnd$, the same argument used above also provides the inequality
\[
  \lint{\dom'\cap\ov{U}}|u(\vx)|^p\dd\vx
  \le
  \lint{E}|u(\vx)|^p\dd\vx\\
  \le
  M_{E,0}\lint{E}\lint{\ren}
    |u(\vy)-u(\vx)|^p\mu(\vx,\vy-\vx)\chi_{\Psi_0(\vx)}(\vy)\dd\vy\dd\vx.
\]
Adding this to~\eqref{E:V'PoincIneq} yields
\begin{align*}
  \lint{\dom'}|u(\vx)|^p\dd\vx
  &\le
  \lint{V'}|u(\vx)|^p\dd\vx+\lint{E}|u(\vx)|^p\dd\vx\\
  &\le M_V
    \lint{Y'}\lint{Z(\vx)}
      |u(\vx+\vz)-u(\vx)|^p\mu(\vx,\vz)\dd\vz\dd\vx\\
  &\qqquad+\lint{E}\lint{\ren}
    |u(\vy)-u(\vx)|^p\mu(\vx,\vy-\vx)
    \sum_{\alpha\in A_0}M_{E,\alpha}\chi_{\Psi_\alpha(\vx)}(\vy)\dd\vy\dd\vx.
\end{align*}
The proof is complete upon using the definition of $M_U$ and $M$ and changing variables in the inner integral of the last term.
\end{proof}

We now present two corollaries. We continue to use the notation from Theorem~\ref{T:MainPoinc2}. For both corollaries, we assume that $\bnd=\ov{U}$, so it is not necessary to verify assumption~(\ref{A:gammaPoinc}).

For Corollary~\ref{C:PoincCor1}, there are two main simplifying assumptions made. The first is that there is less (volume) compression of the support of the kernel along orbits with large absorption indices. Larger absorption indices are typically associated with parameters $\vzet$ such that $\|\vz(\cdot,\vzet)\|_{\ren}$ is small. Thus $\|\vy(\vx,\vzet)-\vx\|_{\ren}$ will be small, and it is reasonable to expect the distortion in the kernel's support between these two points to also be small. This is essentially the content of the first part of assumption~\ref{A:detBoundsCor1}. The second simplifying assumption, in Corollary~\ref{C:PoincCor1}, is that for each absorbtion index, there is a uniform bound for the Banach indicatrix.
\begin{corollary}\label{C:PoincCor1}
Let $U\subseteq\ren\bs\ov{\dom}$ be an open absorption set for $\dom$, and set $\bnd:=\ov{U}$. Assume the following.
\begin{enumerate}[(i)]
\item\label{A:AbsPoincAssumCor1} The hypotheses in~(\ref{A:yzPoinc}) of Theorem~~\ref{T:MainPoinc2} are satisfied.
\item\label{A:detBoundsCor1} There exists $0\le\lambda,\theta<\infty$ and $\{\Lambda_\alpha\}_{\alpha=1}^\infty\subset(0,\infty)$ such that for each $\alpha\in\nats$ and a.e. $(\vx,\vzet)\in\dom\times\Phi'_\alpha$, we have
\[
  |\det(\partial_{\vx}\vy(\vx,\vzet))|
  \ge
  \lp\frac{\alpha}{\alpha+\lambda}\rp^\theta
  \nd
  |\det\partial_{\vzet}\vz(\vx,\vzet)|\Lambda_\alpha\ge 1.
\]
\item\label{A:muBoundsCor1} There exists $M_0<\infty$ such that for each $\alpha\in\nats$ and for a.e. $\vx\in\dom$,
\begin{equation}\label{A:muBoundCor1}
  \mu(\vx,\vz')M_0
  \ge
  \alpha^p N_\alpha\frac{\Lambda_\alpha}{|\Phi|},
  \quad\text{ for a.e. }\vz'\in\vz(\vx,\Phi'_\alpha),
\end{equation}
with
\[
  N_\alpha
  :=
  \esssup_{(\vx,\vzet)\in\dom\times\Phi'_\alpha}
    N_{\vy^{\alpha-1}(\cdot,\vzet)}(\vx,\dom)<\infty
\]
\end{enumerate}
If $u:\ren\to\re$ is measurable and $u=0$ a.e. in $\bnd$, then
\[
  \lint{\dom}|u(\vx)|^p\dd\vx\le M_0\ee^{\theta\lambda}
    \lint{\dom}\lint{Z(\vx)}
    |u(\vx+\vz)-u(\vx)|^p\mu(\vx,\vz)\dd\vz\dd\vx.
\]
\end{corollary}
\begin{proof}
As indicated in Remark~\ref{R:SpecPoincCases}, to invoke Theorem~\ref{T:MainPoinc2}, it is unnecessary to verify assumption~(\ref{A:gammaPoinc}). The proof simply requires checking assumption~(\ref{A:IndicatrixPoinc}) of Poincar\'{e} Inequality II and identifying $\{\Theta_\alpha\}_{\alpha=1}^\infty$ and $M_V<\infty$ that ensures hypothesis~(\ref{A:BoundsPoinc}) of Theorem~\ref{T:MainPoinc2} is satisfied. We observe that $\{N_k\}_{k=0}^\infty$ is an increasing sequence of functions, so for each $\alpha\in\nats$ and a.e. $(\vx,\vzet)\in\dom\times\Phi'_\alpha$,
\[
  N_k(\vx,\vz(\vx,\vzet))\le N_{\alpha-1}(\vx,\vz(\vx,\vzet))\le N_\alpha.
  \frl k=0,\dots,\alpha-1.
\]
For each $\alpha\in\nats$, put
\[
  \Theta_\alpha:=\lp\frac{\alpha}{\alpha+\lambda}\rp^{-\theta}
  =\lp1+\frac{\lambda}{\alpha}\rp^\theta.
\]
Thus, for each $k=0,\dots,\alpha-1$,
\[
  \Theta_\alpha^k
  \le
  \Theta_\alpha^\alpha=\lp1+\frac{\lambda}{\alpha}\rp^{\alpha\theta}
  \le\ee^{\theta\lambda}.
\]
Hence, by assumption~(\ref{A:muBoundsCor1}), for each $\alpha\in\nats$ and $k=0,\dots,\alpha-1$, we have
\[
  \mu(\vx,\vz(\vx,\vzet'))M_0\ee^{\theta\lambda}
  \ge
  \alpha^pN_k(\vx,\vz(\vx,\vzet'))\Theta_\alpha^k
  \frac{\Lambda_\alpha}{|\Phi|},
\]
for a.e. $\vx\in\dom\cap Y'_{k,\alpha}$ and a.e. $\vzet\in\Phi'_\alpha$. The result follows from Theorem~\ref{T:MainPoinc2}.
\end{proof}

For the second corollary, we assume that there exists $\alpha_0\in\nats$ such that $\Phi'_\alpha=\emptyset$, for all $\alpha\ge\alpha_0$. Thus, the lengths of the paths to the absorption set are uniformly bounded by $\alpha_0$. In this case, it is reasonable to expect uniform bounds for the other parameters in assumptions~(\ref{A:IndicatrixPoinc}) and~(\ref{A:BoundsPoinc}) of Theorem~\ref{T:MainPoinc2}.
\begin{corollary}\label{C:PoincCor2}
Let $U\subseteq\ren\bs\ov{\dom}$ be an open absorption set for $\dom$, and set $\bnd:=\ov{U}$. Assume the following.
\begin{enumerate}[(i)]
  \item The hypotheses in~(\ref{A:yzPoinc}) of Theorem~~\ref{T:MainPoinc2} are satisfied.
  \item There exists $\alpha_0\in\nats$ such that the absorption index for a.e. $\vx\in\dom$ is at most $\alpha_0$.
  \item There exists $\Lambda_0,M_0,N_0,\Theta_0<\infty$ such that
\begin{itemize}
  \item[$\bullet$] $\esssup_{(\vx,\vzet)\in\dom\times\Phi}
    N_{\vy^{\alpha_0-1}(\cdot,\vzet)}(\vx,\dom)\le N_0$;
  \item[$\bullet$] for a.e. $(\vx,\vzet)\in\dom\times\Phi$
\[
  |\det(\partial_{\vx}\vy(\vx,\vzet))|\Theta_0\ge 1
  \nd
  |\det\partial_{\vzet}\vz(\vx,\vzet)|\Lambda_0\ge 1;
\]
  \item[$\bullet$] for a.e. $\vx\in\dom$
\[
  \mu(\vx,\vz')M_0
  \ge
  \frac{\Lambda_0}{|\Phi|},\quad\text{ for a.e. }\vz'\in Z(\vx).
\]
\end{itemize}
\end{enumerate}
If $u:\ren\to\re$ is measurable and $u=0$ a.e. in $\bnd$, then
\[
  \lint{\dom}|u(\vx)|^p\dd\vx
  \le
  \alpha_0^pM_0N_0\Theta_0^{\alpha_0}
  \lint{\dom}\lint{Z(\vx)}
    |u(\vx+\vz)-u(\vx)|^p\mu(\vx,\vz)\dd\vz\dd\vx.
\]
\end{corollary}
\begin{proof}
The proof is a straightforward application of Theorem~\ref{T:MainPoinc2}.
\end{proof}
\begin{remark}\label{R:CorRem}
Clearly the Poincar\'{e} inequalities in Corollaries~\ref{C:PoincCor1} and~\ref{C:PoincCor2} may be refined so that $U$ absorbs $\dom'\subset\dom$.
\end{remark}

\section{General Examples}\label{S:GenExamples}

We use $\tI^n\in\re^{n\times n}$ to denote the identity matrix.

The first example is a generalization of Example~\ref{Ex:BasicEx2} given earlier.
\begin{example}\label{Ex:GFlow}
Let $\te{F}\in\cnt(\ov{\dom};\re^{n\times n})$ be a Lipschitz function, with Lipschitz constant $K<\infty$. Define $\vz\in\cnt(\dom\times\Phi;\ren)$ by $\vz(\vx,\vzet):=\te{F}(\vx)\vzet$, so
\[
  \vy(\vx,\vzet)=\lc\begin{array}{ll}
    \vx+\te{F}(\vx)\vzet, & \vx\in\dom,\\
    \vx, & \vx\notin\dom.
  \end{array}\right.
\]
Let $\{\bbD_\vzet\}_{\vzet\in\Phi}$ denote the parameterized dynamical system generated by $\{\vy(\cdot,\vzet)\}_{\vzet\in\Phi}$. Assume the following.
\begin{enumerate}[(i)]
\item There exists open set $U\subseteq\ren\bs\dom$ that absorbs $\dom$ a.e. in $\{\bbD_\vzet\}$.
\item\label{A:GInjectivity} There exists an $R<\infty$ such that $KR\le1$ and $\Phi\subseteq\bll_R(\ve{0})$, and $\ve{0}\notin\Phi$.
\item\label{A:FlowAbsIndexBounds} There exists an $R\le S<\infty$ such that, for each $\alpha\in\nats$ and $\vzet\in\Phi$,
\[
  \vzet\in\bigcup_{k=\alpha}^\infty\Phi'_\alpha
  \Longrightarrow\vzet\in\Phi\cap\ov{\bll_{S/\alpha}}(\ve{0}).
\]
\item\label{A:FlowDetBounds} There exists $0\le\theta,\tau,\Lambda_0<\infty$ such that for all $(\vx,\vzet)\in\dom\times\Phi$,
\[
  \left|\det\lp\tI^n+\partial_{\vx}[\te{F}(\vx)\vzet]\rp\right|
  \ge
  \lp\frac{1}{1+\tau\|\vzet\|_{\ren}}\rp^\theta
  \nd
  |\det\te{F}(\vx)|\Lambda_0\ge 1.
\]
\item\label{A:FlowmuBounds} There exists $C<\infty$ such that
\[
  \mu(\vx,\vz)C
  \ge
  \frac{1}{\|\vz\|^p_{\ren}}\frac{\Lambda_0|Z(\vx)|}{|\Phi|}
  \chi_{\dom\times\Phi}(\vx,\vz).
\]
\end{enumerate}
If $u:\ren\to\re$ is measurable and $u=0$ a.e. in $\bnd:=\ov{U}$, then
\[
  \lint{\dom}|u(\vx)|^p\dd\vx
  \le CS^p\ee^{\theta S\tau}\lint{\dom}
  \flint{Z(\vx)}|u(\vx+\vz)-u(\vx)|^p\mu(\vx,\vz)\dd\vz\dd\vx,
\]
\end{example}
\begin{remark}\label{R:FlowRem}
\begin{enumerate}[(a)]
\item We present a condition implying the inequalities in~(\ref{A:FlowDetBounds}) of Example~\ref{Ex:GFlow}. Given $\te{B}\in\re^{n\times n}$, define $\|\te{B}\|_{\max}:=\max_{i,j\in\{1,\dots,n\}}|B_{(i,j)}|$. The following is established in~\cite{BreOsb:15a}: for any $\te{B}\in\re^{n\times n}$,
\[
  n\|\te{B}\|_{\max}\le1
  \Longrightarrow
  \det(\tI^n+\te{B})\ge1-n\|\te{B}\|_{\max}.
\]
Suppose that $\te{F}$ satisfies the seconnd inequality in assumption~(\ref{A:FlowDetBounds}) and that $nKR<1$, where $K$ is the Lipschitz constant for $\te{F}$ and $R$ is identified in assumption~(\ref{A:GInjectivity}). Then
\[
  n\|\partial_{\vx}[\te{F}(\vx)\vzet]\|_{\max}
  \le
  n\|\partial_{\vx}\te{F}(\vx)\|_{\max}\|\vzet\|_{\ren}
  \le
  nKR<1.
\]
Put $\tau:=nK$. It follows that
\begin{align*}
  \det(\tI^n+\partial_{\vx}[\te{F}(\vx)\vzet])
  &\ge
  \frac{1-nK\|\vzet\|_{\ren}}{1+\tau\|\vzet\|_{\ren}}
    (1+\tau\|\vzet\|_{\ren})\\
  &=
  \frac{1-nK\|\vzet\|_{\ren}}{1+\tau\|\vzet\|_{\ren}}
    \lp1+\frac{nK\|\vzet\|_{\ren}}{1-nK\|\vzet\|_{\ren}}\rp
  >\frac{1}{1+\tau\|\vzet\|_{\ren}}.
\end{align*}
\item\label{R:Hadamard} We note that by Hadamard's inequality, given any $\vx_0\in\dom$
\begin{multline*}
  \sup_{\vx\in\dom}|\det\te{F}(\vx)|
  \le
  n^\frac{n}{2}\sup_{\vx\in\dom}\|\te{F}(\vx)\|^n_{\re^{n\times n}}
  \le
  n^\frac{n}{2}\ls K\diam(\dom)+\|\te{F}(\vx_0)\|^n_{\re^{n\times n}}\rs\\
  \Longrightarrow
  |Z(\vx)|
  \le
  n^\frac{n}{2}\ls K\diam(\dom)+\|\te{F}(\vx_0)\|^n_{\re^{n\times n}}\rs|\Phi|,
  \frl\vx\in\dom.
\end{multline*}
\item\label{R:pFlowRem} If $1\le p<n$, the lower bound in assumption~(\ref{A:FlowmuBounds}) is integrable, for each $\vx\in\dom$, if $p<n$. For $p\ge n$, the lower bound remains integrable, at each $\vx\in\dom$, if there exists $r>0$ such that $\Phi\cap\bll_r(\ve{0})=\emptyset$.
\item We also note that this example can be further generalized by considering countably Lipschitz maps $\te{F}$.
\end{enumerate}
\end{remark}
\begin{proof}
We will verify the hypotheses of Corollary~\ref{C:PoincCor1}. The Lipschitz requirement is clearly satisfied. By assumption~(\ref{A:FlowDetBounds}), the map $\vzet\mapsto\vz(\vx,\vzet)$ is injective on $\Phi$ for each $\vx\in\dom$. Given $\vzet\in\Phi$, assumption~(\ref{A:GInjectivity}) implies, for each $\vx_1,\vx_2\in\dom$,
\begin{align*}
  \|\vy(\vx_1,\vzet)-\vy(\vx_2,\vzet)\|_{\ren}
  &\ge
  \|\vx_1-\vx_2\|_{\ren}
    -\|\te{F}(\vx_1)-\te{F}(\vx_2)\|_{\re^{n\times n}}\|\vzet\|_{\ren}\\
  &\ge
  (1-K\|\vzet\|_{\ren})\|\vx_1-\vx_2\|_{\ren}\\
  &>
  (1-KR)\|\vx_1-\vx_2\|_{\ren}\ge0,
\end{align*}
whenever $\vx_1\neq\vx_2$. Hence $\vx\mapsto\vy(\vx,\vzet)$ is injective on $\dom$, for each $\vzet\in\Phi$. The Lusin requirements follow from the Lipschitz continuity properties and assumption~(\ref{A:FlowDetBounds}). Thus hypothesis~(\ref{A:AbsPoincAssumCor1}) of Corollary~\ref{C:PoincCor1} is satisfied. By assumption~(\ref{A:FlowAbsIndexBounds}) and~(\ref{A:FlowDetBounds}), given $\alpha\in\nats$,
\[
  \vzet\in\Phi'_\alpha
  \Longrightarrow
  \|\vzet\|\le\frac{S}{\alpha}
  \Longrightarrow
  |\det(\tI^n+\partial_{\vx}[\te{F}(\vx)\vzet])|
  \ge\lp\frac{\alpha}{\alpha+S\tau}\rp^\theta.
\]
Hence, assumption~(\ref{A:detBoundsCor1}) of Corollary~\ref{C:PoincCor1} is satisfied with $\lambda=S\tau$. Since $\vy$ is injective, we find $N_\alpha'\in\{0,1\}$, as defined in assumption~(\ref{A:muBoundsCor1}) of Corollary~\ref{C:PoincCor1}. For each $\alpha\in\nats$, we saw that $\alpha<S\|\vzet\|^{-1}$. Thus assumption~(\ref{A:FlowmuBounds}) and Corollary~\ref{C:PoincCor1} yields
\[
  \lint{\dom}|u(\vx)|^p\dd\vx
  \le CS^p\ee^{\theta S\tau}\lint{\dom}\frac{1}{|Z(\vx)|}
  \lint{Z(\vx)}|u(\vx+\vz)-u(\vx)|^p\mu(\vx,\vz)\dd\vz\dd\vx,
\]
and the result follows.
\end{proof}

In the following example, the domain is split into two regions. The two regions are ``connected'' by an interface strip, which in turn is ``connected'' to the absorption set $U$.
\begin{example}\label{Ex:Discontinuous}
Consider $\dom=(0,2)^2$, and set $\dom_i:=(i-1,i)\times(0,2)$, $i\in I=\{1,2\}$. Given $0<\delta_1<1$, put $U:=U_1\cup U_2$ and $\bnd:=\ov{U}$, with
\[
  U_1:=(1-\delta_1,1+\delta_1)\times(-1,0),
  \quad
  U_2:=(1-\delta_1,1+\delta_1)\times(2,3).
\]
Let $0<\theta\le\frac{\pi}{4}$ and $0<\delta_2<\delta_1\sin\theta$ be given, and set
\[
  Z_1:=
  \{\vz=(z_1,z_2)\in\set{C}_\theta(\ve{0};\ve{e}_1):z_1<\delta_1
    \text{ and }|z_2|>\delta_2\}
  \nd
  Z_2:=-Z_1.
\]
Define $Z,\Psi:\dom\twoheadrightarrow\ren$ by
\[
  Z(\vx):=\lc\begin{array}{ll}
    Z_1, & \vx\in\dom_1,\\
    Z_2, & \vx\in\dom_2
  \end{array}\right.
  \nd
  \Psi(\vx):=(\dom\cup U)\cap(\vx+Z(\vx)).
\]
Suppose there exists a $C<\infty$ such that
\[
  \mu(\vx,\vy-\vx)C
  \ge
  \frac{1}{|y_2-x_2|^{p+1}|Z_1|}\chi_{\dom}(\vx)\chi_{\Psi(\vx)}(\vy),
  \quad\text{for a.e. }(\vx,\vy)\in\ren\times\ren.
\]
If $u(\vx)=0$ for a.e. $\vx\in U$, then
\[
  \lint{\dom}|u(\vx)|^p\dd\vx
  \le
  2^{p+4}
  C\lint{\dom}\lint{\Psi(\vx)}|u(\vy)-u(\vx)|^p\mu(\vx,\vy-\vx)\dd\vy\dd\vx.
\]
\end{example}
\begin{remark}
\begin{enumerate}[(a)]
\item We note that $|Z_1|=\frac{(\delta_1\sin\theta-\delta_2)^2}{\sin\theta}$.
\item The example can be rescaled and repeated periodically to produce a Poincar\'{e} inequality for a laminate like structure.
\end{enumerate}
\end{remark}
\begin{proof}
We produce the Poincar\'{e} inequality by applying Corollary~\ref{C:PoincCor1} on members of an open partition of $\dom$ into four subsets. We provide the details for one of the members, the others being handled in a similar manner. Set $\dom':=(0,1)\times(0,1)$. (The other members of the partition is $(1,2)\times(0,1)$, $(0,1)\times(1,2)$, and $(1,2)\times(1,2)$.) We define
\[
  \Phi:=\{\vzet=(\zeta_1,\zeta_2)\in Z_1:\zeta_2>\delta_2\},
\]
the correspondences $S_1,S_2,S:\Phi\twoheadrightarrow\ren$ by
\[
  S_1(\vzet):=(1-\zeta_1,1)\times(0,3),\quad
  S_2(\vzet):=(1,1+\zeta_1)\times(0,3),\nd
  S(\vzet):=S_1(\vzet)\cup S_2(\vzet)
\]
and the piecewise Lipschitz map $\vz:\ren\times\Phi\to\ren$ by
\[
  \vz(\vx,\vzet):=\lc\begin{array}{ll}
    (\zeta_1,\zeta_2), & \vx\in\dom_1,\\
    (-\zeta_1,\zeta_2), & \vx\in\dom_2,\\
    \ve{0}, & \vx\notin\dom.
  \end{array}\right.
\]
We see that $S(\vzet)\bs\ov{\dom}=U_2$, for all $\vzet\in\Phi$. The map $\vy:\ren\times\Phi\to\ren$ satisfies
\[
  \vy(\vx,\vzet):=\vx+\vz(\vx,\vzet)
  \in
  \lc\begin{array}{ll}
    \dom\bs S(\vzet), & 0<x_1\le1-2\zeta_1
      \text{ and }0<x_2<2-\zeta_2,\\
    \dom\cap S(\vzet), & 1-2\zeta_1<x_1<1+\zeta_1
      \text{ and }0<x_2<2-\zeta_2,\\
    U_2, & \vx\in S(\vzet)\cap\dom\text{ and }2-\zeta_2<x_2<2.\\
  \end{array}\right.
\]
Moreover, since $\theta\le\frac{\pi}{4}$, we find $0<\zeta_2<\zeta_1$ and thus $x_2\le 2-\zeta_1<2-\zeta_2$, for each $\vzet\in\Phi$, $\vx\in\vy^k(\dom'\times\{\vzet\})\bs S(\vzet)$, and $k\in\whls$. We conclude that $\vy^k(\dom'\times\{\vzet\})\subset\dom_1\cup S(\vzet)$, and if $\vx\in\dom\cap\vy^k(\dom'\times\{\vzet\})$, then $\vx+\vz(\vx,\vzet)\in\Psi(\vx)$. Given $\vx=(x_1,x_2)\in\dom$ and $\vzet\in\Phi$, we find that
\[
  \vy^{-1}_{\vzet}(\{\vx\})=\lc\begin{array}{ll}
    \{(x_1-\zeta_1,x_2-\zeta_2),(x_1+\zeta_1,x_2-\zeta_2)\},
      & \vx\in S_1(\vzet)\text{ and }x_2>\zeta_2,\\
    \{(x_1-\zeta_1,x_2-\zeta_2)\},
      & \vx\in S_2(\vzet)\text{ and }x_2>\zeta_2,\\
    \{(x_1-\zeta_1,x_2-\zeta_2)\},
      & \vx\in\dom_1\bs S_1(\vzet),\\
    \emptyset,
      & 0<x_1\le\zeta_1\text{ or }0<x_2\le\zeta_2.
  \end{array}\right.
\]
Hence, if $\vx\in\dom_1\bs S_1(\vzet)$, then $N_{\vy^k_{\vzet}}(\vx,\dom')\in\{0,1\}$, for all $k\in\whls$. We deduce that
\begin{align*}
  \vx\notin S_1
  & \Longrightarrow N_{\vy^{k+1}_\vzet}(\vx,\dom')\le N_{\vy^k_{\vzet}}(\vx,\dom')\\ \vx\in S_1
  & \Longrightarrow
  N_{\vy^{k+1}_{\vzet}}(\vx,\dom')=1+N_{\vy^k_{\vzet}}(\vx,\dom').
\end{align*}
It follows that $N_{\vy^k_{\vzet}}(\vx,\dom')\le k+1$. Clearly, given $\vzet\in\Phi$, the absorption index $\alpha$ for $\dom'$ satisfies $\alpha\le 2/\zeta_2$. This also implies
\[
  N'_\alpha:=\esssup_{(\vx,\vzet)\in\dom\times\Phi'_\alpha}
    N_{\vy^{\alpha-1}_{\vzet}}(\vx,\dom')\le\alpha\le\frac{2}{\zeta_2}.
\]
Set $Y':=\dom\cap\bigcup_{k=0}^\infty\vy^k(\dom'\times\Phi)$. We may invoke Corollary~\ref{C:PoincCor1} and Remark~\ref{R:CorRem}, with $|\Phi|=|Z_1|/2$, $\lambda,\theta=0$ and $\Lambda_\alpha=1$, for each $\alpha\in\nats$, to conclude that
\begin{align*}
  \lint{\dom'}|u(\vx)|^p\dd\vx
  &\le
  2^{p+2}
  C\lint{Y'}\lint{\Phi}|u(\vx+\vz)-u(\vx)|^p\mu(\vx,\vz)\dd\vz\dd\vx\\
  &\le
  2^{p+2}
  C\lint{\dom}\lint{\Psi(\vx)}|u(\vy)-u(\vx)|^p\mu(\vx,\vy-\vx)\dd\vy\dd\vx.
\end{align*}
Adapting the above argument to the remaining partition members yields the result.
\end{proof}

The following is an extension of Example~\ref{Ex:BasicEx1}. Here we assume that the support of the kernel is a measurable set that can vary throughout $\dom$. As in Example~\ref{Ex:BasicEx1}, we use $\eta$ to denote the nonzero coordinate for a hemisphere of $\bll_1(\ve{0})$.
\begin{example}\label{Ex:BasicEx1Ext}
Let $0<\delta\le 1$. Suppose that $\Psi:\dom\twoheadrightarrow\ren$ satisfies the following assumptions:
\begin{itemize}
\item[$\bullet$] $\Psi$ and $\Psi^{-1}$ are measurable valued and $[\vx\mapsto|\Psi(\vx)|]\in L^+(\dom)$;
\item[$\bullet$] For each $\vx\in\dom$,
\[
  \Psi(\vx)\subseteq\bll_\delta(\vx)\cap\set{C}_{\frac{\pi}{2}}(\vx;\ve{e}_1);
\]
\item[$\bullet$] There exists a $1-\frac{\eta\delta}{1+\diam(\dom)}<\tau\le 1$ such that $|\Psi(\vx)|\ge\frac{\tau}{2}|\bll_\delta|$, for all $\vx\in\dom$.
\end{itemize}
Put
\[
  \nu_\tau:=\frac{1}{\tau}\ls1-\frac{\eta\delta}{1+\diam(\dom)}\rs,
\]
and set $\bnd:=\ov{\bigcup_{\vx\in\dom}\Psi(\vx)\bs\dom}$. If $u\in L^1(\ren)$ satisfies $u=0$ a.e. on $\bnd$, then
\[
  \lint{\dom}|u(\vx)|^p\dd\vx
  \le
  C_\tau\lp1+\diam(\dom)\rp\lint{\dom}
    \left|\flint{\Psi(\vx)}[u(\vy)-u(\vx)]\dd\vy\right|^p\dd\vx.
\]
Here $C_\tau=C(\nu_\tau,p)$ is the constant defined in~\eqref{D:ControlLemmaBaseConst} of Lemma~\ref{L:ControlLemmaBase}.
\end{example}
\begin{proof}
Without loss of generality, we assume that $\inf_{\vx\in\dom}\vx\bdot\ve{e}_1=0$. We define $\rho\in L^\infty(\dom\times\ren)$ by $\rho(\vx,\vz):=|\Psi(\vx)|^{-1}\chi_{\Psi(\vx)}(\vz)$. As defined in Theorem~\ref{T:MainPoinc1}, we compute $R(\vx)=|\Psi(\vx)|^{-1}\le\frac{\tau}{2|\bll_\delta|}$. For each $\vx\in\ren$, set $\bll_\delta^-(\vx):=\bll_\delta(\vx)\cap\set{C}_\frac{\pi}{2}(\vx;-\ve{e}_1)$. Although we cannot explicitly identify $\Psi^{-1}$, we do have
\[
  \Psi^{-1}(\vy)\subseteq\dom\cap\bll^-_\delta(\vy)
  \frl\vy\in\dom.
\]
The hypothesis on $\tau$ implies $\nu_\tau<1$. As in Example~\ref{Ex:BasicEx1}, we define $\gamma\in\cnt^\infty(\ren)$ by $\gamma(\vx)=1+\vx\bdot\ve{e}_1$. Given $\vy\in\dom$, we find $\gamma(\vy)\ge 0$, for all $\vx\in\bll_\delta^-(\vy)$, since $\delta\le 1$.  We deduce that, for each or all $\vy\in\dom$,
\[
  \lint{\Psi^{-1}(\vy)}\frac{\gamma(\vx)}{|\Psi(\vx)|}\dd\vx
  \le
  \frac{1}{\tau}\flint{\bll_\delta^-(\vy)}\gamma(\vx)\dd\vx
  =
  \frac{1}{\tau}\lp1+\vy\bdot\ve{e}_1-\gamma_0\eta\delta\rp
  \le\nu_\tau\gamma(\vy).s
\]
By assumption $|\Psi|$ is uniformly positive in $\dom$. It follows that~(BC) is satisfied (see Remark~\ref{R:SpecPoincCases}). Poincar\'{e} Inequality I yields the result.
\end{proof}

For Examples~\ref{Ex:SignChange1} and~\ref{Ex:SignChange2}, we assume $1<p<\infty$ and put $q:=\frac{p}{p-1}$. The examples are analogues of Example~\ref{Ex:BasicEx1} with a sign-changing kernel that has symmetric support $\ov{\bll}_\delta$. In Example~\ref{Ex:SignChange1} the support may be relatively small, when compared to $\diam(\dom)$; in Example~\ref{Ex:SignChange2} it can be relatively large.
\begin{example}\label{Ex:SignChange1}
Let $\alpha<\frac{n}{q}$ and $0<\delta\le\frac{1}{2}\diam(\dom)$ be given. Put
\begin{equation}\label{E:R_0Example}
  R_0
  :=
  \lp\frac{n-\alpha}{n}\rp\lp\frac{n-\alpha}{n-q\alpha}\rp^{p-1}.
\end{equation}
Suppose that $\diam(\dom)\ge\frac{2}{\sqrt{3}}$. Let $\sigma:\dom\times\sph^n_1\to\{-1,1\}$ be a measurable function such that, for some $0<\tau\le 1$,
\[
  \lint{\sph^n_1}\sigma(\vx,\vom)\dd\haus^{n-1}(\vom)
  =\tau\haus^{n-1}(\sph^n_1)=\tau n|\bll^n_1|,
  \frl\vx\in\dom
\]
Define $\rho_\delta:\dom\times\ren\to\re$ by
\begin{equation}\label{E:rhoExample}
  \rho_\delta(\vx,\vz):=\lc\begin{array}{ll}
  \ds{\frac{\delta^\alpha}{\tau|\bll_\delta|}\lp\frac{n-\alpha}{n}\rp
    \|\vz\|^{-\alpha}\sigma(\vx,\smfrac{\vz}{|\vz|})},
  & \vz\in\bll_\delta\bs\{\ve{0}\},\\
  0, & \vz\notin\bll_\delta\bs\{\ve{0}\}.
  \end{array}\right.
\end{equation}
Suppose that
\begin{equation}\label{E:Example1Constraint}
  \nu_\delta:=\frac{R_0}{\tau}
  \ls1-\frac{n}{n+2}\lp\frac{\delta}{\diam(\dom)}\rp^2\rs<1.
\end{equation}
If $u\in L^1(\ren)$ satisfies $u=0$ for a.e. on $\Gamma_\delta:=\ann_{[0,\delta]}\bs\dom$, then
\begin{equation}\label{E:Example1Poinc}
  \lint{\dom}|u(\vx)|^p\dd\vx
  \le
  C_\delta\diam(\dom)^2\lint{\dom}
    \left|\lint{\bll_\delta}[u(\vx+\vz)-u(\vx)]
    \rho_\delta(\vx,\vz)\dd\vz\right|^p\dd\vx.
\end{equation}
Here $C_\delta=C(\nu_\delta,p)$ is the constant defined in~\eqref{D:ControlLemmaBaseConst} of Lemma~\ref{L:ControlLemmaBase}.
\end{example}
\begin{remark}
\begin{enumerate}[(a)]
\item The requirement in~\eqref{E:Example1Constraint} restricts that range of admissible values for $\alpha$ and $\tau$. Since $\delta\le\frac{1}{2}\diam(\dom)$, we must have $\frac{R_0}{\tau}<\frac{4(n+2)}{3n+8}$. For $\delta>0$ small, we must have both $\alpha>0$ and $1-\tau>0$ small as well.
\item The restrictions on $\diam(\dom)$ and $\delta>0$ can be relaxed by considering, for example, the function $\vx\mapsto\gamma_0-\|\vx-\vx_0\|^2$ in place of the $\gamma$ defined in the proof.
\item If $\alpha=0$ and $\sigma\equiv 1$, then $R_0=\tau=1$. Using Remark~\ref{R:OptControlConst}, if $p=2$, then we find that $C_\delta\le4\lp\frac{n+2}{n}\rp^2\lp\frac{\diam(\dom)}{\delta}\rp^4$, so~\eqref{E:Example1Poinc} implies
\[
  \lint{\dom}|u(\vx)|^2\dd\vx
  \le
  4\diam(\dom)^6\lp\frac{n+2}{n}\rp^2\lint{\dom}
    \left|\flint{\bll_\delta}\frac{u(\vx+\vz)-u(\vx)}{\delta^2}\dd\vz
    \right|^2\dd\vx.
\]
\end{enumerate}
\end{remark}
\begin{proof}
Define $\Psi:\dom\twoheadrightarrow\ren$ by $\Psi(\vx):=\bll_\delta(\vx)$, so $\Psi^{-1}(\vx)=\bll_\delta(\vx)\cap\dom$. We see that, for all $\vx\in\dom$
\[
  R(\vx)=\lp\lint{\bll_\delta}|\rho_\delta(\vx,\vz)|^q\dd\vz\rp^{p-1}
  =
  \lp\frac{\delta^\alpha}{\tau|\bll_\delta|}\rp^p\lp\frac{n-\alpha}{n}\rp^{p-1}
    \lp\lint{\bll_\delta}\|\vz\|^{-q\alpha}\dd\vz\rp^{p-1}
  =
  \frac{1}{\tau|\bll_\delta|}R_0,
\]
Let $\vx_0\in\ren$ be the unique point such that $\dom\subseteq\bll_{\frac{1}{2}\diam(\dom)}(\vx_0)$. Define $\gamma\in\cnt^\infty(\ren)$ by
\[
  \gamma(\vx):=\diam(\dom)^2-\|\vx-\vx_0\|^2,
\]
so
\[
  1\le\frac{3}{4}\diam(\dom)^2\le\gamma(\vx)
  \le\diam(\dom)^2,
  \frl\vx\in\dom
\]
We will verify
\begin{equation}\label{E:gamma_deltaVer}
  \frac{R_0}{\tau|\bll_\delta|}\lint{\Psi^{-1}(\vx)}\gamma(\vx')\dd\vx'
  =
  \frac{R_0}{\tau|\bll_\delta|}
    \lint{\bll_\delta(\vx)\cap\dom}\gamma(\vx')\dd\vx'
  \le
  \nu_\delta\gamma(\vx),
  \frl\vx\in\dom.
\end{equation}
Given $\vz\in\ren$, we find $\gamma(\vx+\vz)=\gamma(\vx)-2\lp\vx-\vx_0\rp\bdot\vz-\|\vz\|^2$. Since $0<\delta\le\frac{1}{2}\diam(\dom)$ and $\|\vx-\vx_0\|<\frac{1}{2}\diam(\dom)+\delta$, we see that $\gamma(\vx)\ge 0$, for all $\vx\in\dom_\delta$. Thus
\begin{align*}
  \frac{R_0}{\tau|\bll_\delta|}
    \lint{\bll_\delta(\vx)\cap\dom}\gamma(\vx')\dd\vx'
  &\le
  \frac{R_0}{\tau}\flint{\bll_\delta}\gamma(\vx+\vz)\dd\vz
  =
  \frac{R_0}{\tau}\gamma(\vx)-\frac{2R_0}{\tau}(\vx-\vx_0)
    \bdot\underbrace{\flint{\bll_\delta}\vz\dd\vz}_{=\ve{0}}
  -\frac{R_0}{\tau}\flint{\bll_\delta}\|\vz\|^2\dd\vz\\
  &=
  \gamma(\vx)
    \frac{R_0}{\tau}\ls1-\frac{n}{n+2}\frac{\delta^2}{\gamma(\vx)}\rs
  \le
  \gamma(\vx)
    \frac{R_0}{\tau}\ls1-\frac{n}{n+2}\lp\frac{\delta}{\diam(\dom)}\rp^2\rs
  =\nu_\delta\gamma(\vx).
\end{align*}
Finally, as in the previous example, we have $|\bnd_\delta|>0$ and thus~(BC) is satisfied. Theorem~\ref{T:MainPoinc1} yields~\eqref{E:Example1Poinc}.
\end{proof}

\begin{example}\label{Ex:SignChange2}
We again assume $\alpha<\frac{n}{q}$. Let $R_0$ and $\rho_\delta$ be defined as in~\eqref{E:R_0Example} and~\eqref{E:rhoExample}, respectively. Suppose that
\[
  \nu_\delta:=R_0\lp\frac{|\dom|}{\tau|\bll_\delta|}\rp<1.
\]
If $u\in L^1(\ren)$ satisfies $u=0$ for a.e. $\Gamma_\delta$, then
\begin{equation}\label{E:Example2Poinc}
  \lint{\dom}|u(\vx)|^p\dd\vx
  \le
  C_\delta\lint{\dom}
  \left|\lint{\bll_\delta}[u(\vx+\vz)-u(\vx)]\rho_\delta(\vx,\vz)\dd\vz\right|^p
    \dd\vx.
\end{equation}
\end{example}
\begin{proof}
We see that $R(\vx)=\frac{R_0}{\tau|\bll_\delta|}$ and $\Psi^{-1}(\vx)=\bll_\delta(\vx)\cap\dom\subseteq\dom$, for each $\vx\in\dom$. We put $\gamma\equiv1$ throughout $\ren$. Then
\[
  \frac{R_0}{\tau|\bll_\delta|}\lint{\Psi^{-1}(\vx)}\gamma\dd\vx'
  \le R_0\lp\frac{|\dom|}{\tau|\bll_\delta|}\rp=\nu_\delta\gamma.
\]
The result follows from Theorem~\ref{T:MainPoinc1}.
\end{proof}
\begin{remark}
We observe that $\lim_{\delta\to\infty}\rho_\delta(\vx,\vz)=0$, for all $(\vx,\vz)\in\dom\times\ren$, and $\lim_{\delta\to\infty}C_\delta=1$. Thus
\[
  \lim_{\delta\to\infty}
  \ls\lint{U\cap\bll_\delta}[u(\vx+\vz)-u(\vx)]\rho_\delta(\vx,\vz)\dd\vz\rs
  =-u(\vx),
  \frl\vx\in\dom
\]
If there is a $\delta>0$ such that the upper bound in~\eqref{E:Example2Poinc} is finite, then $u\in L^p(\dom)$ and by the Lebesgue dominated convergence theorem,
\[
  \lim_{\delta\to\infty}C_\delta\lint{\dom}
  \left|\lint{\bll_\delta}[u(\vx+\vz)-u(\vx)]\rho_\delta(\vx,\vz)\dd\vz\right|^p
    \dd\vx
  =\lint{\dom}|u(\vx)|^p\dd\vx.
\]
\end{remark}

The next two examples provide an application of Poincar\'{e} Inequality II in the setting where Dirichlet-like conditions, in the form of~\eqref{E:uAssumMainPoinc}, are imposed on a manifold with co-dimension of at least one. By a $\cnt^2$-manifold we mean a topological manifold with a $\cnt^2$-structure as defined in~\cite{Lee:13a}.
\begin{example}[Constraints $m$-Dimansional $\cnt^2$-Manifolds]\label{Ex:LowDimCompMan}
Let $\bnd\subset\ov{\dom}$ be a $0\le m\le n-1$ dimensional a compact $\cnt^2$-manifold without boundary, and let $\beta>n-m$ be given. There exists $C_\bnd<\infty$ and an open set $U\subseteq\dom$, satisfying $\bnd\subseteq\ov{U}$, such that
\begin{equation}\label{E:ManifoldDirichletPoinc}
  \lint{U}|u(\vx)|^p\dd\vx
  \le C_\bnd\lint{U\bs\bnd}\flint{\bll_{d_\bnd(\vx)}(\ve{0})}
    \frac{|u(\vx+\vz)-u(\vx)|^p}{d_\bnd(\vx)^\beta}\dd\vz\dd\vx,
\end{equation}
for any $u\in L^1(\ren)$ satisfying
\[
  \lint{\dom\bs\bnd}\frac{1}{d_\bnd(\vx)^\beta}
    \lp\flint{\bll_{d_\bnd(\vx)}}|u(\vy)|\dd\vy\rp^p\dd\vx<\infty.
\]
\end{example}
\begin{remark}\label{R:ManExample}
\begin{enumerate}[(a)]
\item The constant $C_\bnd$ can be identified in terms of the upper bound for $|\Psi_0|^{-1}$ in Remark~\ref{R:SmoothMan}(\ref{R:PsiMeasBound}) and depends on $m$, $n$, $\beta$, $p$, and the curvature of $\bnd$ (see~\eqref{E:C'_x0Def}.
\item The set $\bnd$ may include a portion of $\partial\bnd$ or be completely contained within $\bnd$. The example can be furthermore extended to manifolds with boundary by requiring $\beta>n-m+1$ (see Remark~\ref{R:SmoothMan}(\ref{R:ManifoldWithBound})).
\item We can adapt the argument below to uniformly Lipschitz manifolds. Based on Remark~\ref{R:SmoothMan}(\ref{R:LipManifold}), we may establish the following version of~\eqref{E:ManifoldDirichletPoinc}:
\[
  \lint{U}|u(\vx)|^p\dd\vx
  \le C_\bnd\lint{U\bs\bnd}\flint{\dom\cap\bll_{K_0^2\cdot d_\bnd(\vx)}(\vx)}
    \frac{|u(\vy)-u(\vx)|^p}{d_\bnd(\vx)^\beta}\dd\vy\dd\vx,
\]
  where $U\subseteq\dom$ is an open set and $1\le K_0<\infty$ is the Lipschitz constant for $\bnd$. With the introduction of curvature measures, it may be possible to extend the result to a more general class of compact sets and refine the inequality to the one in~\eqref{E:ManifoldDirichletPoinc}.
\end{enumerate}
\end{remark}
\begin{proof}
We will produce $U$ as the union of a finite cover of open sets $\{U_\ell\}_{\ell=1}^{\ell_0}$, where upon a rescaling, we can use Corollary~\ref{C:LowDimSmoothMan} and Poincar\'{e} Inequality II to establish an appropriate Poincar\'{e} inequality for each $\ell=1,\dots,\ell_0$.  We set $G:=\mathbb{H}^m_{(m+1,\dots,n)}$. As discussed in Remark~\ref{R:SmoothMan}(\ref{R:Rescale}), for each $\vx_0\in\bnd$, there exists $1\le K_{\vx_0}<\infty$, an open neighborhood $U_{\vx_0}$ of $\vx_0$ and a bijective $\vg_{\vx_0}\in\cnt^2((-1,1)^n;U_{\vx_0})$ such that
\begin{itemize}
  \item[$\bullet$] $\vg_{\vx_0}((-1,1)^n\cap G)=U_{\vx_0}\cap\bnd$;
  \item[$\bullet$] $\vg^{-1}_{\vx_0}\in\cnt^2(U_{\vx_0};(-1,1)^n)$;
  \item[$\bullet$] $\vg_{\vx_0}(\vP_G(\vv))=\vP_\bnd(\vg_{\vx_0}(\vv))$, for all $\vv\in(-1,1)^n$;
  \item[$\bullet$] $\|\partial_{\vv}\vg_{\vx_0}\|_{\re^{n\times n}}\le K_{\vx_0}$, $\|\partial_{\vx}\vg_{\vx_0}\|_{\re^{n\times n}}\le K_{\vx_0}$, and $\|\partial^2_{\vv}\vg_{\vx_0}\|_{\re^{n\otimes3}}\le K_{\vx_0}$.
\end{itemize}
Since $\partial\dom\cup\bnd$ is compact and $\dom\bs\bnd$ is open, we may assume that $d_\bnd(\vx)\le d_{\partial\dom}(\vx)$, for all $\vx\in U_{\vx_0}$. We see that $\vg_{\vx_0}$ satisfies all but the last part of requirement~(\ref{A:ContDiffg}) in Corollary~\ref{C:LowDimSmoothMan}. We ``blow-up'' around $\vx_0$ to obtain a $\vg$ that possesses the final property. Put $R_0:=2K_{\vx_0}^2\ge 2$ and define $\vg\in\cnt^2((-R_0,R_0);\ren)$ by
\[
  \vg(\vv):=R_0\ls\vg_{\vx_0}(\smfrac{\vv}{R_0})-\vx_0\rs.
\]
Thus $\vg^{-1}(\vx):=R_0\vg_{\vx_0}^{-1}(\frac{\vx+\vx_0}{R_0})$. A straightforward computation shows that for all $\vv\in(-R_0,R_0)$ and all $\vx\in\vg((-R_0,R_0)^n)$
\[
  \partial_{\vv}\vg(\vv)=\partial_{\vv}\vg_{\vx_0}(\smfrac{\vv}{R_0}),
  \;
  \partial_{\vx}\vg^{-1}(\vx)
    =\partial_{\vx}\vg^{-1}_{\vx_0}(\smfrac{\vx+\vx_0}{R_0}),
  \text{ and }
  \partial^2_{\vv}\vg(\vv)
    =\frac{1}{R_0}\partial_{\vv}^2\vg_{\vx_0}(\smfrac{\vv}{R_0}).
\]
Set $\bnd_0:=\vg((-R_0,R_0)^n\cap G)$. It is clear that $\vg$ commutes with the projection operator, since $\vg_{\vx_0}$ does. Restricting $\vg$ to $(-1,1)^n$ and setting $U_0:=\vg((-1,1)^n)$ yields a map $\vg$ possessing all of the properties listed in Corollary~\ref{C:LowDimSmoothMan}(\ref{A:ContDiffg}), with $K_0:=K_{\vx_0}$. With $b,\theta$ satisfying the appropriate constraints and $\Psi_0:U_0\bs\bnd_0\twoheadrightarrow U_0\bs\bnd_0$ defined as in Corollary~\ref{C:LowDimSmoothMan}, we conclude that
\[
  \lint{\Psi_0^{-1}(\vy)}\frac{1}{d_{\bnd_0}(\vx)^\beta|\Psi_0(\vx)|}\dd\vx
  \le
  \frac{\nu}{d_{\bnd_0}(\vy)^\beta},
  \frl\vy\in U_0.
\]
Define $U'_{\vx_0}:=\vx_0+\frac{1}{R_0}U'_{\vx_0}\subseteq U_{\vx_0}$ and $\Psi_{\vx_0}:U'_{\vx_0}\bs\bnd\twoheadrightarrow U'_{\vx_0}\bs\bnd$ by $\Psi_{\vx_0}(\vx):=\vx_0+\frac{1}{R_0}\Psi_0(R_0(\vx-\vx_0))$. As $\ve{0}\in U_0$, we find $\vx_0\in U'_{\vx_0}$. Since
\[
  d_\bnd(\vx)\le d_{\partial\dom}(\vx),
  \quad
  d_\bnd(\vx)=\frac{1}{R_0}d_{\bnd_0}(R_0(\vx-\vx_0)),
  \nd
  \Psi_0(\vx)\subseteq\bll_{d_{\bnd_0}(\vx)}(\vx),
\]
we see that $\Psi_{\vx_0}\subseteq\bll_{d_\bnd(\vx)}(\vx)\subseteq\dom\bs\bnd$, for all $\vx\in U'_{\vx_0}$. Moreover
\begin{equation}\label{E:Psi_x0Containment}
  F_{\vx_0}
  :=\bigcup_{\vx\in U'_{\vx_0}\bs\bnd}
  \Psi_{\vx_0}(\vx)\subseteq U'_{\vx_0}\subseteq\dom\bs\bnd.
\end{equation}
Put $\gamma_{\vx_0}:=\lp R_0/d_\bnd\rp^\beta\in L^+(\ren\bs\bnd)$, so $\gamma_{\vx_0}\ge 1$ on $\ren\bs\bnd$. The change of variables $\vx\mapsto R_0(\vx-\vx_0)$ yields
\[
  \lint{\Psi_{\vx_0}^{-1}(\vy)}
    \frac{\gamma_{\vx_0}(\vx)}{|\Psi_{\vx_0}(\vx)|}\dd\vx
  =R_0^\beta\nsss\lint{\Psi_0^{-1}(R_0(\vy-\vx_0))}\nsss\nss
    \frac{1}{d_{\bnd_0}(\vx)^\beta|\Psi_0(\vx)|}\dd\vx
  \le
  \nu\gamma_{\vx_0}(\vy),
  \text{ for all }\vy\in U'_{\vx_0}.
\]
Before applying the Poincar\'{e} Inequality II on $U'_{\vx_0}$, we need to verify~\eqref{E:uAssumMainPoinc}. Using Remark~\ref{R:SmoothMan}(\ref{R:PsiMeasBound}), there exists $C'_{\vx_0}=C'_{\vx_0}(m,n,b,\theta,\nu,K_{\vx_0})$ such that
\begin{equation}\label{E:C'_x0Def}
  |\Psi_{\vx_0}(\vx)|^{-1}
  \le
  R_0^n|\Psi_0(\vx)|^{-1}
  \le
  C'_{\vx_0}\cdot\lp\frac{R_0}{b\cdot d_{\bnd_0}(R_0(\vx-\vx_0))}\rp^n
  =C'_{\vx_0}\cdot\lp\frac{R_0^2}{b\cdot d_{\bnd}(\vx)}\rp^n.
\end{equation}
It follows that
\[
  \lint{U'_{\vx_0}\bs\bnd}\lp\flint{\Psi_{\vx_0}(\vx)}|u(\vy)|\dd\vy\rp^p
  \gamma_{\vx_0}(\vx)\dd\vx
  \le
  C'_{\vx_0}\cdot\lp\frac{R_0^{2n+\beta}}{b^n}\rp
  \lint{\dom\bs\bnd}\frac{1}{d_\bnd(\vx)^\beta}
    \lp\flint{\bll_{d_\bnd}(\vx)}|u(\vy)|\dd\vy\rp^p
  \dd\vx<\infty.
\]
At this point, we fix the values of $b$, $\theta$, and $\nu$. Recalling~\eqref{E:Psi_x0Containment},~\eqref{E:C'_x0Def}, and the definitions of $R_0$ and $\gamma_{\vx_0}$, we may apply the Poincar\'{e} Inequality II, with $A_0=\{0\}$, and obtain
\begin{align*}
  \lint{U'_{\vx_0}}|u(\vx)|^p\dd\vx
  &\le
  C(\nu,p)\lint{U'_{\vx_0}\bs\bnd}\flint{\Psi_{\vx_0}(\vx)}|u(\vy)-u(\vx)|^p
    \gamma_{\vx_0}(\vx)\dd\vy\dd\vx\\
  &\le
  \underbrace{2C'_{\vx_0}\lp\frac{K_{\vx_0}^{2(n+\beta)}}{b^n}\rp\cdot C(\nu,p)}
    _{:=C''_{\vx_0}}\,
  \lint{U'_{\vx_0}\bs\bnd}\flint{U\cap\bll_{d_\bnd(\vx)}(\vx)}
    \nss\frac{|u(\vy)-u(\vx)|^p}{d_\bnd(\vx)^\beta}\dd\vy\dd\vx.
\end{align*}
The compactness of $\bnd$ allows us to extract a finite subcover $\{U'_{\vx_\ell}\}_{\ell=1}^{\ell_0}$ of $\{U'_{\vx_0}\}_{\vx_0\in\bnd}$. The result follows with $U:=\bigcup_{\ell=1}^{\ell_0}U'_{\vx_\ell}$ and $C_\bnd:=\ell_0\cdot\max_{\ell=1,\dots,\ell_0}C''_{\vx_\ell}$.
\end{proof}

Our final example extends the previous one to the entire set $\dom$.
\begin{example}\label{Ex:ManExample2}
Suppose that $\dom\subseteq\ren$ is an open convex set. Let $\bnd\subset\ov{\dom}$ be a $0\le m\le n-1$ dimensional a compact $\cnt^2$-manifold without boundary, and let $\beta>n-m$ be given. With $\delta_0>0$, define $\delta\in\cnt(\ren)$ and $\Psi:\dom\bs\bnd\twoheadrightarrow\dom\bs\bnd$ by
\[
  \delta(\vx):=\min\{\delta_0,d_\bnd(\vx)\}
  \nd
  \Psi(\vx):=\dom\cap\bll_{\delta(\vx)}(\vx).
\]
There exists $C_\bnd<\infty$ such that
\begin{equation}\label{E:ManifoldDirichletPoinc2}
  \lint{\dom}|u(\vx)|^p\dd\vx
  \le C_\bnd\lint{\dom\bs\bnd}\flint{\Psi(\vx)}
    \frac{|u(\vy)-u(\vx)|^p}{\delta(\vx)^\beta}\dd\vz\dd\vx,
\end{equation}
for any $u\in L^1(\ren)$ satisfying
\[
  \lint{\dom\bs\bnd}\frac{1}{\delta(\vx)^\beta}
    \lp\flint{\Psi(\vx)}|u(\vy)|\dd\vy\rp^p\dd\vx<\infty.
\]
\end{example}
\begin{remark}
\begin{enumerate}[(a)]
\item The statements in Remark~\ref{R:ManExample} are also applicable in the setting of this example.
\item We only require $\bnd$ to be a $\cnt^2$-manifold. The set $\partial\dom\bs\bnd$ may be less regular.
\item The convexity requirement for $\dom$ may be dropped provided there is a transformation, such a $\cnt^2$-diffeomorphism, of $\dom$ into a convex set.
\end{enumerate}
\end{remark}
\begin{proof}
As in Example~\ref{Ex:LowDimCompMan}, we obtain a $0<\nu<1$, $\{\vx_\ell\}_{\ell=1}^{\ell_0}\subseteq\bnd$, open sets $\{U'_\ell\}_{\ell=1}^{\ell_0}$, and $\{\Psi_{\vx_\ell}:U'_\ell\bs\bnd\twoheadrightarrow U'_\ell\bs\bnd\}_{\ell=1}^{\ell_0}$ such that
\begin{itemize}
\item[$\bullet$] $\vx_\ell\subset U_\ell\subseteq\dom$ and $\bnd\subset\ov{\bigcup_{\ell=1}^{\ell_0}U_\ell}$,
\item[$\bullet$] $\Psi_{\vx_\ell}(\vx)\subseteq\bll_{d_\bnd(\vx)}(\vx)$, for all $\vx\in U'_\ell$,
\item[$\bullet$] $\int_{\Psi^{-1}_{\vx_\ell}(\vy)}
  d_\bnd(\vx)^{-\beta}|\Psi_{\vx_\ell}(\vx)|^{-1}\dd\vx\le\nu d_\bnd(\vx)^{-\beta}$, for all $\vy\in U'_\ell$.
\end{itemize}
Without loss of generality, we may assume that $d_\bnd(\vx)\le 1$, for all $\vx\in U:=\bigcup_{\ell=1}^{\ell_0}U'_{\vx_\ell}$, so $d_\bnd^{-\beta}\ge 1$ on $U$. For each $\ell\in\nats$, let $\dom_\ell\subseteq\dom$ the open and convex Voronoi set in $\dom$ associated with $\vx_\ell$. Thus, for each $\vx\in\dom_\ell$, we find $\|\vx-\vx_\ell\|_{\ren}<\|\vx-\vx_k\|_{\ren}$, for all $k\neq\ell$. The collection $\{\dom_\ell\}_{\ell=1}^{\ell_0}$ constitutes an open partition of $\dom$. We will use the Poincar\'{e} Inequality II on each $\dom_\ell$.

To this end, fix $\ell=1,\dots,\ell_0$. Without loss of generality, we assume that $\vx_\ell=\ve{0}$. Using the notation of Theorem~\ref{T:MainPoinc2}, put $U=U'_\ell\cap\dom_\ell$, so $V=\dom_\ell\bs\ov{U}$, $\bnd\cap\dom_\ell\subset\ov{U}$, and $E=\ov{U}\bs\bnd$. Put
\[
  \vep:=\min\{1,\delta_0,\smfrac{1}{2}\inf_{\vx\in V}d_\bnd(\vx)\}>0.
\]
Thus $\Phi:=\bll_\vep(\ve{0})\subset\bll_{2\vep}(\ve{0})\subseteq U$. Define $\sigma\in\cnt^1((0,\infty))$ and $\vz,\vy\in L^\infty(\ren\times\Phi;\ren)$ by
\[
  \sigma(\tau):=\frac{\vep}{\tau+\vep},\quad
  \vz(\vx,\vzet):=\sigma(\|\vx\|_{\ren})\lp\vzet-\vx\rp\chi_V(\vx),
  \nd
  \vy(\vx,\vzet):=\vx+\vz(\vx,\vzet).
\]
Since $\|\vx\|_{\ren}-\vep<\|\vzet-\vx\|_{\ren}<\|\vx\|_{\ren}+\vep$, we find $\vz\in\bll_\vep(\ve{0})\subseteq\bll_{\delta(\vx)}(\ve{0})$, for all $(\vx,\vzet)\in V\times\Phi$. Moreover, we must have $\sigma(\|\vx\|_{\ren})<\frac{1}{3}$, since $\|\vx\|_{\ren}>2\vep$. This and the convexity of $\dom_\ell$ implies $\vy(V,\vzet)\subseteq U\cup V$, for each $\vzet\in\Phi$. In fact, if $\vx\in V$, then
\begin{align*}
  d_\bnd(\vy(\vx,\vzet)
  &\ge
  d_\bnd(\vx)-\|\vz(\vx,\vzet)\|
  =d_\bnd(\vx)-\frac{\vep\|\vzet-\vx\|_{\ren}}{\|\vzet-\vx\|_{\ren}+\vep}\\
  &>
  d_\bnd(\vx)-\vep\ge\frac{1}{2}d_\bnd(\vx)>\vep.
\end{align*}
It follows that $\dist(\vy^k(V\times\{\vzet\}),\bnd)>\vep$, for all $k\in\whls$ and $\vzet\in\Phi$. The maps $\vy$ and $\vz$ are clearly differentiable throughout $V$. Using $\otimes$ for a tensor product, we compute
\[
  \partial_{\vx}\vy(\vx,\vzet)
  =
  (1-\sigma(\|\vx\|_{\ren}))\tI^n
  +\frac{\sigma'(\|\vx\|_{\ren})}{\|\vx\|_{\ren}}\vx\otimes(\vzet-\vx).
\]
Since the second term is a rank-one matrix, we deduce that
\[
  \det\partial_{\vx}\vy(\vx,\vzet)
  =(1-\sigma(\|\vx\|_{\ren}))^n
    \lp1+\frac{\sigma'(\|\vx\|_{\ren})}{\|\vx\|_{\ren}(1-\sigma(\|\vx\|_{\ren}))}
    \vx\bdot(\vzet-\vx)\rp.
\]
Now $|\sigma'(\|\vx\|_{\ren})|=\frac{\vep}{(\|\vx\|_{\ren}+\vep)^2}$ and $(1-\sigma(\|\vx\|_{\ren}))=\frac{\|\vx\|_{\ren}}{\|\vx\|_{\ren}+\vep}$, so
\[
  \left|\frac{\sigma'(\|\vx\|_{\ren})}{\|\vx\|_{\ren}(1-\sigma(\|\vx\|_{\ren}))}
    \vx\bdot(\vzet-\vx)\right|
  \le
  \frac{\vep}{\|\vx\|_{\ren}}<\frac{1}{2}.
\]
Consequently,
\[
  \det\partial_{\vx}\vy(\vx,\vzet)>\lp\frac{1}{3}\rp^n.
\]
For each $\vx\in V$, we have the bound $\sigma(\|\vx\|_{\ren})<\frac{\vep}{\diam(V)+\vep}$, and hence
\[
  \det\partial_{\vzet}\vz(\vx,\vzet)=\sigma(\|\vx\|_{\ren})^n
  >\lp\frac{\vep}{\diam(V)+\vep}\rp^n.
\]
Next, we argue that $U$ essentially absorbs $V$ in parameterized dynamical system generated by $\vy$ and that $\vx\mapsto\vy(\vx,\vzet)$ is injective on $V$. Let $\vzet\in\Phi$ and $\vx_0\in V$ be given. For each $k\in\nats$, put $x_k:=\vy(\vx_{k-1},\vzet)$, so $\{\vx_k\}_{k=0}^\infty$ is the forward orbit of $\vx_0$. We define $\wh{\vx}_k:=\vx_k-\vzet$. It follows that, for each $k\in\whls$, if $\vx_k\in V$, then $\{\vx_\ell\}_{\ell=0}^k\subset V$ and
\begin{align*}
  &\vx_{k+1}=\wh{\vx}_k+\vzet-\sigma(\|\vx_k\|_{\ren})\wh{\vx}_k
  =(1-\sigma(\|\vx_k\|_{\ren}))\wh{\vx}_k+\vzet\\
  &\Longrightarrow
  \wh{\vx}_{k+1}=(1-\sigma(\|\vx_k\|_{\ren}))\wh{\vx}_k
  =\prod_{\ell=0}^k(1-\sigma(\|\vx_\ell\|_{\ren}))\wh{\vx}_0.
\end{align*}
As noted above, we $\sigma(\|\vx\|_{\ren})<\frac{\vep}{\diam(V)+\vep}$, for all $\vx\in V$. It follows that
\[
  \|\vx_{k+1}-\vzet\|_{\ren}
  =\|\wh{\vx}_{k+1}\|_{\ren}\le\lp\frac{\diam(V)}{\diam(V)+\vep}\rp^k\diam(V).
\]
Select $k_0\in\nats$ such that $\lp\frac{\diam(V)}{\diam(V)+\vep}\rp^{k_0-1}\diam(V)<\vep$. Since $\vzet\in\bll_\vep(\ve{0})$, we deduce that $\vx_{k_0}\in\bll_{2\vep}(\ve{0})\subseteq U$. We conclude that $U$ absorbs $V$, since $k_0$ is independent of $\vx_0\in V$. In fact, the absorption index is at most $k_0$ for all $\vzet\in\Phi$. To show injectivity, suppose that there exists $\vx_1,\vx_2\in V$ such that $\vx_1\neq\vx_2$ and $\vy(\vx_1,\vzet)=\vy(\vx_2,\vzet)$. Using the notation already introduced, we may assume that $\|\wh{\vx}_1\|_{\ren}<\|\wh{\vx}_2\|_{\ren}$. Then $(1-\sigma(\vx_1))<(1-\sigma(\vx_2))$
\[
  (1-\sigma(\|\vx_1\|_{\ren}))\wh{\vx}_1=(1-\sigma(\|\vx_2\|_{\ren}))\wh{\vx}_2
  \Longrightarrow
  \wh{\vx}_1=\frac{1-\sigma(\|\vx_2\|_{\ren})}{1-\sigma(\|\vx_1\|_{\ren})}
    \wh{\vx}_2.
\]
This implies $\|\wh{\vx}_1\|_{\ren}>\|\wh{\vx}_2\|_{\ren}$, which is a contradiction. Therefore the map $\vx\mapsto\vy(\vx,\vzet)$ is injective on $V$.

We are now in position to apply Theorem~\ref{T:MainPoinc2}. The set of absorption indices $A$ satisfies $A\subseteq\{1,\dots,k_0\}$, with $k_0$ identified above. Let $\Psi'_\ell:\dom_\ell\twoheadrightarrow U'_\ell\cup V$ be defined by
\[
  \Psi'_\ell(\vx):=\lc\begin{array}{ll}
    \Psi_{\vx_\ell}(\vx), & \vx\in \ov{U},\\
    \vx+\vz(\{\vx\}\times\Phi), & \vx\in V,
  \end{array}\right.
\]
so $\Psi'_\ell(\vx)\subseteq\dom\cap\bll_\vep(\vx)\subseteq\dom\cap\bll_{\delta(\vx)}(\vx)$ and $|\Psi'_\ell(\vx)|\ge\sigma(\vx)^2|\Phi|>\lp\frac{\vep^2}{\diam(V)}\rp^n|\bll_1|$, for all $\vx\in V$. For each $\alpha\in A\cup\{0\}$, we may select $M_{E,\alpha},M_V<\infty$ sufficiently large so that assumption~(\ref{A:gammaPoinc}) is satisfied. Recalling~\eqref{E:C'_x0Def}, Theorem~\ref{T:MainPoinc2} yields
\[
  \lint{\dom_\ell}|u(\vx)|^p\dd\vx
  \le
  C_\ell\lint{\dom_\ell\bs\bnd}\flint{\Psi'_\ell(\vx)}
    \frac{|u(\vy)-u(\vx)|^p}{\delta(\vx)^\beta}\dd\vz\dd\vx
  \le
  C'_\ell\lint{\dom_\ell\bs\bnd}\flint{\dom\cap\bll_{\delta(\vx)}(\vx)}
    \frac{|u(\vy)-u(\vx)|^p}{\delta(\vx)^\beta}\dd\vz\dd\vx.
\]
Put $C_\bnd:=\max_{\ell=1,\dots,\ell_0}C'_\ell$. Since $\{\dom_\ell\}_{\ell=1}^{\ell_0}$ is an open partition of $\dom$, we may sum the above inequality over $\ell=1,\dots,\ell_0$ to obtain~\eqref{E:ManifoldDirichletPoinc2}.
\end{proof}
\bibliographystyle{plain}
\bibliography{PoincareIneqBib}

\end{document}